\documentclass[reqno]{amsart}

\usepackage[
left=3cm,% left margin
right=3cm,% right margin
top=3cm, % top margin
bottom=3cm,% bottom margin
a4paper]{geometry}

\usepackage[T1]{fontenc}
\usepackage{lmodern}
\usepackage[utf8]{inputenc}
\usepackage{scrextend}
		
\usepackage{mathrsfs}

\usepackage{tikz-cd}
\usepackage{bbm}

\usepackage{mathtools}

\usepackage[UKenglish]{babel}

\usepackage{amssymb}
\usepackage{amsthm}
\usepackage{amsmath}
\usepackage{mathrsfs}
\usepackage{comment}
\usepackage{tikz-cd}
\usepackage{tikz}

\usetikzlibrary{decorations.markings}
\usetikzlibrary{shapes.geometric}

\usepackage{amsmath,amssymb,amsthm,graphicx,amscd,amsfonts,latexsym}
\usepackage{pstricks}

\usepackage[normalem]{ulem}

\usepackage{lscape}
\usepackage{hyperref} % for links in PDF
\usepackage{tikz}
\usetikzlibrary{arrows}

\usetikzlibrary{shapes}
\usetikzlibrary{decorations.pathreplacing,decorations.markings} %arrows on smooth curves in tikz pictures
\usepackage{graphicx}
\usepackage{caption} %for pictures in morphism section
\usepackage{float}
\usepackage{upgreek}
\usepackage{xfrac} % fractions
\usepackage{faktor} 

\usetikzlibrary{shapes.geometric,positioning}
\usetikzlibrary{arrows,decorations.pathmorphing,decorations.pathreplacing}

\usepackage{ifthen} % needed for tikz
\newcounter{pos} % needed for tikz
\tikzset{									% for pictures
	initcounter/.code={\setcounter{pos}{0}},
	style between/.style n args={3}{
		postaction={
			initcounter,
			decorate,
			decoration={
				show path construction,
				curveto code={
					\addtocounter{pos}{1}
					\pgfmathtruncatemacro{\min}{#1 - 1}
					\ifthenelse{\thepos < #2 \AND \thepos > \min}{
						\draw[#3]
						(\tikzinputsegmentfirst)
						..
						controls (\tikzinputsegmentsupporta) and (\tikzinputsegmentsupportb)
						..
						(\tikzinputsegmentlast);
					}{}
				}
			}
		},
	},
}

\renewcommand{\P}{P^\bullet}

\DeclareMathOperator{\Img}{Im}
\DeclareMathOperator{\module}{mod}

\DeclareMathOperator{\Hom}{Hom}
\DeclareMathOperator*{\proj}{proj}
\DeclareMathOperator{\End}{End}
\DeclareMathOperator{\Ext}{Ext}

\DeclareMathOperator{\Aut}{Aut}
\DeclareMathOperator{\Coh}{Coh}
\DeclareMathOperator{\Int}{Int}
\DeclareMathOperator{\Diffeo}{Diff}
\DeclareMathOperator{\Simp}{Simp}
\DeclareMathOperator{\Perf}{Perf}

\DeclareMathOperator{\Inn}{Inn}
\DeclareMathOperator{\GL}{GL}
\DeclareMathOperator{\PGL}{PGL}
\DeclareMathOperator{\Equ}{Equ}
\DeclareMathOperator{\T}{\mathcal{F}}

\DeclareMathOperator{\basis}{\mathfrak{B}}
\DeclareMathOperator{\rad}{rad}
\DeclareMathOperator{\marked}{\mathcal{M}}

\DeclareMathOperator{\MCG}{\mathcal{M}\mathcal{C}\mathcal{G}}

\DeclareMathOperator{\out}{Out}

\newcommand\restr[2]{{% we make the whole thing an ordinary symbol
		\left.\kern-\nulldelimiterspace % automatically resize the bar with \right
		#1 % the function
		\vphantom{\big|} % pretend it's a little taller at normal size
		\right|_{#2} % this is the delimiter
}}

\tikzset{
	set arrow inside/.code={\pgfqkeys{/tikz/arrow inside}{#1}},
	set arrow inside={end/.initial=>, opt/.initial=},
	/pgf/decoration/Mark/.style={
		mark/.expanded=at position #1 with
		{
			\noexpand\arrow[\pgfkeysvalueof{/tikz/arrow inside/opt}]{\pgfkeysvalueof{/tikz/arrow inside/end}}
		}
	},
	arrow inside/.style 2 args={
		set arrow inside={#1},
		postaction={
			decorate,decoration={
				markings,Mark/.list={#2}
			}
		}
	},
}

\makeatletter
\tikzset{commutative diagrams/.cd,arrow style=tikz,diagrams={>=latex'}}\tikzset{join/.code=\tikzset{after node path={%
			\ifx\tikzchainprevious\pgfutil@empty\else(\tikzchainprevious)%
			edge[every join]#1(\tikzchaincurrent)\fi}}}
\makeatother

\tikzset{>=stealth',every on chain/.append style={join},
	every join/.style={->}}

\tikzset{every loop/.style={min distance=25mm,in=50,out=100,looseness=5}}

\newtheorem{prf}{Proof}[section]

\theoremstyle{remark}
\theoremstyle{plain}
\newtheorem{thm}[prf]{Theorem}

\newtheorem{Introthm}{Theorem}
	
\newtheorem{lem}[prf]{Lemma}
\newtheorem{prp}[prf]{Proposition}
\newtheorem{cor}[prf]{Corollary}

\theoremstyle{definition}
\newtheorem{exa}[prf]{Example}
\newtheorem*{convention}{Convention}

	\newtheorem{definition}[prf]{Definition}
	\newtheorem*{ndefinition}{Definition}
	
		\newtheorem{rem}[prf]{Remark}
			
\theoremstyle{remark}
\newtheorem{notation}[prf]{Notation}

\theoremstyle{plain}

\def\centerarc[#1](#2)(#3:#4:#5) %[draw options] (center) (initial angle:final angle:radius)
{ \draw[#1] ($(#2)+({#5*cos(#3)},{#5*sin(#3)})$) arc (#3:#4:#5); }

\def\triangle[#1][#2][#3][#4][#5][#6][#7]%
{\begin{displaymath}\begin{tikzcd}[ampersand replacement=\&] #1 \arrow{r}{#4} \& #2 \arrow{r}{#5} \& #3 \arrow{r}{#6} \& #1[#7]\end{tikzcd}\end{displaymath}
}

\begin{document}%\parindent 0pt
\title{\MakeUppercase{On auto-equivalences and complete derived invariants of gentle algebras}}

\author{Sebastian Opper}
\address[]{Mathematisches Institut, Universität Paderborn, Warburger Str. 100, 33098 Paderborn.}
\email{sopper@math.uni-paderborn.de}

%%%%%%%%%%%%%%%%%%%%%%%%%%%%%%%%%%%%%%%%%%%%%%%%%%%%%%%%%%%%%%%%%%%%%%%%%%%%%%%%%%%%%%%%%%%%%%%%%%%%%%%%%%%%%%%%%%%%%%%%%%%%%%%%%%%%%%%%%%%%%%%
\begin{abstract}
We study triangulated categories which can be modeled by an oriented marked surface $\mathcal{S}$ and a line field $\eta$ on $\mathcal{S}$. This includes bounded derived categories of gentle algebras and -- conjecturally -- all partially wrapped Fukaya categories introduced by Haiden-Katzarkov-Kontsevich \cite{HaidenKatzarkovKontsevich}. We show  that triangle equivalences between such categories induce diffeomorphisms of the associated surfaces preserving orientation, marked points and line fields up to homotopy. This shows that the pair $(\mathcal{S}, \eta)$ is a triangle invariant of such  categories and prove that it is a complete derived invariant for gentle algebras of arbitrary global dimension. We deduce that the group of auto-equivalences of a gentle algebra is an extension of the stabilizer subgroup of $\eta$ in the mapping class group and a group, which we describe explicitely in case of triangular gentle algebras. We show further that diffeomorphisms associated to spherical twists are Dehn twists.
\end{abstract}

\maketitle

\tableofcontents

\section*{Introduction}\label{Introduction}
\noindent The present work evolved from parts of the authors PhD thesis \cite{OpperThesis} and is devoted to the study of auto-equivalences and triangle invariants of a class of triangulated categories which we call \textit{surface-like categories}.
Examples of such categories include bounded derived categories of gentle algebras which have been studied extensively  since their introduction by Assem-Skowro\`nksi \cite{AssemSkowronski} over thirty years ago. More recently, the work of Bocklandt \cite{BocklandtPuncturedSurface} and Haiden-Katzarkov-Kontsevich \cite{HaidenKatzarkovKontsevich} revealed that derived categories of gentle algebras arise as special cases of the class of partially wrapped Fukaya categories of surfaces. The latter are $A_{\infty}$-categories associated with the following data.
\begin{itemize}
\setlength\itemsep{1ex}

\item A compact, oriented, smooth surface $S$ with non-empty boundary $\partial S$,

\item a set $\marked \subseteq \partial S$ of marked points and

\item  the homotopy class of a line field $\eta$, i.e.\ the homotopy class of a section of the projectivized tangent bundle of $S$.
\end{itemize} 
On the other hand, it was shown in \cite{LekiliPolishchukGentle} and \cite{OpperPlamondonSchroll} that every (graded) gentle algebra gives rise to the datum of a marked surface. However, the surface constructions are not identical. \medskip 

\noindent In \cite{OpperPlamondonSchroll}, the authors provided a ``dictionary'' which relates information on the triangulated structure of the derived category with geometric information in the spirit of the partially wrapped Fukaya category in the sense of \cite{Auroux}. As an example, the dictionary includes correspondences between indecomposable objects and curves on the model surface equipped with local systems. Further examples are  geometric interpretations of morphisms, mapping cones and Auslander-Reiten triangles.   These results extend preceeding results in \cite{HaidenKatzarkovKontsevich} about the classification of indecomposable objects in the partially wrapped Fukaya category. It shows that an essential part of the information contained in the derived category of a gentle algebra is encoded in its surface. This relationship serves as prototypical example for the definition of surface-like categories. We expect all partially wrapped Fukaya categories to be examples of surface-like categories.\medskip

\noindent Beyond their presence in Homological Mirror Symmetry in form of partially wrapped Fukaya categories, gentle algebras have a surprising number of connections to areas within and outside of representation theory. This includes

\begin{enumerate}
\setlength\itemsep{1ex}

\item categorical resolutions of tame, non-commutative, nodal curves including (stacky) chains and cycles of projective lines, c.f.\ \cite{BurbanDrozdTilting}, \cite{BurbanDrozd2018} as well as \cite{LekiliPolishchukGentle}, \cite{LekiliPolishchuk2017}.

\item The trivial extension of a gentle algebra is a Brauer graph algebra \cite{SchrollTrivialExtension}. Brauer graph algebras appear in modular representation theory of finite groups as blocks of the group algebra with cyclic defect, c.f.\ \cite{Dade}, \cite{Donovan}. It was shown in \cite{Antipov2007} that Brauer graph algebras have derived invariants related to surfaces.
\end{enumerate}

\noindent The main objective of the present work is the investigation of triangle equivalences between surface-like categories which leads to a structural understanding of the group of auto-equivalences of such categories and provides us with a better understanding on the question when two such categories are triangle equivalent. Our focus is on applications to gentle algebras for which we obtain the most far reaching results.\medskip 

 \noindent Throughout this paper we assume that our ground field is algebraically closed.\medskip

\noindent Our key result is the following relationship between equivalences of surface-like categories and diffeomorphisms of their model surfaces.

\begin{Introthm}\label{IntroTheoremEquivalencesFukayaLikeCategories}
	Let $\T$ and $\T'$ be surface-like categories modeled by marked surfaces $\mathcal{S}, \mathcal{S}'$ and line fields $\eta$ and $\eta'$.  Then, the following is true.\bigskip
	
	\begin{enumerate}
	\item Every triangle equivalence $T :\T \rightarrow \T'$ induces a diffeomorphism $\Psi(T): \mathcal{S} \rightarrow \mathcal{S}'$ of marked surfaces, i.e.\ it preserves orientations and marked points. Moreover, $\Psi(T)$ maps $\eta$ to a line field in the homotopy class of $\eta'$; \bigskip
			\item $\Psi(T)$ realizes the action of $T$ on indecomposable objects, i.e.\ if $\gamma_X$ is a curve on $\mathcal{S}$ that represents an indecomposable object $X \in \T$ and $\gamma_{T(X)}$ is a curve on $\mathcal{S}'$ which represents $T(X)$, then  
			\begin{displaymath} \Psi(T)(\gamma_X) \simeq_{\ast} \gamma_{T(X)}  ,\end{displaymath}
			\smallskip
			
			where $\simeq_{\ast}$ is an equivalence relation which  is slightly coarser than homotopy.
\end{enumerate}	
\end{Introthm}

\noindent The equivalence relation $\simeq_{\ast}$ agrees with the homotopy relation for most surfaces. 

Algebraically, $\simeq_{\ast}$ groups objects which are stable under the Auslander-Reiten translation into families which are closed under ``degeneration of the continuous parameter''. 

The construction of the map $\Psi(-)$ is given in Section \ref{SectiondiffeomorphismInducedByEquivalences}. Theorem \ref{IntroTheoremEquivalencesFukayaLikeCategories} combines the assertions of Theorem \ref{TheoremPsiRealizesAction}, Proposition \ref{PropositionPsiYieldsOrientationPreservingdiffeomorphisms}, Proposition \ref{PropositionPsiPreservesWindingNumbers} as well as the results from Section \ref{SectionSpecialSurfaces}. \medskip

\noindent It was shown in \cite{LekiliPolishchukGentle} that (graded) gentle algebras of finite global dimension with equivalent surface models are derived equivalent. We extend this result to ungraded gentle algebras of infinite global dimension and deduce from Theorem \ref{IntroTheoremEquivalencesFukayaLikeCategories} that the surface model of a gentle algebra is a complete derived invariant in the following sense (Theorem \ref{TheoremEquivalentSurfacesImplyDerivedEquivalence} in the text).

\begin{Introthm}\label{IntroCorollaryDerivedInvariant}Let $A_1$ and $A_2$ be gentle algebras, let $\mathcal{S}_{A_i}$ denote the surface of $A_i$ and let $\eta_{A_i}$ denote its line field. Then, the algebras $A_1$ and  $A_2$ are derived equivalent if and only if there exists an orientation preserving diffeomorphism $H: \mathcal{S}_{A_1} \rightarrow \mathcal{S}_{A_2}$, which restricts to a bijection between the sets of marked points and which maps $\eta_{A_1}$ to a line field in the homotopy class of $\eta_{A_2}$.
\end{Introthm}
\noindent  An important consequence of Theorem \ref{IntroCorollaryDerivedInvariant} is that the diffeomorphism type of $\mathcal{S}_{A}$ and the orbit of the homotopy class of the line field $\eta_{A}$ under the action of the mapping class group of $\mathcal{S}_A$ is a complete derived invariant for gentle algebras extending the previously defined derived invariants by Avella-Alaminos and Geiss \cite{AvellaAlaminosGeiss} (``AAG invariant'').  The mapping class group of $\mathcal{S}_A$ is the set of all isotopy classes of orientation preserving self-diffeomorphisms preserving the set of marked points. By the results of \cite{LekiliPolishchukGentle}, such an orbit is completely determined by a finite set of integers which can be computed explicitely. \medskip

\noindent The next theorem concerns the groups of auto-equivalences of derived categories of gentle algebras. Recall that a finite dimensional algebra $A$ is called \textit{triangular} if its $\Ext$-quiver contains no oriented cycles, i.e.\ there exists no cyclic sequence of simple $A$-modules $S_0, \dots, S_{n+1} \cong S_0$ with the property that $\Ext^1(S_i, S_{i+1}) \neq 0$ for all $0 \leq i \leq n$.

\begin{Introthm}\label{IntroTheoremShortExactSequence} Let $A$ be a gentle algebra. Then, the following is true. \medskip
\begin{enumerate}
\item There exists a short exact sequence of groups			
\begin{displaymath}
\begin{tikzcd}
0 \arrow{r} & \mathscr{K} \arrow{r} & \Aut(\mathcal{D}^b(A)) \arrow{r}{\Psi} &  \MCG(\mathcal{S}_A, \eta_A) \arrow{r} & 0,
\end{tikzcd}
\end{displaymath}
where $\Aut(\mathcal{D}^b(A))$ denotes the group of auto-equivalences of $\mathcal{D}^b(A)$ up to natural isomorphism and $\MCG(\mathcal{S}_A, \eta_A)$ denotes the stabilizer of $\eta_A$ of the action of the mapping class group of $\mathcal{S}_A$. \bigskip

\item Let $X \in \mathcal{D}^b(A)$ be spherical. Then, there exists a simple loop $\gamma_X$ on $\mathcal{S}_A$ such that
\[\Psi(T_X) = D_{\gamma_X},\]

\noindent where $T_X: \mathcal{D}^b(A) \rightarrow \mathcal{D}^b(A)$ denotes the spherical twist with respect to $X$ and $D_{\gamma_X}$ denotes the Dehn twist about $\gamma_X$.
\bigskip

\item \label{Itempoint} Assume that $A$ is triangular. Then, $\mathscr{K}$ is a subgroup of the group of outer automorphisms of $A$. Moreover, $\mathscr{K}$ is generated by the groups of \textit{rescaling equivalences} and \textit{coordinate transformations}.

\end{enumerate}
\end{Introthm}
\noindent The definition of the groups in Theorem \ref{IntroTheoremShortExactSequence} (\ref{Itempoint}) is given in Section \ref{SectionKroneckerAlgebra} and Section \ref{SectionRescalingEquivalences}.
We expect Theorem \ref{IntroTheoremShortExactSequence} (\ref{Itempoint}) to hold true in general. Under the assumption that every auto-equivalence of a gentle algebra is standard, this follows from our results. Theorem \ref{IntroTheoremShortExactSequence} combines the assertions of Theorem \ref{TheoremSphericalTwistsDehnTwists}, Theorem \ref{TheoremKernelOfPsi} and Theorem \ref{TheoremImageOfPsi}.\medskip

\noindent As we were finalizing this paper, it has come to our attention that the derived equivalence classification of gentle algebras (Theorem \ref{IntroCorollaryDerivedInvariant}) was also proved in \cite{AmiotPlamondonSchroll} by classifying tilting and silting objects in $\mathcal{D}^b(A)$. It seems that our proof that gentle algebras with equivalent surface models are derived equivalent (Theorem \ref{TheoremEquivalentSurfacesImplyDerivedEquivalence}) and their proof are similar and build on ideas in \cite{LekiliPolishchukGentle}. It seems further to be the case that  our constructions of a line field on the surface of a gentle algebra are related. On the other hand, our proof that the graded surface of a gentle algebra is a derived invariant relies on Theorem \ref{IntroTheoremEquivalencesFukayaLikeCategories} and differs from the proof in \cite{AmiotPlamondonSchroll}.\bigskip

\noindent \textbf{Structure of the paper.} \ After recalling relevant notions and facts related to marked surfaces, line fields and local systems in Section \ref{SectionPreliminaries}, we introduce surface-like categories in Section \ref{SectionFukayaLikeCategories}, where we discuss their basic properties and provide the theoretic foundation for subsequent sections. This includes the notion of ``degeneration-closed'' families of objects which are stable under the Auslander-Reiten translation.  Of particular importance in this section is the definition of \textit{segment objects} and the dichotomy of \textit{boundary morphisms} and \textit{interior morphisms}, which mimics the geometric notion of a boundary segment as well as the dichotomy of interior and boundary points of a marked surface.\medskip 

\noindent In Section \ref{SectiondiffeomorphismInducedByEquivalences}, we show how the relationship between diffeomorphisms of a marked surface and isomorphisms of its \textit{arc complex} can be exploited to define the map $\Psi$ in Theorem \ref{IntroTheoremEquivalencesFukayaLikeCategories}.\medskip 

\noindent In Section \ref{SectionOndiffeomorphisms},  we then prove that the diffeomorphism $\Psi(T)$ of an equivalence $T$ between surface-like categories satisfies various properties as claimed in Theorem \ref{IntroTheoremEquivalencesFukayaLikeCategories}. However, the construction of $\Psi(T)$ only guarantees control over the action of $\Psi(T)$ on triangulations of the corresponding surface and it is a priori not clear to which extent $\Psi(T)$ is a good geometric realization of $T$.  In order to overcome this problem, we introduce the notion of a triangulation of a surface-like category (c.f.\ Section \ref{SectionTriangulationsFukayaLikeCategories}) and show that objects in a surface-like category are determined (in a suitable way) by their interaction with a triangulation of the category  in the same way as curves are determined (up to homotopy) by their intersections with a triangulation of their ambient surface (see Section \ref{SectionOnCharacteristicSequences}).\medskip

\noindent In Section \ref{SectiondiffeomorphismDehntwists} we prove that the diffeomorphism of a spherical twist with respect to a spherical object $X$ is the Dehn twist about a simple loop associated with $X$.\medskip 

\noindent Section \ref{SectionKernelOfPsi} and Section \ref{SectionImageOfPsi} contain the proof of Theorem \ref{IntroTheoremShortExactSequence}.

\noindent  Finally, Section \ref{SectionHomeorphismsInduceEquivalences} contains a geometric characerization of tilting complexes in bounded derived categories of gentle algebras and the proof of Theorem \ref{IntroCorollaryDerivedInvariant}.\medskip

\subsection*{Acknowledgements}
I like to thank Igor Burban, Wassilij Gnedin and Alexandra Zvonareva for many helpful discussions on the subject. My research was supported by the DFG grants BU 1866/4-1 and  the Collaborative Research Centre on ``Symplectic Structures in Geometry, Algebra and Dynamics'' (CRC/TRR 191). \smallskip

\section*{Conventions} \noindent  We work over a field $k$ and write $k^{\times}$ for its group of units. Unless stated otherwise every module over a finite dimensional $k$-algebra will be assumed to be finite-dimensional. For an algebra $A$, we write $\mathcal{D}^b(A)$ for its bounded derived category of finite dimensional left $A$-modules and $\Perf(A)$ for its subcategory of perfect complexes. 

The set of natural numbers contains $0$ by definition. If $a,b \in \mathbb{Z}$ and $n$ is an integer valued variable, we write $n \in [a,b]$ instead of $a \leq n \leq b$ and similar for all other types of intervals. 
Arrows of a quiver are composed from left to right. The source (resp. target) of an arrow $\alpha$ is denoted by $s(\alpha)$ (resp.\@ $t(\alpha)$).

\medskip

\section{Preliminaries}\label{SectionPreliminaries}
\noindent In this section, we recall the definitions of marked surfaces and related objects and give an overview on line fields and winding numbers.
 \smallskip
\subsection{Marked surfaces}
\begin{definition}\label{DefinitionMarkedSurface}A \textbf{marked surface} is a pair $\mathcal{S}=(S, \marked)$ consisting of a smooth, oriented compact surface $S$ with a non-empty boundary and a subset of marked points $\marked \subseteq S$ such that $B \cap \marked \neq \emptyset$ for all connected components $B \subseteq \partial S$.

\end{definition}

 \noindent For the rest of this section let $\mathcal{S}=(S, \marked)$ denote a marked surface. The elements of $\marked \setminus \partial S$ are called \textbf{punctures}. Note that by definition our marked surface always contain at least one marked point on the boundary.
  A \textbf{curve} on $\mathcal{S}$ is a map $\gamma:I \rightarrow S$, where $I$ is a compact interval or or the unit circle and $\partial I=\gamma^{-1}(\marked)$. A curve $\gamma$ is called \textbf{arc} if $I$ is an interval. Otherwise, we call it a \textbf{loop}. An arc $\gamma: I \rightarrow S$ is said to be \textbf{finite} (resp \textbf{semi-finite}) if both (resp. exactly one) of its end points lie in $\partial S$. An arc, which is not finite is called \textbf{infinite}.
 A \textbf{homotopy} between curves is a homotopy of paths in the usual sense. However, we require such a homotopy to be constant at end points and that the interiors of the intermediate curves of the homotopy to avoid punctures. A \textbf{boundary curve} is a curve, which is homotopic to a curve on the boundary. A \textbf{boundary segment} is an arc which connects two neighbouring marked points on the boundary.
 We often consider lifts of curves to a universal cover $\tilde{\mathcal{S}}$. By the lift of a loop, we mean an the corresponding map defined on the real line.
\begin{convention}
	If not said otherwise we will always assume curves to be not contractible and loops to be \textbf{primitive}, i.e.\@\ not homotopic to a loop which factors through a non-trivial covering map $S^1 \rightarrow S^1$. 
\end{convention}
\subsubsection{Marked surfaces with unmarked boundary components}\label{SectionUnmarkedBoundaryComponents}
Although Definition \ref{DefinitionMarkedSurface} requires every boundary component to contain at least one marked point but allows for punctures, it is worth pointing out that we could as well have chosen to allow for boundary components without marked points (``unmarked components'') and no punctures and define infinite arcs in a different way. Indeed, by gluing a once-punctured disc to every unmarked boundary component we obtain a new marked surface without unmarked component and punctures. In this picture, the homotopy class of an infinite arc ending on a puncture $p$ on the new surface corresponds to a non-compact curve which winds in clockwise-direction around the boundary component in the old surface corresponding to $p$. This justifies the terminology of ``infinite'' arcs.

\subsubsection{Oriented intersections and minimal position}  Given distinct curves $\gamma_1:I_1 \rightarrow S$ and $\gamma_2: I_2 \rightarrow S$ on $S$, we write $\gamma_1 \overrightarrow{\cap} \gamma_2$ for the set of \textbf{oriented intersections}, i.e.\ the set of pairs $(s_1,s_2)$ such that $\gamma_1(s_1)=\gamma_2(s_2)$ and such that $\gamma_i(s_i)$ is either an interior point, or a point on the boundary such that, locally around the intersection, $\gamma_1$   ``lies before'' $\gamma_2$ in the counter-clockwise orientation as shown Figure \ref{FigureDirectedBoundaryIntersection}.
	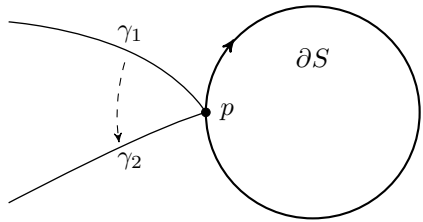
\begin{figure}[H]
		\begin{displaymath}
		\begin{tikzpicture}[scale=0.8]
		\draw[thick, 
		decoration={markings, mark=at position 0.4 with {\arrow{<}}},
		postaction={decorate}
		](0,0) circle (50pt);
		\filldraw (-50pt, 0) circle (2pt);
		\draw (-40pt,0) node{$p$};
		\draw (0,0.9) node {$\partial S$};
		\draw [line width=0.5, color=black] plot  [smooth, tension=1] coordinates {  (-50pt,0) (-3, 1) (-5, 1.5)  };
		\draw [line width=0.5, color=black] plot  [smooth, tension=1] coordinates {  (-50pt,0) (-3, -0.5) (-5, -1.5)  };
		\draw (-3,1.3) node{$\gamma_1$};
		\draw (-3,-0.8) node{$\gamma_2$};
		
		\draw[dashed, ->] ({3.2*cos(165)},{3.2*sin(165)}) arc (164:188:3.2);
		\end{tikzpicture}
		\end{displaymath}
		\caption{An oriented boundary intersection $p$ from $\gamma_1$ to $\gamma_2$.} \label{FigureDirectedBoundaryIntersection}
	\end{figure}  
\noindent	If it causes no ambiguity, we identify an intersection point with its image in $S$. Note that for every oriented interior intersection $p=(t,t') \in \gamma_1 \overrightarrow{\cap} \gamma_2$, we have $\overline{p}\coloneqq (t',t) \in \gamma_2 \overrightarrow{\cap} \gamma_1$ and we call $\overline{p}$ the \textbf{dual} of $p$. We further denote by $\gamma_1 \cap \gamma_2$ the disjoint union of $\gamma_1 \overrightarrow{\cap} \gamma_2$ and $\gamma_2 \overrightarrow{\cap} \gamma_1$. Frequently, we also use the term \textbf{self-intersection} of a curve $\gamma: I \rightarrow S$ for pairs $(s_1, s_2)$ of distinct elements $s_1, s_2 \in I$ such that $\gamma(s_1)=\gamma(s_2)$.  

\begin{ndefinition}
	A set of curves $\{\gamma_1, \dots, \gamma_m\}$ is said to be in \textbf{minimal position} if for all (not necessarily distinct) $i,j \in [1,m]$, the number of (self-)intersections  is minimal within their respective homotopy classes.
\end{ndefinition}
\noindent Without further notice in subsequent parts of this paper, we make use following results about curves:
\begingroup
\setlength{\leftmargini}{20pt}
\begin{enumerate}
 \setlength\itemsep{1em}
	\item As pointed out in \cite{Thurston}, it follows from \cite{FreedmanHassScott} and \cite{Neumann-Coto}, that every  finite set of curves can be homotoped to a set of curves in minimal position and given a set $\{ \gamma_1, \dots, \gamma_m\}$  in minimal position and for any set of curves $\{\gamma_{m+1}, \dots, \gamma_n\}$, there exists a set of curves $\{\gamma_{m+1}', \dots, \gamma_n'\}$ such that $\gamma_i' \simeq \gamma_i$ for all $i \in (m,n]$ and  $\{\gamma_1, \dots, \gamma_m, \gamma_{m+1}', \dots, \gamma_n'\}$ is in minimal position.
	
	\item If $\{\gamma_1, \gamma_2\}$ are in minimal position, then any of their lifts to the universal cover of $S$ intersect at most once in the interior, see  \cite{OpperPlamondonSchroll}, Lemma 3.4.
\end{enumerate}
\endgroup
\

\noindent Given curves $\gamma, \gamma'$, we denote by $\iota(\gamma, \delta)$ the minimal number of intersections between any pair of curves $\gamma'$ and $\delta'$ from the homotopy classes of $\gamma$ and $\delta$. The integer $\iota(\gamma, \delta)$ is called the \textbf{geometric intersection number} of $\gamma $ and $\delta$. A curve with no interior self-intersections is called \textbf{simple}.\ \smallskip

\subsection{Local systems}\label{SectionLocalSystems}
\begin{definition}
The \textbf{fundamental groupoid}  $\pi_1(X)$ of a topological space $X$ is a category with objects given by the set $\{x \, | \, x \in X\}$. The set of morphisms $x \rightarrow y$ between two points $x,y \in X$ is the set of homotopy classes of paths from $x$ to $y$. Morphisms are composed by concatenation of paths. A $k$-linear \textbf{local system} on a curve $\gamma: I \rightarrow S$ is a contravariant functor $\pi_1(I) \rightarrow k-\module$. A morphism between local systems is a natural transformation of functors.
\end{definition} 
\noindent A local system amounts to a choice of a vector space at every point in the parameter space of a curve and compatible choices of isomorphisms between them 

A local system on a loop is equivalent to the choice of an automorphism of a finite dimensional vector space (which depends on the choice of a base point). In particular, isomorphism classes  of indecomposable local systems of dimension $d$ are in one-to-one correspondence with polynomials $(X-\lambda)^d$, $\lambda \in k^{\times}$.\smallskip

\noindent In contrast, there exist a unique indecomposable local system on every arc. It has dimension $1$.
\ \smallskip

\subsection{Line fields and graded surfaces}\label{SectionLineFields}
\subsubsection{Line fields and winding numbers}\label{SectionLineFieldsWindingNumbers}
\begin{definition}
A \textbf{line field} on a smooth marked surface $S$ with punctures $\mathscr{P}$ is a continuous section of the projectivized tangent bundle over $S \setminus \mathscr{P}$, i.e. \@ a map $\eta: S \setminus \mathscr{P} \rightarrow \mathbb{P}(TS)$ such that $\pi \circ \eta=\text{Id}_{S \setminus \mathscr{P}}$, where $\pi: \mathbb{P}(TS) \rightarrow S$ denotes the projection.
\end{definition}
\noindent In other words, a line field $\eta$ is a continuous choice of $1$-dimensional subspaces in the tangent planes of the surface, possibly with singularities at punctures. It determines a trivialization of the projectivized tangent bundle away from punctures.

Every no-where vanishing vector field $\mathscr{X}: S \rightarrow TS$ on $S$ induces a line field by virtue of  the projection of the complement of the zero section in $TS$  onto $\mathbb{P}(TS)$.  However, not every line field arises in this way since the choice of $\mathscr{X}$ amounts to a continuous choice of  orientations on the lines described by $\eta$.

In light of our discussion in Section \ref{SectionUnmarkedBoundaryComponents}, we  extend a line field $\eta$ from an unpunctured surface with unmarked boundary components to the corresponding surface with punctures by virtue of the radial projection $D^2 \setminus \{(0,0)\} \rightarrow S^1$.

\medskip

\noindent Every line field give rise to a function $\omega=\omega_{\eta}$ which associates an integer to every immersed loop $\gamma$ on $S$ called the \textbf{winding number} of $\gamma$.  The idea is to integrate the ``differences''  between the lines $\eta(p)$ and the lines $\mathbb{P}(TS)_p$ determined by the derivative of $\gamma$ (which is non-zero by assumption). 

Regarding the submanifold $\eta(S) \subseteq \mathbb{P}(TS)$ as a class $[\eta(S)] \in H_1(\mathbb{P}(TS), \mathbb{Z})$,  the intersection pairing $(-,-)$ gives rise to a map 
\[\begin{tikzcd} \omega_{\eta}: H_1(\mathbb{P}(TS), \mathbb{Z}) \arrow{r} & \mathbb{Z} \\ u \arrow[mapsto]{r} & ([\eta(S)], u).\end{tikzcd}\]

\noindent If $\gamma \subseteq S$ is an immersed loop, we set $\omega_{\eta}(\gamma)\coloneqq \omega_{\eta}([\dot\gamma])$, where $\dot\gamma$ denotes the derivative of $\gamma$ regarded as a loop in $\mathbb{P}(TS)$.  In other words, $\omega_{\eta}(\gamma)$ equals the number of points (counted with sign) at which $\dot \gamma$ and $\eta \circ \gamma$ agree.

 By definition, $\omega_{\eta}(\gamma)$ only depends on the regular homotopy class of $\gamma$. A \textbf{regular homotopy} of immersed loops is a homotopy such that every intermediate loop is immersed. 
 
Moreover, $\omega_{\eta}$ is determined by its values on all \textbf{unobstructed loops}, i.e.\ loops which contain no nullhomotopic loops. Indeed, $\omega_{\eta}(f)=1$ for the homology class $f \in H_1(\mathbb{P}(TS), \mathbb{Z})$ of any fiber of $\mathbb{P}(TS)$ with the induced orientation and the homology class of the derivative of a loop is the sum of the class of an unobstructed loop and the (inverse) homology classes of fibers of $\mathbb{P}(TS)$. By Theorem 5.5 in \cite{Chillingworth} (see also Lemma 2.4.,\cite{Chillingworth}), two unobstructed loops are homotopic if and only if they are regular homotopic. An unobstructed loop with vanishing winding number is called a \textbf{$\eta$-gradable loop}, or just gradable loop if the choice of $\eta$ is apparent from the context.   \medskip

\subsubsection{Graded surfaces and graded curves}
\ \smallskip

\noindent
\begin{definition}
A \textbf{graded surface} is a pair $(\mathcal{S}, \eta)$ consisting of a marked surface $\mathcal{S}$ and a line field $\eta$ on $\mathcal{S}$. A \textbf{grading} on a  smooth, immersed path $\gamma: I \rightarrow S$ is
a homotopy class of a path $\mathfrak{g}$ from $\eta \circ \gamma$ to $\dot\gamma$. A grading on a piecewise smooth path $\gamma$ is a collection of gradings on its smooth segments.
\end{definition}
\noindent Often, we do not distinguish between a grading and a respresentative of its homotopy class.\smallskip

\noindent Gradings on $\gamma$ are in bijection with the integers $\mathbb{Z}$ and there exists a canonical action of $\mathbb{Z}$ on grading.  To make this more explicit, let $(\gamma, \mathfrak{g})$ be a smooth, graded curve. Then, for any point $p$ of  $\gamma$ the corresponding integer is associated with $u \in \pi_1(\mathbb{P}(TS)_{\gamma(p)}, \eta(p)) \cong \mathbb{Z}$, where $u$ is the concatenation of $\mathfrak{g}|_{\{p\} \times [0,1]}$ followed by a counter-clockwise, non-surjective path from $\dot \gamma(p)$ to $\eta(p)$. 

This does not depend on the choice of $p$. There is a canonical action of $\mathbb{Z}$ on gradings of $\gamma$ which correspond to the $n \mapsto n+1$ on the integers. The action of $n \in \mathbb{Z}$ of $\mathfrak{g}$ is denoted by $\mathfrak{g}[n]$.\medskip

\noindent Let $(\gamma, \mathfrak{g})$ and $(\delta, \mathfrak{d})$ be graded, smooth paths which are transversal and let $p \in \gamma \overrightarrow{\cap} \delta$. The \textbf{degree} $\deg(p) \in \mathbb{Z}$ of $p$ is the integer associated with $u \in \pi_1(\mathbb{P}(TS)_{\gamma(p)}, \eta(p))$, where $u$ is the concatenation of $\mathfrak{d}|_{\{p\} \times [0,1]}$ followed by first, a clockwise, non-surjective path from $\dot \delta(p)$ to $\dot \gamma(p)$ $\dot \gamma(p)$ and then followed by inverse path of $\mathfrak{g}|_{\{p\} \times [0,1]}$. Note that this differs slightly from the definition in \cite{HaidenKatzarkovKontsevich}. It follows that $\deg(\overline{p})=1-\deg(p)$. \medskip

\noindent We want define winding  numbers for piecewise smooth loops. Suppose that $\gamma: I \rightarrow \mathcal{S}$ is a graded, immersed loop which is smooth everywhere except at pairwise distinct points $z_1, \dots, z_m$ on the boundary which are totally ordered according to the orientation of $\gamma$. For $i \in [1,m]$, let $\gamma_i\coloneqq \gamma|_{[z_i, z_{i+1}]}$ be a smooth segment (indices modulo $m$). We assume that for all $i \in [1,m]$, $\gamma_i$ and $\gamma_{i+1}$ intersect transversally. We write $\gamma_{\text{sm}}$ for a smoothing of $\gamma$ which agrees with $\gamma$ everywhere outside a union of small neighborhoods around the points $p_i$.

\begin{lem}\label{LemmaWindingNumberPiecewiseSmoothCurves} Let $\gamma$ be a piecewise smooth, graded loop satisfying the conditions above. Then, the 
\[
    \omega(\gamma_{\text{sm}})=\sum_{i=1}^m{\deg(p_i)} - \sum_{i=1}^m{\sigma_i},
\]

\noindent where 
\[\sigma_i= {\begin{cases}1, & \textrm{if $\gamma_{i-1}$ meets $\gamma_{i}$ at $p_i$ from the right hand side};\\  0, & \text{otherwise.}  
		\end{cases}}\]
\end{lem}
\begin{proof}
The assertion is a local statement and the proof reduces to the following oberservation. Let $(\delta_1,\mathfrak{g}_i)$ ($i=1,2$) be immersed, graded, simple paths which are transverse such that the end point of $\delta_1$ is the start point of $\delta_2$ and such that the corresponding intersection $p\in \delta_1\overrightarrow{\cap} \delta_2$ has degree $0$. Let $\delta$ be the smoothing of $p$ which agrees with $\delta_1$ and $\delta_2$ outside a small neighborhood $U$ of $p$ which does not contain any other end point of $\delta_1$ or $\delta_2$. Then, for each $i \in [1,2]$, there exists a unique grading $\mathfrak{f}_i$ on $\delta$ which agrees with $\mathfrak{g}_i$ on the portion of $\delta_i$ which lies outside of $U$. Then, $\mathfrak{f}_2[1]=\mathfrak{f}_1$.
\end{proof}

We refer to the expression $\omega(\gamma_{\text{sm}})+ \sum_{i=1}^m{\sigma_i}$ from Lemma \ref{LemmaWindingNumberPiecewiseSmoothCurves} as the \textbf{winding number of the piecewise smooth curve $\gamma$}.

\medskip

\section{Surface-like categories}\label{SectionFukayaLikeCategories}

\noindent In this section we introduce the notion of a surface-like category which is our main object of study in this paper.
Their definition is heavily inspired by the geometric nature of derived categories of gentle algebras (cf. \cite{OpperPlamondonSchroll} and \cite{LekiliPolishchukGentle}) and partially wrapped Fukaya categories as defined in \cite{Auroux}.\medskip

\noindent Let $\T$ be a $k$-linear triangulated category and let $[1]$ denote its shift functor.  We assume that $\T$ is Krull-Schmidt and that its class of isomorphism classes of indecomposable objects is a set.

\begin{definition}\label{DefinitionFukayaLikeQuintuples}Let $(\mathcal{S}, \eta)$ be a graded marked surface.   
A quintuple $(\T, \mathcal{S}, \eta, \gamma, \beta)$ is called \textbf{surface-like} and $(\mathcal{S}, \eta)$ is called a \textbf{surface model} of $\T$, if all of the following six relations are satisfied.\smallskip

	\begin{addmargin}[1.5em]{0em}
\noindent \textbf{1) Indecomposable objects \& curves}: \ \ The map $\gamma$ is a bijection  from \medskip
\begin{itemize}
\setlength\itemsep{0.5em}
\item the isomorphism classes of indecomposable objects, to
\item pairs $([\gamma], [\mathcal{V}])$ consisting of \ \medskip
\begin{itemize} \setlength\itemsep{0.5em}
\item the homotopy class of an unoriented graded curve $\gamma$, and
\item the isomorphism class of an indecomposable $k$-linear local system $\mathcal{V}$ on $\gamma$. 
\end{itemize} 
\end{itemize}
Moreover, we require that $\gamma$ is compatible with shifts.
\medskip

	\noindent 	In what follows let $X_1, X_2, X_3 \in T$ be indecomposable objects corresponding to arcs or loops equipped with a $1$-dimensional local system. Let  $(\gamma_{i}, \mathfrak{g}_i, \mathcal{V}_i)$ ($1 \leq i \leq 3$) be a representative of $\gamma(X_i)$. We assume that $\{\gamma_1, \gamma_2, \gamma_3\}$ is in minimal position.  \ \smallskip
		
\noindent	\textbf{2) Intersections \& morphisms}: \ \ The following is true: \ \smallskip
	
		\begin{itemize}
		\setlength\itemsep{1em}
		 \item[I)] There exists an injection $\basis$ of $\gamma_1 \overrightarrow{\cap} \gamma_2$ into a basis of $\Hom^*(X_1, X_2)$ consisting of morphisms. Moreover, $\basis$ is compatible with the gradings.
		 \item[II)] For every intersection $q \in \gamma_1 \overrightarrow{\cap} \gamma_2$ at a puncture, there exists a family of morphisms $\left(\basis(q)(j)\right)_{j \in \mathbb{N}}$ and $m \in \mathbb{Z}$, such that 
		 \[\basis(q)(j) \in \Hom(X_1, X_2[m+j\cdot w_q])\]
		and  $\basis(q)(0)=\basis(q)$, where $w_q$ is the winding number of the simple loop, which winds around $q$ once in clockwise direction.

			\item[III)] If $\gamma_1$ and $\gamma_2$ are not homotopic loops, then a basis of $\Hom^*(X_1, X_2)$ is given by the set consisting of \ \smallskip
			\begin{itemize}
			\setlength\itemsep{1ex}
				\item[a)] all  morphisms $\basis(p)$ for all $p \in\gamma_1 \overrightarrow{\cap} \gamma_2$, which are not punctures, and
				\item[b)] all morphisms $\basis(q)(j)$, where $q \in \gamma_1 \overrightarrow{\cap} \gamma_2$ is a puncture and $j \in \mathbb{N}$. 
			\end{itemize}  

			\item[IV)] If $\gamma_1$ and $\gamma_2$ are homotopic loops, then $\basis$ is not surjective and the quotient of $\Hom^*(X_1, X_2)$ by the image of $\basis$ is spanned by the residue class of an isomorphism and the residue class of a connecting morphism $h$ in an Auslander-Reiten triangle
			\[\begin{tikzcd} X_2[n] \arrow{r} & Y \arrow{r} & X_1 \arrow{r}{h} & X_2[n+1]. \end{tikzcd}\]

			\item[V)]	If $Y \in T$ is represented by a loop with a local system of dimension at least $2$, then \[\dim \Hom^*(Y,Y)~\geq~3.\]
			
		\end{itemize}
		
	\noindent	\textbf{3) Mapping cones \& resolutions of intersections}: \ \ Let  $p \in \gamma_1 \overrightarrow{\cap} \gamma_2$ be different from a puncture. Then, the resolution of $p$ (Figure \ref{FigureResolutionOfCrossings}) is a representative of the mapping cone of $\basis(p)$.
		\begin{figure}[H]
			\begin{displaymath}
			\begin{tikzpicture}
			\draw[white] (-1.5,0)--(7.5,0);

			\draw (0, -1)--(0,1);
			\draw (-1,0)--(1,0);
			\draw[->] (2,0)--(4,0);
			\draw[rounded corners] (6,-1)--(6,0)--(7,0);
			\draw[rounded corners] (6,1)--(6,0)--(5,0);			
			\filldraw (0,0) circle (2pt);
			\draw (0.3,0.3) node{$p$};
			\draw (0,-1.3) node{$\gamma_1$};
			\draw (-1.3, 0) node{$\gamma_2$};
			
			\end{tikzpicture}\end{displaymath}
			\begin{displaymath}
			\begin{tikzpicture}
		\draw[white] (-1.5,0)--(7.5,0);
			\draw (0, 0)--(0,1);
			\draw (-1,0)--(0,0);
			\draw[->] (2,0)--(4,0);

			\draw[rounded corners] (6,1)--({6+0.2*cos(135)},{0+0.2*sin(135)})--(5,0);
			\draw[decoration={markings, mark=at position 0.6 with {\arrow{<}}},
			postaction={decorate}
			] ({6.5},{-0.5}) circle ({sqrt(0.5)});
			
			\filldraw (0,0) circle (2pt);
			\filldraw (6,0) circle (2pt);
			\draw (0.3,0.3) node{$p$};
			
			\draw (0,1.3) node{$\gamma_1$};
			\draw (-1.3, 0) node{$\gamma_2$};
			\draw[decoration={markings, mark=at position 0.6 with {\arrow{<}}},
			postaction={decorate}
			] ({0.5},{-0.5}) circle ({sqrt(0.5)});
	
			\end{tikzpicture}\end{displaymath}
			\caption{} \label{FigureResolutionOfCrossings}
		\end{figure} \ \\
		
	\noindent	\textbf{4) Compositions \& immersed triangles}: \ \ For $i \in \{1,2\}$, let $q_i \in \gamma_{i} \overrightarrow{\cap} \gamma_{i+1}$ of degree $0$. We assume that if $\gamma_i \simeq \gamma_j$ for all $i, j \in [1,3]$, then $\gamma_1, \gamma_2$ and $\gamma_3$ are arcs. The following is true.\smallskip
	
		\begin{itemize}
		\setlength\itemsep{0.5em}

			\item[I)]  Then,
		$\basis(q_2) \circ \basis(q_1)$ is a linear combination of precisely those morphisms $\basis(q_3)$ ($q_3 \in \gamma_1 \overrightarrow{\cap} \gamma_3$), such that there exist lifts $\widetilde{\gamma}_i$ of $\gamma_i$ to the universal cover of $\mathcal{S}$ intersecting in an \textit{intersection triangle}, a \textit{fork} or a \textit{double-bigon} as shown in Figure \ref{FigureIntersectionTriangle}.

			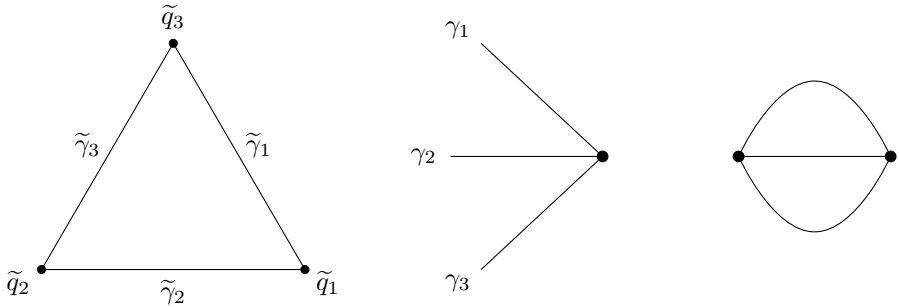
\begin{figure}[H]
				\begin{displaymath}\arraycolsep=10pt
				\begin{array}{ccc}
		{		
				\begin{tikzpicture}
				\draw[white] (0,2)--(0,-2);
				\foreach \i in {1,2,3}
				{
					\draw  ({2*cos(90+\i * 120)},{2*sin(90+\i *120)})--({2*cos(210+\i * 120)},{2*sin(210+\i *120)});
					\filldraw ({2*cos(90+\i * 120)},{2*sin(90+\i *120)}) circle (1.5pt);
					\draw ({2.35*cos(90-\i * 120)},{2.35*sin(90-\i *120)}) node{$\widetilde{q}_{\i}$};
					\draw ({1.3*cos(150-\i * 120)},{1.3*sin(150-\i *120)}) node{$\widetilde{\gamma}_{\i}$};
				}
				\end{tikzpicture}
			} & 
		{				\begin{tikzpicture}
		\draw[white] (0,2)--(0,-2.5);

			\filldraw (-1,0) circle (2pt);
			
			\foreach \i in {1,2,3}
			{
				\pgfmathsetmacro\v{\i-2}
				\draw (-1,0)--({3*cos(180+\v*30)},{3*sin(180+\v*30)});
				\draw ({3.35*cos(180+\v*30)},{3.35*sin(180+\v*30)}) node{$\gamma_{\i}$};
			} 
			\end{tikzpicture}
		}   &

				{\begin{tikzpicture}
				\draw[white] (0,2)--(0,-2.5);
				
					\filldraw (-1,0) circle (2pt);
					\filldraw (-3,0) circle (2pt);
					\foreach \i in {1,2,3}
					{
						\draw [color=black] plot  [smooth, tension=1] coordinates {   (-1,0) (-2,2-\i) (-3,0)};	
					}

					\end{tikzpicture}}
		\end{array}
				\end{displaymath}
				\caption{From left to right: A intersection triangle, a fork and a double-bigon.} \label{FigureIntersectionTriangle}
			\end{figure}

			\item[II)] Suppose $\gamma(X_1)=\gamma(X_3)$ are homotopy classes of loops and $q_1=q_2 \in \mathcal{S}$. Then $\basis(q_2) \circ \basis(q_1)$ is not a linear combination of morphisms associated to intersections.
			\item[III)] If $p \in \gamma_1 \overrightarrow{\cap} \gamma_2$ is a puncture and $q$ denotes the corresponding self-intersection of $\gamma_2$, then $\basis(p)(n)=\basis(q)^n \circ \basis(p)(0)$ and $\basis(q)(n)=\basis(q)^n$. 
		\end{itemize} \ \\

\noindent	\textbf{5) Auslander-Reiten theory}: \ \  Let $(\gamma, \mathfrak{g})$ be a graded loop, let $Q \in k[t]$ be irreducible and for all $i \geq 0$, let $\mathcal{V}_i$ be an indecomposable local system of type $Q^i$. The, the Auslander-Reiten quiver of $\T$ contains a homogeneous tube 
		
		\[
		\begin{tikzcd}
		\cdots \arrow[bend right]{r} &
		X_3 \arrow[bend right]{r} \arrow[bend right]{l} &
		X_2 \arrow[bend right]{r} \arrow[bend right]{l} &
		X_1, \arrow[bend right]{l}
		\end{tikzcd}
		\]
		
\		\smallskip

\noindent where $X_i \in \T$ is indecomposable and $\gamma(X_i)=(\gamma,\mathfrak{g}, \mathcal{V}_i)$. In particular, $X_i$  is $\tau$-invariant.\\

\end{addmargin}
\noindent A triangulated category $\T$ is called \textbf{surface-like} if there exists a surface-like quintuple for $\T$.
\end{definition}
\smallskip

\subsubsection{Conventions and notation concerning surface-like categories}
\ \smallskip

\noindent Let $(\T, \mathcal{S}, \eta, \gamma, \basis)$ be a surface-like quintuple. In most parts of this paper we only consider indecomposable objects in a surface-like category which are represented by either arcs or loops with a $1$-dimensional local system and the term ``indecomposable'' in a surface-like triangulated category refers to an indecomposable object of the above type unless stated otherwise (usually indicated by the word ``arbitrary'').

Frequently, we write $X_{\gamma}$ for the indecomposable object (in the above sense), whose associated homotopy class of curves contains the curve $\gamma$. Moreover, we often omit local systems and regard $\gamma(X)$ as the homotopy class of curves associated with $X$. If $X$ is indecomposable, $\gamma_X$ shall denote a curve in minimal position in the homotopy class $\gamma(X)$. 

Given two indecomposable objects $X ,Y \in \T$ (in the sense above) and representing arcs $\gamma_X \in \gamma(X)$, $\gamma_Y \in \gamma(Y)$ in minimal position, we say that $f=\sum_{i=1}^m{f_i}$ is a \textbf{standard decomposition} of a morphism $f: X \rightarrow Y$ with respect to $\gamma_X$ and $\gamma_Y$, if each $f_i$ is a non-zero multiple of $\basis(p_i)$ for some intersection $p_i \in \gamma_X \overrightarrow{\cap} \gamma_Y$ and $p_i \neq p_j$ for all $i \neq j$.

 We say that an indecomposable object $X \in\T$ is a (finite or infinite) \textbf{arc object} if $\gamma(X)$ is a homotopy class of (finite or infinite) arcs. Otherwise, we call it a \textbf{loop object}.\medskip

\noindent In particular, all loop objects are $\tau$-invariant and we will see that -- with finitely many exceptions up to shift-- these are the only $\tau$-invariant objects.\medskip

\subsection{Derived categories of gentle algebras are surface-like}\label{SectionGentleSurfacelLike}
\noindent As pointed out earlier, bounded derived categories of gentle algebras are prototypical examples of surface-like categories.  We prove the following.

\begin{thm}\label{TheoremGentleAlgebrasFukayaLike} Let $A$ be a gentle algebra. Let $\mathcal{S}_{A}$ denote the surface of $A$ as constructed in \cite{OpperPlamondonSchroll}, where we replace each boundary component without marked points by a puncture. Then, there exists a surface-like quintuple $(\mathcal{D}^b(A), \mathcal{S}_{A}, \eta_A, \gamma, \basis)$.
\end{thm}
\noindent This essentially follows from the results in \cite{OpperPlamondonSchroll}. However, what remains to be shown is that compositions of morphisms can be understood geometrically. In order to give a proof, we recall some of the results in \cite{OpperPlamondonSchroll}.  In particular, we included a short description surface of a gentle algebra and the maps $\gamma$ and $\basis$. The line field $\eta_A$ is described in Section \ref{SectionGradedSurfaceGentleAlgebra}.

\subsubsection{The graded surface of a gentle algebra}\label{SectionGradedSurfaceGentleAlgebra}
\ \smallskip

\noindent Our main reference is Section 1 in \cite{OpperPlamondonSchroll}. To begin with, we recall the definition of a gentle algebra.

\begin{definition}
A pair $(Q,I)$ consisting of a quiver $Q$ and an ideal $I$ of $kQ$ is called \textbf{gentle} if
	\begin{enumerate}
\setlength\itemsep{0.5em}
	
		\item $Q$ is a finite quiver;
		\item $I$ is admissible, i.e.\ $R^m\subset I \subset R^2$ for some $m \geq 0$, where $R=(Q_1)$ is the ideal generated by the arrows of $Q$;
		\item $I$ is generated by paths of length $2$;
		\item for every arrow $\alpha$ of $Q$, there is at most one arrow $\beta$ such that $\alpha\beta \in I$;
		at most one arrow $\gamma$ such that $\gamma\alpha \in I$;
		at most one arrow $\beta'$ such that $\alpha\beta' \notin I$;
		and at most one arrow $\gamma'$ such that $\gamma'\alpha \notin I$.
	\end{enumerate}
A \textbf{gentle algebra} is an algebra of the form $kQ /I$ for a gentle quiver $(Q,I)$.
\end{definition}

Rather than giving a constructive description of the surface, it is more convenient for our purposes to define the surface of a gentle algebra as the unique surface with certain properties.
The surface of a gentle algebra $A$ is defined in terms of its gentle quiver. It is up to diffeomorphism the unique marked surface $\mathcal{S}_A=(S_A, \marked)$ with the following four properties:

\begin{enumerate}
\setlength\itemsep{0.5em}

    \item There exists a set $\{L_x \, | \, x \in Q_0\}$ of pairwise disjoint embedded paths (called \textit{laminates}) which  intersect the boundary only at its end points and transversally.

    \item The complement of all laminates in $S_A$ is a collection of discs $\{ \Delta_p \, | \, p \in \marked\}$ and for each $p \in \marked$, $\partial \Delta_p$ contains $p$ on its boundary and no other marked point.
    
    \item There exists a bijection between the set of those discs $\Delta_p$ such that $ \partial \Delta_p$ contains more than one laminate and the maximal admissible paths in $(Q,I)$. Moreover, if $\Delta_p$ corresponds to an admissible path $\alpha_n \dots \alpha_1$, then the ordered sequence of laminates on $\partial \Delta_p \setminus \{p\}$ (following the clockwise orientation of the boundary) is  $L_{s(\alpha_1)}, \dots, L_{s(\alpha_n)}, L_{t(\alpha_n)}$ .
\end{enumerate}
The surface $\mathcal{S}_A$ has no punctures but contains unmarked boundary components if $A$ has infinite global dimension.\medskip 

\noindent We turn $\mathcal{S}_A$ into a graded surface by gluing a line field $\eta_A$ from line fields $\eta_p$  on the closures of the discs $\Delta_p$. Let $\eta_p$ be any line field such that the following is true.

\begin{itemize}
\setlength\itemsep{0.5em}

    \item For all $p \in \marked$ and all laminates $L \subset \partial \Delta_p$, $\eta_p$ and $\dot L$ are parallel at all points of $L$.

    \item Any segment of $\partial \mathcal{S}$ between two consecutive laminates in $\partial \Delta_p$, which does not contain $p$, has winding number $1$.
    
\end{itemize}
Then, $\eta_A$ is the unique line field which restricts to $\eta_p$ on $\Delta_p$ for all $p \in \marked$. This determines $\eta$ uniquely up to homotopy. Figure \ref{FigureFoliationEtaP} shows the foliation defined by $\eta_p$. One may think of the line field $\eta_p$ as having a singularity outside of $\Delta_p$.

   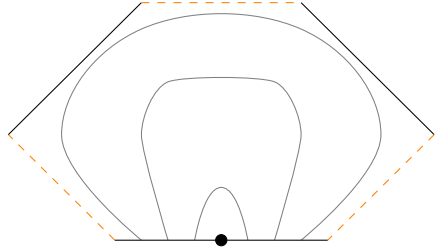
\begin{figure}[H]

	\centering

	\begin{tikzpicture}[scale=0.7]

\draw[dashed, orange] (-2,-2)--(-4,0);
\draw (-4,0)--(-1.5,2.5);
\draw[dashed, orange]  (-1.5,2.5)--(1.5,2.5);
\draw (1.5,2.5)--(4,0);
\draw[dashed, orange]  (4,0)--(2,-2);
\draw (-2,-2)--(2,-2);

\filldraw (0,-2) circle(3pt);

\draw[opacity=0.5]  plot[smooth, tension=.7] coordinates {(-1.5,-2) (-3,0) (-1.5,2) (1.5,2) (3,0) (1.5,-2)};
\draw[opacity=0.5]  plot[smooth, tension=.4] coordinates {(-1,-2) (-1.5,0) (-1,1) (1,1) (1.5,0) (1,-2)};
\draw[opacity=0.5] plot[smooth, tension=1.2] coordinates {(-0.5,-2) (0,-1) (0.5,-2)};

 	\end{tikzpicture}
	\caption{The foliations of a line field $\eta_p$. Dashed lines correspond to laminates. Solid black line indicate boundary segments.}  \label{FigureFoliationEtaP}
\end{figure} 
For the remainder of Section \ref{SectionGentleSurfacelLike}, we assume that all curves and the laminates are in minimal position with the laminates.
Given a curve $\gamma$ on $\mathcal{S}$, the intersections with the laminates dissect $\gamma$ into \textit{segments} which start and end at consecutive intersections.
    \begin{lem}\label{LemmaWindingNumberGentle}
    Let $\gamma$ be a curve. Denote by $\gamma_1, \dots, \gamma_m$ the segments of $\gamma$ and by $\Delta_i$ the disc which contains $\gamma_i$. Then, 
    \[\begin{array}{ccc} {\omega_{\eta_A}(\gamma)=\sum_{i=1}^m {\omega_{\eta_A}(\gamma_i)}} & \text{and} & 
    {\omega_{\eta_A}(\gamma_i) = \begin{cases} \phantom{-}1, & \text{if } p_i \text{ is on the left hand side of }\gamma_i \text{ in }\Delta_i; \\ -1 & \text{otherwise},  \end{cases}} \end{array}\]
where $p_i \in \Delta_i$ is the unique marked point.
        \end{lem}
\begin{proof}
Since $\gamma$ is transversal to all laminates and laminates are parallel to $\eta_A$, the algebraic intersection number of $\dot \gamma$ with $\eta_A(\mathcal{S})$ is the sum of the algebraic intersection numbers of its segments. 
\end{proof}

\noindent Let  $(\gamma, \mathfrak{g})$ be a graded curve. We equip the laminates with the grading corresponding to $0 \in \mathbb{Z}$. We write $\gamma \overrightarrow{\cap} L$ for the union of all sets $\gamma \overrightarrow{\cap} L_x$, where $L_x$ is a laminate. It inherits an order from the orientation of $\gamma$. The grading $\mathfrak{g}$ defines a function 
\[\begin{tikzcd} g: &  \gamma \overrightarrow{\cap} L \arrow{r} & \mathbb{Z} \\
& p \arrow[mapsto]{r} & \deg(p). \end{tikzcd}\]
Let $p \in L \overrightarrow{\cap} \gamma$ with successor $p'$. Let $\delta$ denote the segment of $\gamma$ between $p$ and $p'$ and write $\Delta$ for the disc which contains $\delta$. Then,
 \begin{equation}\label{Equation2}
  g(p') = \begin{cases}
              g(p) + 1, & \textrm{if the unique marked point of $\Delta$ is on the left of $\delta$}; \\
              g(p) - 1, & \textrm{otherwise.}
             \end{cases}
\end{equation}
Unravelling the definitions, we see that a choice of $\mathfrak{g}$ is equivalent to a choice of a function $g: \gamma \overrightarrow{\cap} L \rightarrow \mathbb{Z}$ satisfying (\ref{Equation2}).

\subsubsection{The bijection between objects and curves}\label{SectionBijectionObjectsCurves}
\ \smallskip

\noindent Our main reference is Section 2 of \cite{OpperPlamondonSchroll}. Rather than describing $\gamma$, it is easier to describe its inverse $X_{(-)}$, which associates an object in $\mathcal{K}^{b,-}(A-\proj)$ to any triple $(\gamma, \mathfrak{g}, \mathcal{V})$ of a graded curve $(\gamma, \mathfrak{g})$ on $\mathcal{S}_A$ (as a surface without punctures) and a local system $\mathcal{V}$ on $\gamma$.\medskip

\noindent Let $q_1, \dots, q_l$ denote the totally ordered sequence of intersections of $\gamma$ with the laminates. We write $L_{x_i}$ for the laminate which contains $q_i$. In case of loops, indices are considered modulo $l$. Denote by $\gamma_i$ the segment of $\gamma$ between $q_i$ and $q_{i+1}$ and denote by $\Delta_i$ the disc which contains $\gamma_i$.  We choose the orientation of $\gamma_i$ such that the marked point of $\Delta_i$ lies on the left. We denote by $g: \gamma \overrightarrow{\cap} L \rightarrow \mathbb{Z}$ the function corresponding to $\mathfrak{g}$.\medskip

\noindent Then, $X_{(\gamma, \mathcal{V})}$ is the following complex. 

\begin{itemize}
\setlength\itemsep{0.5em}

    \item In degree $n \in \mathbb{Z}$, $X_{(\gamma, \mathcal{V})}$ is given by the direct sum  
    \[\bigoplus_{i \in [1,l], g(q_i)=n}{P_{x_i} \otimes_k \mathcal{V}(q_i)},\]
    where $P_{x_i}$ is the indecomposable projective module $Ax_i$.

    \item The differential $\partial$ of $X_{(\gamma, \mathcal{V})}$ is defined as $\partial= \sum_{i} \left(\partial_i \otimes_k \mathcal{V}(\gamma_i)\right)$, where $\partial_i$ is the map between $P_{x_i}$ and $P_{x_{i+1}}$ induced by the following admissible path $u_i$.\smallskip
    
    \noindent Let $\alpha=\alpha_m \cdots \alpha_1$ denote the maximal admissible path corresponding to $\Delta_i$. If $L_{y_1}, \dots, L_{y_s}$ is the totally ordered sequence of laminates in $\partial \Delta_i \setminus \marked$ between $L_{x_i}$ and $L_{x_{i+1}}$ (including the two), then $u_i$ is the subpath of $\alpha$ which passes through the vertices $y_1, \dots, y_s$ in the given order.
\end{itemize}
The isomorphism class of $X_{(\gamma, \mathcal{V})}$ depends on $\gamma$ only up to homotopy and on $\mathcal{V}$ only up to isomorphism. We may assume that $\mathcal{V}$ is trivial on all segments $\gamma_i$ except for $\gamma_1$. In this case and if $\gamma$ is an arc (resp. loop), then $X_{(\gamma, \mathcal{V})}$ is a \textbf{string complex} (resp. \textbf{band complex}) in the sense of \cite{BekkertMerklen}, see also \cite{BurbanDrozd2004}.
\begin{thm}[Theorem 2.12, \cite{OpperPlamondonSchroll}]\label{TheoremBijectionObjectsCurves}
$X_{(-)}$ is invertible and its inverse $\gamma(-)$ satisfies Property 1) of Definition \ref{DefinitionFukayaLikeQuintuples}.
\end{thm}
\smallskip

\subsubsection{The morphism of an intersection}\label{SectionMorphismsOfIntersection}
\ \smallskip

\noindent We have the following result.

\begin{thm}[Theorem 3.3, Remark 3.8, Theorem 4.2, \cite{OpperPlamondonSchroll};  Proposition 5.16, \cite{ArnesenLakingPauksztello}]\label{TheoremInjectionMorphisms}
There exists an injection $\basis$ satisfying Property 2) and Property 3) in Definition \ref{DefinitionFukayaLikeQuintuples} with respect to $\gamma(-)$.
\end{thm}
\noindent The definition of $\basis$ is given in Section 3, \cite{OpperPlamondonSchroll}. and we present its construction in a suitable way. \smallskip

\noindent Let $(\gamma_i, \mathfrak{g}_i, \mathcal{V}_i)$ ($i=1,2$) be graded curves on $\mathcal{S}_A$ equipped with local systems. To be clear, $\gamma_i$ is a curve on the punctured surface. However, in what follows, we have to consider the corresponding curves on the unpunctured version of $\mathcal{S}_A$. In particular, we are dealing with non-compact curves which wind around one or two  unmarked boundary components in clockwise direction.  

\noindent The map $\basis$ associates to an intersection a chain map in the basis of $\mathcal{C}^{-, b}(A-\proj)$ (resp. $\mathcal{K}^{-, b}(A-\proj)$) as described by Arnesen, Laking and Pauksztello \cite{ArnesenLakingPauksztello}. We refer to elements of such bases as \text{ALP maps}.\medskip

\noindent The construction of $\basis$  can be described as follows. Let $p \in \gamma_1 \overrightarrow{\cap} \gamma_2$. We may assume that $p$ is of degree $0$ and want to describe a morphism $\basis(p): X_{(\gamma_1, \mathfrak{g}_1, \mathcal{V}_1)} \rightarrow X_{(\gamma_2, \mathfrak{g}_2, \mathcal{V}_2)}$.
We choose lifts $\widetilde{\gamma}_i$ ($i \in \{1,2\})$ of $\gamma_i$ to a universal cover $\widetilde{\mathcal{S}}_A$ which intersect in a lift $\widetilde{p}$ of $p$. Then, $\widetilde{p}$ is the only intersection of the two curves. Moreover, $\widetilde{\gamma}_i$ is simple and intersects every lift of a laminate at most once.
The lifts of marked points and $\eta_A$ turn $\widetilde{\mathcal{S}}_A$ into a graded  marked surface and the collection of all lifts of laminates cut $\widetilde{\mathcal{S}}_A$ into discs.

The collection of the discs in $\widetilde{\mathcal{S}}_A$ which contains segments of both $\widetilde{\gamma}_1$ and $\widetilde{\gamma}_2$ form a (not necessarily compact) subsurface $S_{\tilde{p}}$ as in Figure \ref{FigureSurfaceOfIntersection}. We denote by $\delta_i$ the segment of $\widetilde{\gamma}_i$ which is contained in $S_{\tilde{p}}$. An intersection of $\delta_i$ with a laminate corresponds uniquely to an intersection of $\gamma_i$ with a laminate.

\begin{figure}[H]
% 	\captionsetup{labelformat=empty}
% 	\captionsetup{justification=centering,singlelinecheck=false, format=hang}
	\centering
	\begin{tikzpicture}[scale=0.8]
	%----arcs
	\draw [line width=0.5, color=blue] plot  [smooth, tension=1] coordinates {  (-3.5,1.5) (-1, 0.5) (2, 0.5) (4, -0.5) (6, -0.5) (8.5, -1.5)};
	\draw [line width=0.5, color=red] plot  [smooth, tension=1] coordinates {  (-3.5,-1.5) (-1, -0.5) (2, -0.5) (3, 0.5) (6, 0.5) (8.5, 1.5)};
	%----boundary
	%---left+right
	\draw [line width=0.5, color=black] plot  [smooth, tension=1] coordinates {  (-3.5,-1) (-2, 0) (-3.5, 1)};
	\draw [line width=0.5, color=black] plot  [smooth, tension=1] coordinates {  (8.5,-1) (7, 0) (8.5, 1)};
	%---upper right+ lower right
	\draw [line width=0.5, color=black] plot  [smooth, tension=1] coordinates {  (8.5,2) (7, 1.5) (4, 1) (3,2.5)};	
	\draw [line width=0.5, color=black] plot  [smooth, tension=1] coordinates {  (8.5,-2) (7, -1.5) (2, -1) (0.5, -3)};	
	%---upper left+ lower left	
	\draw [line width=0.5, color=black] plot  [smooth, tension=1] coordinates {  (-3.5,2) (-2, 1.5) (1, 1) (2,2.5)};
	\draw [line width=0.5, color=black] plot  [smooth, tension=1] coordinates {  (-3.5,-2) (-1.5, -1.5) (-0.5, -3)};		
	%----laminates
	%---upper left+ lower left
	\draw [line width=1, color=black, dashed] plot  [smooth, tension=1] coordinates {  (-3.5,1) (-3.5, 2)};
	\draw [line width=1, color=black, dashed] plot  [smooth, tension=1] coordinates {  (-3.5,-1) (-3.5, -2)};
	%---upper right+ lower right
	\draw [line width=1, color=black, dashed] plot  [smooth, tension=1] coordinates {  (8.5,1) (8.5, 2)};
	\draw [line width=1, color=black, dashed] plot  [smooth, tension=1] coordinates {  (8.5,-1) (8.5, -2)};
	%---top & down
	\draw [line width=1, color=black, dashed] plot  [smooth, tension=1] coordinates {  (3,2.5) (2, 2.5)};
	\draw [line width=1, color=black, dashed] plot  [smooth, tension=1] coordinates {  (0.5,-3) (-0.5, -3)};
	%---other arcs (left to right)
	\draw [line width=1, color=black, dashed] plot  [smooth, tension=1] coordinates {  (-1.5, -1.5) (-1,0) (-2,1.5)};
	\draw [line width=1, color=black, dashed] plot  [smooth, tension=1] coordinates {  (2,-1) (1,0) (1,1)};
	\draw [line width=1, color=black, dashed] plot  [smooth, tension=1] coordinates {  (3,-0.88) (3.5,0) (4,1)};
	\draw [line width=1, color=black, dashed] plot  [smooth, tension=1] coordinates {  (5, -1.09) (6,0) (7,1.5)};
	
	%----nodes
	\draw[color=blue] (-4, 1.5) node{$\delta_1$};
	\draw[color=red] (-4, -1.5) node{$\delta_2$};
	
	\end{tikzpicture}
	\caption{The surface $S_{\tilde{p}}$: Dashed curves indicate laminates. The blue and the red curve show $\delta_1$ and $\delta_2$.}  \label{FigureSurfaceOfIntersection}
\end{figure}
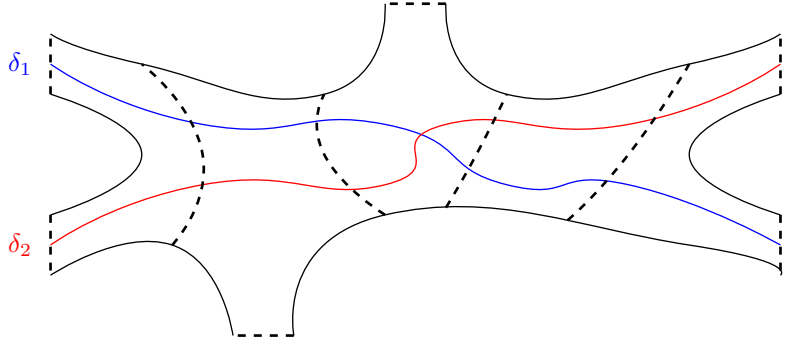
\noindent We assume that $\delta_1$ and $\delta_2$ are oriented in such a way that they cross the laminates in $S_{\tilde{p}}$  in the same order. The ALP map $\basis(p)$ is constructed in a recursive procedure by ``propagating maps along $S_{\tilde{p}}$''.\medskip

\noindent We use the notation from Section \ref{SectionBijectionObjectsCurves} and write $q_i^j$ (resp. $L_{x_i^j}$ and so forth) for the intersections of $\gamma_j$ with the laminates (resp. for the laminate containing $q_i^j$). 

Let $i$ and $i'$ be such that $f_1(q_i^1)=f_2(q_{i'}^2)$ and such that the lifts of $q_i^1$ and  $q_{i'}^2$ in $S_{\tilde{p}}$ lie on the boundary of a disc $\Delta$. As in Section \ref{SectionBijectionObjectsCurves}, $q_i^1$ and $q_{i'}^2$ give rise to a map $f_{(i,i')}: P_{x_i^1} \otimes \mathcal{V}_1(q_i^1) \rightarrow P_{x_{i'}^2} \otimes \mathcal{V}_2(q_{i'}^2)$. We denote by $\partial_i^j$ the component of the differential $\partial^j$ of $X^j=X_{(\gamma_i, \mathcal{V}_i)}$ defined by the segment between $q_i^1$ and $q_{i+1}^1$. Then, the composition $f_{(i,i')} \circ \partial_i^1$ is defined and non-zero if and only if a lift of $q_{i+1}^1$ lies on $\partial \Delta$ and comes before $q_i^1$ in the total order of $\partial \Delta \setminus \widetilde{\marked}$, where $\widetilde{\marked}$ denotes the set of lifts of marked points. In this case, the pair $(i+1, i'+1)$ defines a map $f_{(i+1,i'+1)}$ as before. Moreover, $f_{(i,i')} \circ \partial_i^1=\partial_{i'+1}^2 \circ f_{(i+1,i'+1)}$. Analogous constructions define the map $f_{(i+1,i'+1)}$ in case $\partial^2_i \circ f_{(i,i')} \neq 0$ and the map $f_{(i-1,i'-1)}$ under the appropriate assumptions. 

Performing the previous step as often as possible, we obtain maps $f_{(j,j')}$ for various pairs $(j,j')$ and their sum is a chain map. Its homotopy class is the desired morphism $\basis(p)$. The observation that elements in the ALP basis can be reconstructed from any of their components was made in Lemma 4.3, \cite{ArnesenLakingPauksztello}.

\begin{rem}\label{RemarkIntersectionsGentle}\
\begin{enumerate}
\setlength\itemsep{0.5em} 

    \item Different choices of the initial pair $(i,i')$ may give different but homotopic chain maps.
    
    \item\label{Point2} If  $p \in \gamma_1 \overrightarrow{\cap} \gamma_2$ is a puncture, it means that the surface $S_{\tilde{p}}$ as above is non-compact and the lifts of $\gamma_1$ and $\gamma_2$ intersect ``at infinity''. In this case, either $\basis(p)$ consists of a single component or is not finitely supported and the components become cyclic eventually.
    
    \item The construction provides us with a chain map in the ALP basis of $\mathcal{C}^{-, b}(A-\proj)$ even if $\delta_1$ and $\delta_2$ do not intersect. However, if the overlap (the surface $S_{\tilde{p}}$ in case of intersections) is compact the resulting maps are null-homotopic. If the overlap is not compact and we are not in the situation of (\ref{Point2}), then $\gamma_1$ and $\gamma_2$ are homotopic loops. In this case, the associated ALP map is invertible  or the connecting morphism of an Auslander-Reiten triangle.

\end{enumerate}

\end{rem}
\smallskip

\subsubsection{Composition of morphisms in the derived category of a gentle algebra}
\ \smallskip

\begin{thm}\label{TheoremComposition}
The map $\basis(-)$ satisfies Property 4) of Definition \ref{DefinitionMarkedSurface}.
\end{thm}
\begin{proof}

Let $\gamma_1, \gamma_2$ and $\gamma_3$ be graded curves on $\mathcal{S}_A$ which are in minimal position with the laminates and let $p \in \gamma_1 \overrightarrow{\cap} \gamma_2$ and $q \in \gamma_2 \overrightarrow{\cap} \gamma_3$ be of degree $0$. Write $X^i=X_{(\gamma_i, \mathcal{V}_i)}$.
Let $u: P^1 \rightarrow P^3$ be a component of the chain map $\basis(q) \circ \basis(p)$ between indecomposable projective modules $P^1$ and $P^3$ of $X^1$ and $X^3$, i.e.\ $u$ is induced by an admissible path. By construction, there exist components $u_1: P^1 \rightarrow P^2$ of $\basis(p)$ and $u_2:P^2 \rightarrow P^3$ of $\basis(q)$ such that $u$ and $u_2 \circ u_1$ agree up to scalar. Propagating $u$ along the complexes  as in Section \ref{SectionMorphismsOfIntersection}, we obtain a chain map $h: X^1 \rightarrow X^3$. Following the reasoning in the proof of Proposition 4.1, $h$ appears in the standard decomposition of $\basis(q) \circ \basis(p)$ with respect to the ALP basis of $C^{b,-}(A-\proj)$. For appropriate lifts $\tilde{p}$ and $\tilde{q}$ of $p$ and $q$, it follows that the intersection of subsurfaces $S_{\tilde{p}}$ and $S_{\tilde{q}}$ contains the disc which contains the intersection corresponding to $P^1$, $P^2$ and $P^3$. Moreover, the intersection are distributed as in Figure \ref{FigureForkComposition}.

\begin{figure}[H]
		\begin{displaymath}
		\begin{tikzpicture}
		\filldraw ({2*cos(10)},0)  circle (2pt); % marked point
		
		\foreach \u in {4} % laminates
		{\draw [line width=0.5, color=orange, dashed] plot  [smooth, tension=1] coordinates {   ({2*cos(-10+360/5*\u+360/5)}, {2*sin(-10+360/5*\u+360/5)}) ({1.3*cos(360/5*\u+360/10)}, {1.3*sin(360/5*\u+360/10)}) ({2*cos(10+360/5*\u)}, {2*sin(10+360/5*\u)})};
		}
		
		\foreach \u in {1} % laminates
		{
			\draw [ line width=0.5, color=orange, dashed] plot  [smooth, tension=1] coordinates {   ({2*cos(-10+360/5*\u+360/5)}, {2*sin(-10+360/5*\u+360/5)}) ({1.3*cos(360/5*\u+360/10)}, {1.3*sin(360/5*\u+360/10)}) ({2*cos(10+360/5*\u)}, {2*sin(10+360/5*\u)})};
		}
		
		\foreach \u in {1, 2,3, 4, 5} % laminates
		{
			\draw [color=black, dashed, orange] plot  [smooth, tension=1] coordinates {   ({2*cos(-10+360/5*\u+360/5)}, {2*sin(-10+360/5*\u+360/5)}) ({1.3*cos(360/5*\u+360/10)}, {1.3*sin(360/5*\u+360/10)}) ({2*cos(10+360/5*\u)}, {2*sin(10+360/5*\u)})};
		}
		\foreach \u in {5} % boundary
		{
			\draw[line width=0.75] ({2*cos(-10+360/5*\u)}, {2*sin(-10+360/5*\u)})--({2*cos(+10+360/5*\u)}, {2*sin(+10+360/5*\u)});
		}
		\foreach \u in {1,2,3,4}
		{% blue boundary
		\draw[ thick, color=black] ({2*cos(-10+360/5*\u)}, {2*sin(-10+360/5*\u)})--({2*cos(+10+360/5*\u)}, {2*sin(+10+360/5*\u)});
		
		}

		\foreach \u in {2,3} % 
		{
			\draw [line width=0.5, color=black, dotted] plot  [smooth, tension=1] coordinates {    ({1.3*cos(360/5*\u-360/10)}, {1.3*sin(360/5*\u-360/10)}) ({0.75*cos(360/5*\u)}, {0.75*sin(360/5*\u)}) ({1.3*cos(360/5*\u+360/10)}, {1.3*sin(360/5*\u+360/10)})  } [decoration={markings, mark=at position 0.51 with {\arrow{>}}}, postaction={decorate} ];
			
			\filldraw ({1.3*cos(360/5*\u-360/10)}, {1.3*sin(360/5*\u-360/10)}) circle(1.5pt);
		
		}
		
		\draw [line width=0.5, color=black, dotted] plot  [smooth, tension=1] coordinates {    ({1.3*cos(360/5*2-360/10)}, {1.3*sin(360/5*2-360/10)}) (0,0) ({1.3*cos(360/5*3+360/10)}, {1.3*sin(360/5*3+360/10)})  } [decoration={markings, mark=at position 0.51 with {\arrow{>}}}, postaction={decorate} ];
		\filldraw  ({1.3*cos(360/5*3+360/10)}, {1.3*sin(360/5*3+360/10)}) circle(1.5pt);

		% projectives
		
			\draw ({1.7*cos(360/5*2-360/10)}, {1.7*sin(360/5*2-360/10)}) node{$P^{1}$};
			\draw ({1.7*cos(360/5*3-360/10)}, {1.7*sin(360/5*3-360/10)}) node{$P^{2}$};
			\draw ({1.7*cos(360/5*4-360/10)}, {1.7*sin(360/5*4-360/10)}) node{$P^{3}$};

		\end{tikzpicture}
		\end{displaymath}
		\caption{The components $u$, $u_1$ and $u_2$.} \label{FigureForkComposition}
	\end{figure}
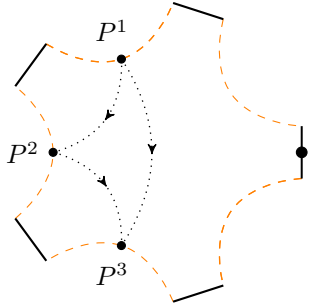

If $h \not \simeq 0$, then again, by the reasoning in the proof of Proposition 4.1, it appears in a standard decomposition of $\basis(q) \circ \basis(p)$ considered as a morphism in the homotopy category. It implies that, if $h \not \simeq 0$

and $\widetilde{\gamma_i}$ is the lift of $\gamma_i$ corresponding to the choices of $\tilde{p}$ and $\tilde{q}$, then $\widetilde{\gamma_1}$ and $\widetilde{\gamma_3}$ intersect in a point corresponding to $h$. Therefore, the lifts intersect pairwise and exactly once. Thus, they form a triangle, a fork or a double-bigon.

Next, assume that lifts $\widetilde{\gamma}_1, \widetilde{\gamma}_2, \widetilde{\gamma}_3$  intersect in lifts of $p, q$ and an intersection $q' \in \gamma_1 \overrightarrow{\cap} \gamma_3$. 

We want to show that a multiple of $\basis(q')$ appears in a standard decomposition of $\basis(q) \circ \basis(p)$. As before, Our assumptions guarantee that we find a disc as in Figure \ref{FigureForkComposition} which implies that there exist components $u_1:P^1 \rightarrow P^2$ and $u_2: P^2 \rightarrow P^3$ of $\basis(p)$ and $\basis(q)$ respectively such that $u=u_2 \circ u_1$ is a multiple of a component of $\basis(q')$ and suppose that $\basis(q')$ is not a summand of $\basis(q) \circ \basis(p)$. Since ALP maps are determined by any of their components, this requires $u$ to cancel with a sum of other maps of the form  $v:P^1 \xrightarrow{v_1} Q^2 \xrightarrow{v_2} P^3$ with $Q^2 \neq P^2$ and where $v_1$ and $v_2$ are components of $\basis(p)$ and $\basis(q)$. However, it this case $u$ and $v$ are components of the same ALP map.\smallskip

\noindent Property 4) II) in Definition \ref{DefinitionFukayaLikeQuintuples} follows from compactness of the intersection of $S_{\tilde{p}}$ and $S_{\tilde{q}}$.
\end{proof}

\begin{proof}[Proof of Theorem \ref{TheoremGentleAlgebrasFukayaLike}] We have to verify Properties 1)--5) of Definition \ref{DefinitionFukayaLikeQuintuples}.
They follow from  Theorem \ref{TheoremBijectionObjectsCurves}, Theorem \ref{TheoremInjectionMorphisms}, Theorem 4.2 in \cite{OpperPlamondonSchroll}, Theorem \ref{TheoremComposition} and the description of Auslander-Reiten triangles for band complexes from \cite{Bobinski}.
\end{proof}
\smallskip

\subsection{Auslander-Reiten theory of arc objects}
We show the existence of Auslander-Reiten triangles for arc objects in a surface-like category and give a geometric description of the Auslander-Reiten translation and the connecting morphism in Auslander-Reiten triangles. Our result is based purely on geometric nature of surface-like categories and provides a new perspective on the combinatorial proofs of related results for gentle algebras in \cite{ArnesenLakingPauksztello} and \cite{OpperPlamondonSchroll}.

\begin{definition}\label{DefinitionFractionalTwist}Let $B \subseteq \partial \mathcal{S}$ be a component with $N$ marked points. Let $D$ be a tubular neighborhood of  $B \subseteq \partial S$ with an orientation preserving diffeomorphism 
\[\begin{tikzcd}\phi:D \arrow{r} & \{z \in \mathbb{C} \, | \,  1 \leq \lVert z \rVert \leq 2\},\end{tikzcd}\]

\noindent such that 
$\phi(B \cap \marked)$ is the set of roots of unity of order $N$.
 The \textbf{fractional twist} $\tau_B: S \rightarrow S$ is the map which extends the identity of $\overline{S\setminus D}$ and restricts to the map
\[\begin{tikzcd}
x \arrow[mapsto]{r} & \phi^{-1}\left(\phi(x) \cdot \exp\left(\frac{2\pi i}{|I|} \cdot (2-\left\lVert \phi(x)\right\rVert) \right)\right)\end{tikzcd}\] 

on $D$, where $\exp$ denotes the exponential function.
\end{definition}
\noindent The diffeomorphism $\tau_B$ rotates the surface in a neighborhood of $B$ in counter-clockwise direction, as shown in Figure \ref{FigureFractionalTwist}. 
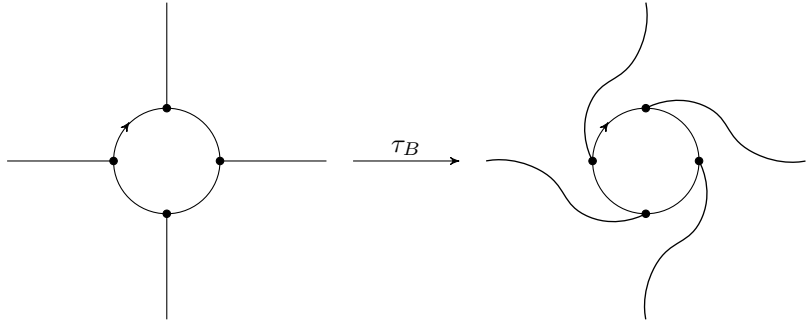
\begin{figure}[H]
\begin{displaymath}
\begin{tikzpicture}[scale=0.7]
\draw[decoration={markings, mark=at position 0.4 with {\arrow{<}}},
postaction={decorate}
] (0,0) circle (1);
\foreach \x in {0,1,2,3}
{
	\filldraw ({1*cos(\x*180*0.5)},{1*sin(\x*180*0.5)}) circle (2pt);
	\draw ({3*cos(\x*180*0.5)},{3*sin(\x*180*0.5)})--({1*cos(\x*180*0.5)},{1*sin(\x*180*0.5)});
}

\draw[->] (3.5, 0)--(5.5,0);
\draw (4.5,0.25) node{$\tau_B$};

\draw[decoration={markings, mark=at position 0.4 with {\arrow{<}}},
postaction={decorate}
] (9,0) circle (1);
\foreach \x in {0,1,2,3}
{
	\filldraw ({1*cos(\x*180*0.5)+9},{1*sin(\x*180*0.5)}) circle (2pt);
	\draw [line width=0.5, color=black] plot  [smooth, tension=1] coordinates  { ({3*cos(\x*180*0.5)+9},{3*sin(\x*180*0.5)})  ({2*cos(\x*180*0.5+5)+9},{2*sin(\x*180*0.5+5)}) ({1.5*cos(\x*180*0.5+45)+9},{1.5*sin(\x*180*0.5+45)}) ({1*cos(\x*180*0.5+90)+9},{1*sin(\x*180*0.5+90)})};
	
}

\end{tikzpicture}
\end{displaymath}
\caption{The action of $\tau_B$ on arcs} \label{FigureFractionalTwist}
\end{figure}
\noindent The isotopy class of $\tau_B$ relative to the boundary is independent of the choice of $D$ and $\phi$. Further, $\tau_B$ and $\tau_{B'}$ commute up to isotopy relative to the boundary. The diffeomorphism $\tau$ is defined as the product of all $\tau_B$ and is well-defined up to isotopy relative to the boundary. Note that $\tau$ is isotopic to the identity (not relative to the boundary). Thus, $\tau(\gamma)$ is naturally graded.  For every finite, graded arc $\gamma$ in minimal position, there exists a distinguished intersection $\gamma \overrightarrow{\cap} \tau(\gamma)$ of degree $1$, see Figure \ref{FigureDistinguishedIntersection}.\medskip 

\begin{figure}[H]
\centering
\begin{tikzpicture}[scale=0.8]

\draw (1.5,2) node{$\gamma$};
\draw (1.5,0.75) node{$\tau(\gamma)$};

\filldraw (2,1) circle (2pt);
\filldraw (-1,1) circle (2pt);
\filldraw (3,2) circle (2pt);
\filldraw (-2,0) circle (2pt);

\draw[decoration={markings, mark=at position 0.4 with {\arrow{<}}},
postaction={decorate}
]  (3,1) ellipse (1 and 1);
\draw[decoration={markings, mark=at position 0.4 with {\arrow{<}}},
postaction={decorate}
]  (-2,1) ellipse (1 and 1);
\draw  plot[smooth, tension=.7] coordinates {(3,2) (2,2) (-1,0) (-2,0)};
\draw  plot[smooth, tension=.7] coordinates {(2,1) (-1,1)};
\end{tikzpicture}
\caption{} \label{FigureDistinguishedIntersection}
\end{figure}

\noindent We prove that the diffeomorphism $\tau$ is a geometric incarnation of the Auslander-Reiten translation in the following sense.

\begin{prp}\label{PropositionExistenceARTrianglesForArcs} Let $\gamma$ be a finite arc on $\mathcal{S}$. If $p \in \gamma \overrightarrow{\cap} \tau \gamma$ is the distinguished intersection, then $\basis(p)$ is the connecting morphism of an Auslander-Reiten triangle in $\T$.  
\end{prp}
As a preparation, we show the following.
\begin{lem}\label{LemmaBoundaryPointsAreInvertible}
	Let $\gamma_1, \gamma_2$ be homotopic arcs on $\mathcal{S}$ in minimal position and let $p \in \gamma_1 \overrightarrow{\cap} \gamma_2$ be a marked point. If $p \in \partial S$, then $\basis(p)$ is invertible. 
\end{lem}
\begin{proof}
	We prove the assertion when $p$ is on the boundary. The other case is similar. Set $f\coloneqq\basis(p)$. It suffices to show that $\Hom(Y, f)$ is surjective.

Let $\gamma_3\simeq \gamma_1$, such that $\{\gamma_1, \gamma_2, \gamma_3\}$ is in minimal position. Let $\widetilde{\gamma}_1,\widetilde{\gamma}_2$ be lifts of $\gamma_1$ and $\gamma_2$ to the universal cover of $S$, which intersect at their end points. By replacing $\gamma_3$ by another representative we may assume that the lift $\widetilde{p} \in \widetilde{\gamma}_1 \overrightarrow{\cap} \widetilde{\gamma}_2$ of $p$ is an element in $\delta \overrightarrow{\cap} \widetilde{\gamma}_1$ and $\delta \overrightarrow{\cap}\widetilde{\gamma}_2$ for some lift $\delta$ of $\gamma_3$.\\
	 Let $q \in \gamma_3 \overrightarrow{\cap}\gamma_2$ and denote $\widetilde{\gamma}_3$ the lift of $\gamma_3$, which intersects $\widetilde{\gamma}_2$ in a lift $\widetilde{q} \in \widetilde{\gamma}_3 \overrightarrow{\cap} \widetilde{\gamma}_2$ of $q$. As all arcs are in minimal position,   $\widetilde{\gamma}_3 \overrightarrow{\cap}\widetilde{\gamma}_1$ contains a single element $\widetilde{q}'$ and $\widetilde{\gamma}_1$ and $\widetilde{\gamma}_2$ bound a disc. By construction, $\widetilde{p}, \widetilde{q}$ and $\widetilde{q}'$ form a fork or the corners of a triangle in clockwise order. Thus, $\basis(q) \circ \basis(p)$ is a non-zero multiple of $\basis(q')$, where $q' \in \gamma_3 \overrightarrow{\cap} \gamma_1$ is the projection of $\widetilde{q}'$. Since $q$ was arbitrary and $\Hom(Y,f)$ is linear, this shows that $\Hom(Y,f)$ is surjective.
\end{proof}

\begin{rem} In case of gentle algebras, the map which is associated to the end points of an arc is the identity morphism. 
\end{rem}

\begin{proof}[Proof of Proposition \ref{PropositionExistenceARTrianglesForArcs}]
	Choose $X=X_{\gamma}$ and $Y=X_{\tau \gamma}$ and set $h\coloneqq \basis(p)$, i.e.\@ $h \in \Hom(X, Y[m])$ for some $m \in 
	\mathbb{Z}$.

	Since $\T$ is Krull-Schmidt, it suffices to show that $h \circ f$ vanishes for all non-invertible morphisms $f: X' \rightarrow X$, where $X'$ is any type of  indecomposable object in $\T$. By Definition \ref{DefinitionFukayaLikeQuintuples} 2ii), it therefore suffices to prove that $h \circ \basis(q)=0$ for all intersections $q \in \delta \overrightarrow{\cap} \gamma$, where $\delta$ is any curve such that $\gamma$ and $\delta$ are in minimal position. 
	
	Suppose, $q$ is an interior point. By Definition \ref{DefinitionFukayaLikeQuintuples}, $h \circ \basis(q) \neq 0$ only if there exists a triangle in the universal cover of $\mathcal{S}$ bounded by subarcs of lifts $\widetilde{\delta}$, $\widetilde{\gamma}$ and $\widetilde{\tau(\gamma)}$  of $\delta, \gamma$ and $\tau(\gamma)$ in clockwise order. Since $\{\gamma, \tau(\gamma), \delta\}$ are in minimal position, $\widetilde{\delta}$ intersects each of $\widetilde{\gamma}$ and $\widetilde{\tau(\gamma)}$ at most once in the interior. Since $\widetilde{\tau \gamma}$ and $\widetilde{\gamma}$ have neighboring end points we conclude that at least one of the corners of the triangle lies on the boundary. 
	Suppose that two of the corners of such a triangle were boundary intersections, then lifts of $\gamma, \tau(\gamma)$ and $\delta$ would be arranged in one of the ways shown in Figure \ref{FigureExistenceARTriangles}.
	\begin{figure}[H]
		\begin{displaymath}
		\begin{array}{ccc}
		
		{\begin{tikzpicture}
			\draw[->] (-1*1.5,0)--(2*1.5,0);
			\draw [<-](-1*1.5, 2)--(2*1.5, 2);
			
			\draw (0,0)--(1.5, 2);
			\draw (1.5, 0)--(0, 2);
			\draw (0, 2)--(0,0);
			\foreach \i in {0,...,1}
			{
				\filldraw (\i*1.5, 2) circle (2pt);
				\filldraw (\i*1.5, 0) circle (2pt);
			}
			\draw (-0.3,1) node{$\widetilde{\delta}$};
			\draw (1.7,1.4) node{$\widetilde{\tau(\gamma)}$};
			\draw (1.6,0.5) node{$\widetilde{\gamma}$};
			\end{tikzpicture}}
		& {\quad \quad} &
		{\begin{tikzpicture}
			\draw[<-] (1*1.5,0)--(-2*1.5,0);
			\draw[->] (1*1.5, 2)--(-2*1.5, 2);
			
			\draw (0,0)--(-1.5, 2);
			\draw (-1.5, 0)--(0, 2);
			\draw (0, 2)--(0,0);
			\foreach \i in {0,...,1}
			{
				\filldraw (-\i*1.5, 2) circle (2pt);
				\filldraw (-\i*1.5, 0) circle (2pt);
			}
			\draw (0.3,1) node{$\widetilde{\delta}$};
			\draw (-1.6,1.5) node{$\widetilde{\gamma}$};
			\draw (-1.7,0.6) node{$\widetilde{\tau(\gamma)}$};
			\end{tikzpicture}}
		\end{array}
		\end{displaymath}
		\caption{} \label{FigureExistenceARTriangles}
	\end{figure}
	Note that we do not assume that the shown boundary components are distinct. However, in order for $p$ to be interior, all endpoints of the lifts must be pairwise distinct.
	In any case, we observe that the unique intersection of $\widetilde{\tau(\gamma)}$ and $\widetilde{\delta}$ only defines an element in $\tau(\gamma) \overrightarrow{\cap} \delta$ but not in $ \delta\overrightarrow{\cap}\tau(\gamma)$, implying that $h \circ \basis(q)=0$.\\
	In case only one of the corners lies on the boundary, then again, because  every pair of lifts intersects at most once in the interior, it follows that $\delta \simeq \gamma$ or $\delta \simeq \tau(\gamma)$ showing that $\basis(q)$ is invertible as shown in Lemma \ref{LemmaBoundaryPointsAreInvertible}.

	Similarly, if $p \in \partial \mathcal{S}$, then $\gamma$ and $\tau(\gamma)$ are arcs connecting consecutive marked points on the boundary and at all lifts of $p$ to a universal cover of $\mathcal{S}$, there exist lifts of $\gamma$ and $\tau(\gamma)$, which intersect as in Figure \ref{FigureExistenceARTriangles2}.
	\begin{figure}[H]
		\begin{displaymath}
		\begin{tikzpicture}
		\draw[->] (-1,0)--(3,0);
		\foreach \i in {0,...,2}
		{
			\filldraw (\i,0) circle (2pt);
		}
		\foreach \i in {0,1}
		{
			\draw ({\i+1+0.5*sin(0)}, {0}) arc (0:180:0.5);
		}
		
		\end{tikzpicture}
		\end{displaymath}
		\caption{} \label{FigureExistenceARTriangles2}
	\end{figure}
	In particular, if $h \circ \basis(q) \neq 0$, for some $q$ and $\delta$ as above, $q$ is required to be a boundary intersection and certain lifts of $\gamma, \tau(\gamma)$ and $\delta$ are required to form a fork. However, because $\basis(q)$ is a morphism from an object representing $\delta$ to an object representing $\gamma$, $q$ cannot coincide with the unique intersection of the lifts of $\gamma$ and $\tau(\gamma)$, unless $\delta \simeq \gamma$, in which case $\basis(q)$ is invertible according to Lemma \ref{LemmaBoundaryPointsAreInvertible}. In any case, it follows that $h \circ \basis(q)=0$.
\end{proof}
\begin{cor}\label{CorollaryAuslanderReitenTranslationForArcObjects}
	Let $X \in \T$ be an arc object and let $\gamma \in \gamma(X)$. Then, the Auslander-Reiten translate of $X$ exists and $\tau(\gamma) \in \gamma(\tau X)$.
\end{cor}
\noindent In analogy to Corollary 6.3 in \cite{Bobinski}, the previous Corollary enables us to identify the arcs which give rise to Auslander-Reiten triangles with indecomposable middle term:
\begin{cor}\label{CorollaryIndecomposableMiddleTermARTriangle}
	Let $X \in T$ be an arc object and let
	\triangle[X][Y][Z][][][][1]
	be an Auslander-Reiten triangle. Then, $Y$ is indecomposable if and only if $\gamma(X)$ (or, equivalently, $\gamma(Z)$) contains a boundary segment.
\end{cor}
\smallskip

\subsection{The Serre pairing}\label{SectionSerrePairing}
We give a geometric descrition of the Serre pairing.
\noindent Let $\gamma$ be a finite arc on $\mathcal{S}$ and let $p \in \gamma \overrightarrow{\cap} \tau \gamma$ be the distinguished intersection. For an indecomposable object $U \in \T$, the map $h:=\basis(p)$ gives rise to a perfect pairing 

\begin{displaymath}\begin{tikzcd} (-,-)_h: \Hom(X, U) \times \Hom(U,\tau X[1]) \arrow{r} & k, \end{tikzcd}\end{displaymath}

\noindent defined by the formula $(f,g)_h= \eta(g \circ f)$, where $\eta \in  {\Hom(X,\tau X[1])}^{\ast}$ is the projection onto $h$ with respect to the basis associated with all intersections $\gamma \overrightarrow{\cap} \tau(\gamma)$.
\medskip
 
 \noindent The diffeomorphism $\tau: \mathcal{S} \rightarrow \mathcal{S}$ is isotopic to the identity (not relative to the boundary) and, if $\gamma_U \in \gamma(U)$ is a curve in minimal position with $\gamma$, then such an isotopy induces a bijection 
\[\begin{tikzcd}\mathbb{S}:\gamma \overrightarrow{\cap} \gamma_U \arrow{r} & \gamma_U \overrightarrow{\cap} \tau \gamma.\end{tikzcd}\]
  From the geometric description of compositions it follows
 $\left(\basis(q), \basis(\mathbb{S}(q'))\right)_h \neq 0$ for $q,q'  \in \gamma \overrightarrow{\cap} \gamma_U$ if and only if  $q=q'$.

\subsection{The category of perfect objects}
\ \smallskip

\noindent Corollary \ref{CorollaryAuslanderReitenTranslationForArcObjects} motivates the following definition of a perfect subcategory in every surface-like category.
\begin{definition}
	Let $\T$ be a surface-like category. Define the \textbf{category of perfect objects} $\Perf(\T)$ as the full subcategory of $\T$ containing all objects $X \in \T$ such that for all object $Y \in \T$, $\Hom^*(X,Y)$ has finite dimension.

\end{definition}

\noindent We make the following observations: \smallskip

\begin{itemize}
    \setlength\itemsep{0.5em}
    
    \item  It follows from the five lemma that $\Perf(\T)$ is a triangulated subcategory and that triangle equivalences between surface-like categories restrict to their perfect subcategories.
    
    \item The category $\Perf(\T)$ consists of all objects which are finite direct sums of finite arc objects and loop objects. In particular, if $\T=\mathcal{D}^b(A)$ for some gentle algebra $A$, then $\Perf(\T)$ coincides with the usual definition of the category of perfect complexes.

\end{itemize}
\smallskip

\subsection{Fractional Calabi-Yau objects and families of \texorpdfstring{$\tau$}{tau}-invariant objects}\label{SectionFractionalCalabiYauTauInvariantFamilies}

\noindent We characterize fractional Calabi-Yau objects in surface-like categories and recover the geometric interpretation of derived invariants for (graded) gentle algebras in \cite{OpperPlamondonSchroll} and \cite{LekiliPolishchukGentle}. These invariants were introduced in \cite{AvellaAlaminosGeiss} and extended in \cite{LekiliPolishchukGentle}. Our observations lead us to a natural notion families of $\tau$-invariant objects in the spirit of ``families of band modules'' which are indexed by a continuous parameter.
\medskip

\noindent {Recall that an object $X \in \mathcal{T}$ in a Hom-finite triangulated Krull-Schmidt category $\mathcal{T}$ with Serre functor $\mathbb{S}$ is \textbf{$(m,n)$-fractional Calabi-Yau} if and only if $\mathbb{S}^mX \cong X[n]$.} By Theorem I.2.4, \cite{ReitenVanDenBergh}, the existence of a Serre functor of $\mathcal{T}$ is equivalent to the property that $\mathcal{T}$ has  Auslander-Reiten triangles. In particular, $X$ is $(m,n)$-fractional Calabi-Yau if and only if $\tau^mX \cong X[n-m]$, where $\tau=\mathbb{S}[-1]$ is the Auslander-Reiten translation. 
\begin{prp}\label{PropositionFractionallyCalabiYauArcs}
	Let $X \in \T$ be an arc object. The following is true.
	\begin{enumerate}
	\setlength\itemsep{0.5em}
	\item There exist $m,n \in \mathbb{Z}$ such that $\tau^m X \cong X[n]$ if and only if $\gamma(X)$ contains a boundary arc.
	\item  If $\gamma(X)$ contains an arc with image in a component $B \subseteq \partial S$, then $X$ is $(\#(B \cap \marked), -\omega_B)$-fractional Calabi-Yau, where $\omega_B$ denotes the winding number of the simple boundary loop which winds around $B$ in clockwise direction.
	\end{enumerate}
	
\noindent In particular, $\tau X \cong X$ if and only if $\gamma(X)$ contains a boundary arc with image in a component $B$ with a single marked point and $\omega_B=0$. 
\end{prp} 	
\begin{proof}
	Lifting $\gamma \in \gamma(X)$ to the universal cover, we see that $\tau^m\gamma \simeq \gamma$ implies that $\gamma$ is homotopic to a boundary arc with image in a component $B \subseteq \partial S$ and $\tau^m \gamma \simeq \gamma$ if and only if $m$ equals the number of marked points on $B$. Suppose $\tau^m X \cong X [n]$ and let $\gamma \in \gamma(X)$ be a boundary arc contained in a component $B \subset \partial S$. For each $i \in [0,m)$, denote $p_i$  the distinguished intersection in $\tau^i\gamma \overrightarrow{\cap} \tau^{i+1} \gamma$ corresponding to a connecting morphism $\tau^iX \rightarrow\tau^{i+1}X[1]$ of an Auslander-Reiten triangle ending in $\tau^i X$ , see Lemma \ref{PropositionExistenceARTrianglesForArcs}. We obtain a sequence of morphisms
	\[X \rightarrow \tau X[1] \rightarrow \tau^2 X[2] \rightarrow \cdots \rightarrow \tau^m X[m]\cong X[n+m].\]
	The arcs $\tau^i \gamma$ and the intersections $p_0, \dots, p_m$ determine a piecewise smooth loop which is homotopic to a simple loop which winds around $B$ in counter-clockwise direction.
	It follows $-\omega_B=(n+m)- m \cdot 1=n$.
\end{proof}

\noindent The previous lemma shows that shift orbits of fractionally Calabi-Yau objects are encoded in the number of marked points on the boundary components and winding numbers of boundary curves. However, as shown in \cite{AvellaAlaminosGeiss}, this is what the AAG-invariant of a gentle algebra counts. We recover Theorem 6.1, \cite{OpperPlamondonSchroll} and Theorem 3.2.2, \cite{LekiliPolishchukGentle}.
\begin{cor}\label{AGInvariantFukaya}
Let $A$ be a gentle algebra and let $B_1, \dots, B_n$ be the boundary components of $S_A$. Let $n_{B_i}$ be defined as in Lemma \ref{PropositionFractionallyCalabiYauArcs}. The AAG-invariant of $A$ is given by the  set of pairs $(m_i, m_i - \omega_{B_i})$ ($i \in [1,n]$), where 
 $m_i$ is given by the number of marked points on $B_i$ and $\omega_{B_i}$ is defined as in Proposition \ref{PropositionFractionallyCalabiYauArcs}.
\end{cor}

\subsection{Projective and affine coordinate transformations}\label{SectionKroneckerAlgebra}
\ \smallskip

\noindent In the presence of $\tau$-invariant arc objects the topological distinction between arcs and loops is no longer reflected in $\T$ as can be illustrated by the Kronecker quiver and certain canonical algebras of tubular type.
\subsubsection{The Kronecker quiver and the projective linear group}
\ \smallskip

\noindent  Let $A$ denote the Kronecker algebra, i.e. \@ the path algebra of the quiver
	\begin{displaymath}
	\begin{tikzcd}
	x \arrow[bend left]{r}{\alpha} \arrow[bend right, swap]{r}{\beta} & y,
	\end{tikzcd}
	\end{displaymath} 
	 The algebra $A$ is gentle with surface model $(\mathcal{S}_A, \eta_A)$, where $\mathcal{S}_A$ is a cylinder with one marked point on each boundary component and $\eta_A$ is the unique homotopy class of line fields such that $\omega_B=0$ for every boundary component $B$. There exists a natural embedding 
	 \[\begin{tikzcd}\PGL_2(k) \arrow[hookrightarrow]{r} & \Aut(\mathcal{D}^b(A)),\end{tikzcd}\]
	  the elements of the image of which we call  \textbf{ projective coordinate transformations}. Namely, let $M \in \GL_2(k)$ be an invertible matrix and denote by $\sigma=\sigma_M$ the $k$-linear automorphism of $A$ satisfying $\sigma(x)=x$, $\sigma(y)=y$ and
	 
	\[\begin{pmatrix}\sigma(\alpha) \\ \sigma(\beta)\end{pmatrix}= M\begin{pmatrix}\alpha \\ \beta\end{pmatrix}. \] \smallskip
	
\noindent	Denote by $_{\sigma}A$ the $A$-bimodule on the vector space $A$ with its regular right $A$-module structure whose left $A$-module structure is twisted by $\sigma$. By this, we mean that $a \in A$ acts on $x \in {_{\sigma}A}$ by $a.x\coloneqq  \sigma(a) \cdot x$, where $\cdot$ denotes the multiplication of $A$. The derived functor of $_{\sigma}A \otimes -$ defines an element in $\Aut(\mathcal{D}^b(A))$ and two of such functors associated with $M$ and $M'$ are isomorphic if and only if $M$ and $M'$ define the same element in $\PGL_2(k)$.\medskip

	\noindent Consider the family $\{B_{\lambda, \mu} \, | \, (\lambda, \mu) \neq (0,0)\}$ of $\tau$-invariant $A$-modules
	\begin{displaymath}
	\begin{tikzcd}[scale=1.5]
	 k \arrow[bend left]{r}{\lambda} \arrow[bend right, swap]{r}{\mu} & k,
	\end{tikzcd}
	\end{displaymath}
	which are quasi-isomorphic to the complexes
	
	\begin{displaymath}
	\begin{tikzcd}
	\cdots \arrow{r} & 0 \arrow{r} & P_{y} \arrow{rr}{\lambda \alpha + \mu \beta} && P_{x} \arrow{r}  & 0 \arrow{r} & \cdots,
	\end{tikzcd}
	\end{displaymath}	
	
	The family of objects $B_{\lambda, \mu}$ corresponds to the skyscraper sheaf of the point $(\lambda : \mu)$ under Beilinsons's equivalence $\mathcal{D}^b(\Coh \mathbb{P}^1) \cong \mathcal{D}^b(A)$ and the $\PGL_2(k)$-action corresponds to  the action of $\Aut(\mathbb{P}^1)$.\smallskip
		
\noindent The object $B_{\lambda, \mu}$ is a loop object if and only if $\lambda, \mu \neq 0$ and in that case is represented by the unique loop equipped with a $1$-dimensional indecomposable local system. The arc objects $B_{0, \mu}$ and $B_{\lambda, 0}$ are represented by the boundary segments. It is not difficult to see that $\PGL_2(k)$ acts transitively on the isomorphism classes of $B_{\lambda, \mu}$ in this way.\medskip
\subsubsection{Canonical algebras and affine transformations}
\ \smallskip

\noindent In a similar vein, one defines algebra automorphisms for certain canonical algebras. For every gentle quiver $(Q,I)$ a coordinate transformation may be defined for every subquiver $Q'$ of $Q$ of the form.
	
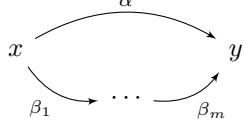
\begin{displaymath}
	\begin{tikzcd}[row sep=1ex]
	x \arrow[bend left]{rr}{\alpha} \arrow[bend right, swap]{dr}{\beta_1} & & y \\ & \cdots \arrow[bend right, swap]{ur}{\beta_m}
	\end{tikzcd}
	\end{displaymath} 
where $\alpha$ is a maximal antipath in $Q$ and $\beta=\beta_m \cdots \beta_1$ is a maximal admissible path in $Q$. By a \textbf{maximal antipath} we mean that if $\delta \in Q_1$, then $\delta \alpha \not \in I$ (resp. $\alpha \delta \not \in I$) whenever the composition of arrows is defined.

\noindent For any such subquiver $Q'$ and every tuple $\Lambda=(\lambda_{\alpha}, \lambda_{\beta}, \lambda_1, \dots, \lambda_m) \in (k^{\times})^{m+1}$, one defines an algebra automorphism $\sigma=\sigma_{\Lambda}$ by requiring that
\[ \sigma(x) = \begin{cases} \lambda_{\alpha} \alpha + \lambda_{\beta} \beta, & \text{if }x=\alpha; \\ \lambda_i \beta_i & \text{if }x=\beta_i; \\ x & \text{otherwise.} \end{cases}\]

\noindent We call $\sigma$ an \textbf{affine coordinate transformation}. The complexes $P_{\lambda \alpha + \mu \beta}^{\bullet}$, given by
\begin{displaymath}
	\begin{tikzcd}
	\cdots \arrow{r} & 0 \arrow{r} & P_{y} \arrow{rr}{\lambda \alpha + \mu \beta} && P_{x} \arrow{r}  & 0 \arrow{r} & \cdots,
	\end{tikzcd}
	\end{displaymath}	
	where $\beta=\beta_m \cdots \beta_1$,
	are $\tau$-invariant band complexes for all pairs $(\lambda, \mu) \in k^{\times} \times k$ but not for $\lambda=0$. Moreover, the complexes of $(\lambda, \mu)$ and $(\lambda', \mu')$ are isomorphic if and only if $\lambda^{-1} \mu= \lambda^{\prime-1} \mu'$. \bigskip

\subsubsection{Families of $\tau$-invariant objects}\label{SectionFamiliesTauInvariantObjects} 
\ \smallskip

\noindent As illustrated by the Kronecker algebra in the previous section, it may happen that an equivalence of surface-like categories maps loop objects to ``degenerations''  of such objects and, as a consequence, whether a $\tau$-invariant object is represented by an arc or a loop can not be distinguished within the triangulated category. However, Corollary \ref{CorollaryEquivalencesPreserveFamilies} below shows that every triangle equivalence between surface-like categories maps a $\tau$-invariant family consisting of  loop objects and their degenerations in a coherent way.\medskip

\noindent Our observations in Section \ref{SectionKroneckerAlgebra} motivate the following definitions. 
\begin{definition}\label{DefinitionEquivalenceRelationDegenerationsOfLoops}Let $(\T, \mathcal{S}, \eta, \gamma, \basis)$ be a surface-like quintuple and $\mathcal{S}=(S, \marked)$. The equivalence relation $\sim_{\ast}$ is defined as the relation which identifies every gradable boundary loop around a boundary component $B$ with a single marked point with unique segment of $B$. Define $\simeq_{\ast}$ as the equivalence relation generated by the homotopy relation and $\sim_{\ast}$.
\end{definition}
\noindent Given an object $X \in \T$, we denote by $\gamma_{\ast}(X)$ its $\simeq_{\ast}$-equivalence class.

\begin{definition}\label{DefinitionFamilyTauInvariantObjects}
	A \textbf{family of $\tau$-invariant objects} in $\T$ is a collection $\mathcal{X}=\{X_i\}_{i\in I}$ of arbitrary indecomposable $\tau$-invariant objects in $\T$, which sit at the mouth of their  homogeneous tubes, such that $X_i \not \cong X_j$ in $\T/[1]$ for all $i \neq j$ and there exists a $\simeq_{\ast}$-equivalence class $\gamma_{\ast}(\mathcal{X})$ with the following property:\smallskip
	
\begin{changemargin}{0.5cm}{0.5cm} 
\textit{If $X \in \T$ is indecomposable and sits at the mouth of a homogeneous tube of the Auslander-Reiten quiver, then a shift of $X$ is isomorphic to an object in $\mathcal{X}$ if and only if $\gamma_{\ast}(X)=\gamma_{\ast}(\mathcal{X})$.}
\end{changemargin}

\noindent	An arc object in $\mathcal{X}$ is called a \textbf{degenerated object} of the family. A family  $\mathcal{X}$ of $\tau$-invariant objects is further called \ \smallskip
	\begingroup
	\begin{itemize}
	\setlength\itemsep{0.5em}
		\item \textbf{$k^{\times}$-family}, if $\gamma_{\ast}(\mathcal{X})$ contains no degenerated object;
		\item \textbf{$\mathbb{A}^1_k$-family}, if $\gamma_{\ast}(\mathcal{X})$ contains precisely one degenerated object;
		\item \textbf{$\mathbb{P}^1_k$-family}, if $\gamma_{\ast}(\mathcal{X})$ contains  two degenerated objects.\ \smallskip
	\end{itemize} 
	\endgroup

\end{definition}

\noindent Since every $\simeq_{\ast}$-class contains at most two homotopy classes of arcs it follows that every family of $\tau$-invariant objects belongs to one of the three types above.

\begin{exa}

We have seen in Section  \ref{SectionFamiliesTauInvariantObjects}  that the modules $(B_{\lambda,\mu})_{[\lambda:\mu] \in \mathbb{P}^1_k}$ form a $\mathbb{P}^1_k$-family of of objects in case of the Kronecker algebra. In the general case of canonical algebras, the complexes $(P_{\alpha + \lambda \beta}^{\bullet})_{\lambda \in \mathbb{A}^1_k}$  form an $\mathbb{A}^1_k$-family of $\tau$-invariant objects.  
\end{exa}

\noindent  The following lemma shows that the presence of $\mathbb{P}^1_k$-families is a rare phenomenon.

\begin{lem}\label{LemmaP1FamiliesKroneckerOnly} Let $\T$ be a surface-like category with surface model $(\mathcal{S}, \eta)$.
Then $\T$ contains a $\mathbb{P}^1_k$-family of $\tau$-invariant objects if and only if $(S, \marked)$ is a cylinder with a single marked point on each boundary component and $\omega=0$. There exists at most one such family (up to isomorphisms of objects).
\end{lem}
\begin{proof}
	Note that the cylinder is the only compact surface with two distinct but homotopic boundary components meaning that the corresponding simple boundary curves (or their inverses) are homotopic. Thus, the assertion follows from Proposition \ref{PropositionFractionallyCalabiYauArcs}.
\end{proof}

\noindent It is not difficult to conclude from the previous lemma that if $A$ is a gentle algebra, then $\mathcal{D}^b(A)$ contains a $\mathbb{P}^1$-family if and only if $A$ is the Kronecker algebra.\bigskip

\subsection{Boundary objects and segment objects}\label{SectionBoundarySegmentObjects}

The assertion of Proposition \ref{PropositionFractionallyCalabiYauArcs} suggests the following definition.
\begin{definition}\label{DefinitionBoundarySegmentObject}
	Let $X \in \Perf(\T)$  be an arbitrary indecomposable. We say that $X$ is a \textbf{boundary object} if the following hold true:
	\begin{itemize}
	\setlength\itemsep{1ex}
		\item[1)] $X$ is fractional Calabi-Yau.
		\item[2)] If $\tau X\cong X$, then for all indecomposable objects $Y \in \T$, 
		\[\max\{\dim \Hom^*(X, Y), \dim \Hom^*(Y,X)\} \leq 2.\]
		\item[3)] There exists a non-zero morphism between $X$ and a perfect object which is not $\tau$-invariant.
	\end{itemize}	
	A boundary object $X$ is further called a \textbf{segment object} if the middle term of any Auslander-Reiten triangle starting in $X$ is indecomposable.
\end{definition}
\noindent It follows from its definition that the image of a boundary (resp. segment) object under a triangle equivalence is again a boundary (resp. segment) object .
The following follows from Property 2 V) of Definition \ref{DefinitionFukayaLikeQuintuples}.
\begin{lem}
A loop object which is a boundary object sits at the mouth of its homogeneous tube and its associated local systems are $1$-dimensional.
\end{lem}

 \noindent The following proposition characterizes boundary objects and segment objects.
\begin{prp}\label{PropositionCharacterizationBoundaryObjects}		
	Let $X \in \Perf(\T)$ be an arbitrary indecomposable. Then,
	\begin{enumerate}	\setlength\itemsep{1ex}
		\item $X$ is a boundary object if and only if it is represented by a simple boundary arc or boundary loop with $1$-dimensional local system.
		\item $X$ is a segment object if and only if it is represented by a boundary segment or a boundary loop with $1$-dimensional local system.
	\end{enumerate} 
\end{prp}
\begin{proof}
The third condition of Definition \ref{DefinitionBoundarySegmentObject} is equivalent to the assumption that $\gamma_X$ intersects a finite arc and hence equivalent to the assumption that $\gamma_X$ is not a loop around a puncture. 
Due to Proposition \ref{PropositionFractionallyCalabiYauArcs}, simplicity of the representing arc of an arc object is equivalent to condition 2) of Definition \ref{DefinitionBoundarySegmentObject}. By Corollary \ref{CorollaryIndecomposableMiddleTermARTriangle} the condition that the middle term of an Auslander-Reiten triangle starting (or equivalently, ending) in $X$ is indecomposable, translates to the property that the end points of an arc $\gamma$ representing $X$ have to be consecutive elements. In other words $\gamma$ is a boundary segment. We therefore assume that $X$ is represented by a loop $\gamma$ and hence is $\tau$-invariant. 

If $\gamma$ is homotopic to a boundary loop, it does not intersect any other  loop which is in minimal position with $\gamma$. Furthermore, it intersects arcs only in a neighborhood of their end points and hence at most twice. If the local system associated to $X$ has dimension $1$, then each intersection determines a single basis element $\Hom^*(X,Z)$ (resp.\@ $\Hom^*(Z,X)$). Finally, simplicity of $\gamma$ implies that $\dim \Hom^*(X,X) \leq 2$.

	Conversely,  let $\gamma$ be a gradable loop and  $U \in \Perf(\T)$ be a representative of $\gamma$. In particular, $U$ is $\tau$-invariant. Assume that $U$ satisfies condition 2) and therefore that $\dim \Hom^*(U,U) \leq 2$. It follows from Definition \ref{DefinitionFukayaLikeQuintuples}, 2\,ii) that the local system associated to $U$ is $1$-dimensional and that $\gamma$ can be chosen to be simple. Suppose $\gamma$ is not homotopic to a boundary curve. 
	Then it follows from Lemma \ref{LemmaTooManyIntersectionsNonBoundary} below that for each $m \in \mathbb{N}$, there exists an arc $\delta$ such that $\iota(\delta, \gamma) \geq m$ -- a contradiction to 2).
\end{proof}

\begin{lem}\label{LemmaTooManyIntersectionsNonBoundary}
Let $\gamma$ be a simple loop on a marked surface  $S$ which is not a boundary curve and let $m \in \mathbb{N}$. Then, there exists an arc $\delta$ such that $\iota(\delta, \gamma) \geq m$, where $\iota(-,-)$ denotes the geometric intersection number. 
\end{lem}
	\begin{proof}
	We cut $S$ along $\gamma$ and denote by $S'$ the resulting surface, i.e.\@ $S'$ contains distinct boundary components $B_1,B_2$, such that $S' \setminus \left(B_1 \cup B_2\right)=S\setminus \gamma$.
	We may assume that the connected component $V$ of $S'$ which contains  $B_1$ on its boundary contains a marked point of $S$ which is not a puncture. For now, we assume that $S'$ is not connected.  None of the components of $S'$ is homeomorphic to a disc or a cylinder as otherwise $\gamma$ would be contractible or a boundary curve. There exist $m$ disjoint simple non-separating arcs $E_1^1, \dots, E_1^m$ in $S'$ connecting pairwise distinct points in $B_1$, such that neither of them is homotopic to a boundary arc. For example, we can choose homotopic but disjoint simple arcs $E_1^2, \dots, E_1^m$. In the same way, we may choose disjoint simple arcs $E_2^1, \dots, E_2^m$ such that  for all $i \in (1,m)$, $E_2^i$ starts on the endpoint of $E_1^i$ and ends on the start point of $E_1^{i+1}$ when regarded as paths in $S$. Furthermore, we choose disjoint simple arcs $W_1, W_2$ in $V$ and require that 
	\begin{itemize}
	\setlength\itemsep{1ex}
		\item for all $i \in [1,m]$, $W_1 \cap E_1^i= \emptyset$;
		\item the arcs $W_1$ and $W_2$ have no interior intersections;
		\item the end point (resp.\@ start point) of $W_1$ (resp.\@ $W_2$) agrees with the start point (resp.\@ end point) of $E_2^1$ (resp.\@ $E_2^m$);
		\item $s(W_1)$ and $t(W_2)$ are marked boundary points.
	\end{itemize}
	The concatenation $\delta\coloneqq W_2 \ast E_2^m \ast E_1^m \ast \cdots \ast E_2^1 \ast W_1$, as indicated by the symbol $\ast$, is a simple arc in $S$ and $|\gamma \cap \delta|=m$.  Suppose $\{\gamma, \delta\}$ is not in minimal position, then by \cite{FarbMargalit}, Proposition 1.7, $\gamma$ and $\delta$ form a bigon. By the assumptions we made on the $E_j^i$, such a bigon must be formed by a subarc $\gamma'$ of $\gamma$ and an arc $E_{b}^a$ for some $a \in [1,m]$ and some $b \in [1,2]$. But this implies that $E_b^a$ is homotopic to  $\delta'$ which contradicts our assumptions on $E_{b}^a$.
	
The same proof works if $S'$ is connected with a difference in case $S$ is a torus with one boundary component. Namely, in this case, $S'$ has genus $0$ and $3$ boundary components and arcs $E_1^i$ and $E_2^j$ as above will intersect in any case so that the bigon argument has to be replaced by \cite{Thurston}, Lemma 10, to verify that the arc $\delta$ we obtain in the same way as before is in minimal position with $\gamma$.
\end{proof}

\medskip

\subsection{Semi-perfect objects}
\ \smallskip
\noindent Segment objects allow us to characterize objects which correpond to semi-finite arcs. In what follows, let $(\T, \mathcal{S}, \eta, \gamma, \basis)$ denote  a surface-like quintuple.
\begin{definition}\label{DefinitionSemiPerfect}
Let  An indecomposable object $X \in \T \setminus \Perf(\T)$ is called \textbf{semi-perfect} if there exists a non-zero morphism between $X$ and a segment object in $\T$.
\end{definition}
Since triangle equivalences map segment objects to segment objects, the property of being semi-perfect is preserved under triangle equivalences betweeen surface-like categories.
\begin{lem}
A indecomposable object $X \in \T$ is semi-perfect if and only if $\gamma(X)$ is semi-finite. 
\end{lem}
\begin{proof}
If an arc $\gamma$ on $\mathcal{S}$ which has an end point $p \in \partial \mathcal{S}$, then $p$ determines a non-zero a morphism between $X$ and any segment object corresponding to a boundary segment with end point $p$. On the other hand an arc which starts and ends at a puncture has no intersections with a boundary curve if they are in minimal position.
\end{proof}

\medskip

\subsection{Interior morphisms and boundary morphisms}\label{SectionSpaceInteriorMorphisms}

\noindent Building on the notion of segment objects, we explore how the distinction between boundary and interior of  a surface can be re-interpreted algebraically as a dichotomy of two classes of morphisms yielding distinguished subspaces of  ``interior''  morphisms between objects in surface-like categories. We recall that, unless stated otherwise, by an indecomposable object we mean an arc object or loop object associated with a $1$-dimensional local system.
\begin{definition}
	Let $X,Y \in \T$ be indecomposable objects each of which is not $\tau$-invariant or sits at the mouth  of a homogeneous tube. A morphism  $f: X \rightarrow Y[i]$ is called  \textbf{interior}, if $X$ or $Y$ is contained in the tube of a $\tau$-invariant boundary object, or, both of the following conditions are satisfied.

	\begin{itemize}
	\setlength\itemsep{1ex}
		\item[i)] The connecting morphism of any Auslander-Reiten triangle ending in a segment object $U \in \Perf(\T)$ does not factor through $f$.
		\item[ii)] The morphism $f$ factors through an object in $\Perf(\T)$.
	\end{itemize}
	If $f$ is not an interior morphism, we call it a \textbf{boundary morphism}.
\end{definition}
\begin{notation}
	For indecomposable objects $X_1, X_2 \in \T$, let $\Hom_{\Int}(X_1, X_2)$ denote the subset of interior morphisms and we set \[\Hom_{\Int}^*(X_1, X_2)\coloneqq \bigoplus_{n \in \mathbb{Z}}{\Hom_{\Int}(X_1, X_2[n])}.\]   	
\end{notation}

\noindent Note that being an interior morphism is invariant under triangle equivalences in the following sense.
\begin{lem}\label{LemmaEquivalencesPreserveInteriorMorphisms}
	Let  $\T, \T'$ be surface-like categories and let  $T: \T \rightarrow \T'$ be a triangle equivalence. Then, the isomorphisms
	\[\begin{tikzcd}\Hom_{\T}(-,-) \arrow{r} & \Hom_{\T'}(T(-),T(-))\end{tikzcd}\] 
	restrict to bijections between the subsets of interior morphisms.
\end{lem} 

\noindent Our main result on interior morphisms is the following and provides justification for the term ``interior''.
\begin{prp}\label{PropositionInteriorMorphismsVectorSpace}
	Let $X, Y \in \T$ be indecomposable. Then, $\Hom_{\Int}(X, Y)$ is a subvector space of $\Hom(X, Y)$. Moreover,
	\begin{enumerate}
			\setlength\itemsep{1ex}
	\item If $X$ or $Y$ is $\tau$-invariant, then 
	\[\Hom_{\Int}(X,Y)=\Hom(X,Y).\]
	\item If neither $X$ nor $Y$ are $\tau$-invariant, then $\Hom_{\Int}(X,Y)$ is generated as a vector space by the morphisms of interior intersections between representatives of $X$ and $Y$.
	\end{enumerate}
\end{prp}

\noindent In order to  prove Proposition \ref{PropositionInteriorMorphismsVectorSpace}, we need to characterize morphisms which factor through $\Perf(\T)$. For objects $X, Y \in \T$, let $\Hom(X,Y)_{\Perf(\T)}$ denote the set of morphisms $X \rightarrow Y$ which factor as a composition $X \rightarrow P \rightarrow Y$, with $P \in \Perf(\T)$. Since $\Perf(\T)$ is closed under direct summands and finite direct sums, this set is a subspace of $\Hom(X,Y)$. 
\begin{lem}\label{LemmaMorphismFactorsThroughPerfIffPuncture}
	Let $X, Y \in \T \setminus \Perf(\T)$ be indecomposable and let $f \in \Hom(X,Y)$. Then, $f \in \Hom(X,Y)_{\Perf(\T)}$ if and only if
	\begin{itemize}
			\setlength\itemsep{1ex}
		\item $X$ is not isomorphic to a shift of $Y$ and $\displaystyle f$ is a sum of morphisms associated with intersections different from punctures, or
		\item $X$ is isomorphic to a shift of $Y$ and $\displaystyle f$ is a sum of morphisms associated with intersections different from endpoints. 
	\end{itemize}
\end{lem}
\begin{proof}Let $\gamma_X, \gamma_Y$ be representatives of $X$ and $Y$ in minimal position and let $f=\sum_{i=1}^m{f_i}$ be a standard decomposition into basis elements. Denote by $p_i \in \gamma_X \overrightarrow{\cap} \gamma_Y$ the corresponding intersection of $f_i$. Assume that $p_0$ is a puncture and suppose there exists $Z \in\Perf(\T)$ and morphisms $g:~X~\rightarrow~Z$ and $h: Z \rightarrow Y$, such that $h \circ g=f$. 
	
It is sufficient to show that for any two morphisms $g': X \rightarrow Z'$ and $h': Z' \rightarrow Y$ with $Z' \in \Perf(\T)$ indecomposable, $\basis(p_0)$ cannot occur as a summand of $h' \circ g'$ if written as a linear combination of basis vectors arising from intersections.  However, since $Z'$ is represented by a finite arc or loop, $g'$ and $h'$ are sums of morphisms associated to intersections which are not punctures. By Definition \ref{DefinitionFukayaLikeQuintuples} 4), a composition of such basis elements is a sum of morphisms associated to intersections, which are not punctures.

	Next, assume that none of the $p_i$ is a puncture. Since $\Perf(\T)$ is closed under finite direct sums, it is sufficient to show that each $\basis(p_i)$ factors through $\Perf(\T)$. Let $\widetilde{\gamma}_X, \widetilde{\gamma}_Y$ be arcs on the universal cover $\widetilde{S}$ of $\mathcal{S}$ representing $X$ and $Y$, such that $\widetilde{\gamma}_X$ and $\widetilde{\gamma}_Y$ intersect in a lift of $p_i$. Let $p_X$ (resp.\@ $p_Y$) be an interior end point of $\widetilde{\gamma}_X$ (resp.\@ $\widetilde{\gamma}_Y$), i.e.\@ $p_X$ and $p_Y$ are punctures.
	In case $p_i$ lies in the interior, denote by $\gamma$ a simple finite arc in $\widetilde{S}$ which crosses $\widetilde{\gamma}_X$ and $\widetilde{\gamma}_Y$ both once. By construction, $p_i$ and the unique interior intersections $q_X$ and $q_Y$ of $\widetilde{\gamma}$ with $\widetilde{\gamma}_X$ and $\widetilde{\gamma}_Y$ form a triangle. Hence $\widetilde{\gamma}$, $p_X$ and $p_Y$  give rise to an object $Z$ and morphisms $\basis(q_X): X \rightarrow Z$ and $\basis(q_Y): Z \rightarrow Y$, such that $\basis(q_Y) \circ \basis(q_X)$ is a multiple of $\basis(p_i)$.
	
	If $p_i$ lies on the boundary, we choose lifts $\widetilde{\gamma}_X$ and $\widetilde{\gamma}_Y$ as above, which intersect exactly once in a lift $q$ of $p$ and choose $\widetilde{\gamma}$ to be a finite arc which intersects $\widetilde{\gamma}_X$ and $\widetilde{\gamma}_Y$ in its end point $q$ and which lies between $\widetilde{\gamma}_X$ and $\widetilde{\gamma}_Y$ locally around $q$. Then $f$ factors through the morphisms determined by $q$.\end{proof}

\begin{rem}
Again, one direction of the equivalence in Lemma \ref{LemmaMorphismFactorsThroughPerfIffPuncture} becomes obvious in case of gentle algebras as a morphism associated with an intersection $p$
is supported in finitely many homological degrees if and only if $p$ is not a puncture. \end{rem}
\begin{proof}[Proof of Proposition \ref{PropositionInteriorMorphismsVectorSpace}]

	If $X$ or $Y$ is contained in the tube of a $\tau$-invariant boundary object, then by definition $f$ is interior. We therefore assume that neither of them is contained in such a tube.\\
	It follows from Lemma \ref{LemmaMorphismFactorsThroughPerfIffPuncture} that $f$ factors through $\Perf(\T)$ if and  only if none of the $p_i$ is a puncture.
	Set $f_i\coloneqq  \lambda_i \cdot \basis(p_i)$. Let $U$ be a segment object, $\gamma_U$ a representing boundary segment or a boundary loop. Let $h: U \rightarrow \tau U[1]$ be non-zero. As $\Hom(U, \tau U[1])$ is $1$-dimensional, $h$ is unique up to a scalar and a connecting morphism in an Auslander-Reiten triangle. Suppose $h$ factors through $f$. 
	Let $i \in [1,m]$ be such that $f_i$ is represented by an interior intersection.\newline
	We distinguish two cases. First, assume that $U$ is an arc object. Then it follows from Property 4) in Definition \ref{DefinitionFukayaLikeQuintuples} that each composition $f_i \circ \alpha$ and $\beta \circ f_i$, with $\alpha:U \rightarrow X$, $\beta: Y \rightarrow \tau U[1]$ being morphisms associated to intersections, is either zero or a sum of morphisms associated with interior intersections. The same applies to $\beta \circ f_i \circ \alpha$. But a multiple of a morphism in $\Hom(U, \tau U[1])$ is represented by the distinguished  (boundary) intersection in $\gamma_U \overrightarrow{\cap} \tau \gamma_U$, proving $\beta \circ f_i \circ \alpha=0$. Hence, if $h$ factors through $f$ as a map $\beta \circ f \circ \alpha$, then  $h$ equals the sum of all $\beta \circ f_j \circ \alpha$, such that $p_j \not \in \partial \mathcal{S}$.\\
	Next, suppose that $U$ is a loop object and denote by $B$ the boundary component which contains $\gamma_U$. None of the  morphisms in $\Hom(U,U[1])$ is represented by an intersection.
	Thus, for $\beta, \alpha$ as before, $\beta \circ f_i \circ \alpha$ is zero, unless $\beta \circ f_i$ and $\alpha$ correspond to dual interior intersections (see Definition \ref{DefinitionFukayaLikeQuintuples} 4 ii)) and in this case, $p_i$ is the unique marked point on $B$.
	This shows that $f$ is interior if all $p_j$ are interior.\\
	Conversely, suppose that for some $i \in [1, m]$, $p_i$ lies on a component $B \subseteq \partial S$. By assumption, $X$ and $Y$ are not $\tau$-invariant arc objects as such objects are boundary objects by Proposition \ref{PropositionFractionallyCalabiYauArcs}.
	We show that $f$ is a boundary morphism. Let  $\delta$ be a simple boundary arc with image in $B$, connecting $p_i$ and its successor, and let $U \in \T$ be a representative of $\delta$. By Definition \ref{DefinitionFukayaLikeQuintuples} 4 i), $p_i$ encodes morphisms $\alpha: U[m] \rightarrow X$ and $\beta: Y \rightarrow \tau U[m+1]$ for some $m \in \mathbb{Z}$ and $\beta \circ f_i \circ \alpha$ is a multiple of $\basis(q)$, where $q \in \delta \overrightarrow{\cap} \tau \delta$ denotes the distinguished intersection.  Finally, we observe that for all $p_j$ with $j \neq i$, $\beta \circ f_j \circ \alpha=0$ as none of the corresponding intersections form a fork but $\Hom^*(U,U)=\Hom(U,U[1])$ is represented by a boundary intersection. Hence, $\beta \circ f \circ \alpha$ is a multiple of $\basis(q)$.
	
	We have seen that if $X$ or $Y$ is $\tau$-invariant, then every morphism $X \rightarrow Y$ is interior. Otherwise, $\Hom_{\Int}(X,Y)$ is spanned by all morphisms associated with interior intersections of representing curves of $X$ and $Y$.
\end{proof}
\smallskip

\section[{D}iffeomorphisms induced by equivalences]{Diffeomorphisms induced by equivalences}\label{SectiondiffeomorphismInducedByEquivalences}
\smallbreak \noindent This section is mainly concerned with the construction of homeomorpisms from triangle equivalences between surface-like categories. It is based on the study of simplicial automorphisms of the  arc complex of a surface and their connection to diffeomorphisms of the surface as studied in \cite{Disarlo}. 
Throughout this section, let $\mathcal{S}=(S, \marked)$ and $\mathcal{S}'=(S', \marked')$ denote marked surfaces.\
 
\subsection{Arc complexes and essential objects}\label{SectionArcComplexesEssentialObjects}
\ \smallskip

\noindent To begin with, we recall the definition of a simplicial complex.
\begin{definition}
An (abstract) \textbf{simplicial complex} is a set $V$ and a collection $K$ of finite subsets of $V$ such that for all $v \in V$, $\{v\} \in K$ and such that $K$ is closed under taking subsets. An $m$-simplex of $K$ is a set $M \in K$ with $m+1$ elements. \ \medskip

 \noindent The $0$-simplices (which we refer to as \textbf{vertices}) are nothing but the sets $\{v\}$. The $1$-simplices are also called \textbf{edges}.
\end{definition}
\noindent Next, we present the definition of an arc complex and its vertices \textit{essential arcs}.
\begin{definition}\label{DefinitionArcComplex}
	An arc on $\mathcal{S}$ is called \textbf{essential} if it is simple, i.e.\@ has no interior self-intersections, and it is not homotopic to a boundary segment connecting consecutive marked points. The \textbf{arc complex} $A_{\ast}=A_{\ast}(\mathcal{S})$ of $\mathcal{S}$ is the cell complex defined as follows. The vertices of $A_{\ast}$ are in one-to-one correspondence with homotopy classes of essential arcs on $\mathcal{S}$. A set of distinct vertices $\{v_0, \dots, v_{n}\}$ of $A_{\ast}$ is contained in an $n$-simplex if and only if there exists a set of simple representatives $\gamma_i$ of $v_i$, such that $\gamma_i$ and $\gamma_j$ have no interior intersections for all $i \neq j$.
\end{definition}
\begin{exa}Let $S=[0,1]\times S^1$ be a cylinder and $\marked=\{(0,1), (1,1)\}$. Denote by $\gamma_m: S^1 \rightarrow S$ the arc which runs around the cylinder $m$ times in clockwise direction, i.e. $\gamma_m( e^{\pi i \varphi})=(\varphi, e^{-m \pi i\varphi})$ for all $\varphi \in [0,1]$.  Then, the arcs $\gamma_m$ are the vertices of $A_{\ast}(S, \marked)$ and $\gamma_m$ and $\gamma_n$ are connected by an edge if and only if $|m-n|=1$. This arc complex does not contain simplices of higher dimension.
\end{exa}

\begin{rem}\label{RemarkFlagComplex}
\begin{itemize}
\setlength\itemsep{1ex}
\item[]

\item The arc complex of $\mathcal{S}$ is empty if and only if $S$ is a disc and $\marked$ contains at most $3$ marked points and no punctures.
\item Two essential arcs are homotopic if and only if they are isotopic, i.e.\@ there exists a continuous family of diffeomorphisms, which restricts to a homotopy of the given arcs. We refer to \cite{FarbMargalit} for a proof of the case of loops without self-intersections.
\item  It follows from \cite{Epstein1966}, Theorem 3.1 that $A_{\ast}(\mathcal{S})$ is  a \textit{flag complex}, i.e.\@ $n+1$ vertices of $A_{\ast}(\mathcal{S})$ constitute an $n$-simplex if and only if they are pairwise connected by an edge. Thus, $A_{\ast}(\mathcal{S})$ is determined by its set of vertices and edges.
\end{itemize}
\end{rem}
\noindent In order to deal with exceptional cases, it will be useful to consider the \textbf{extended arc complex} $\overline{A}_{\ast}(\mathcal{S})$ which we define as the flag complex with vertices given by the homotopy classes of all simple arcs. As before two homotopy classes in $\overline{A}_{\ast}(\mathcal{S})$ are connected by an edge if and only if they can be realized simultaneously by simple arcs with disjoint interiors. 
\medskip

\subsection{Isomorphisms of arc complexes and diffeomorphisms}\label{SectionIsomorphismsArcComplexesdiffeomorphisms}
\noindent If $A_{\ast}(\mathcal{S})$ is non-empty, then diffeomorphisms of $\mathcal{S}$ act on $A_{\ast}(\mathcal{S})$ and $\overline{A}_{\ast}(\mathcal{S})$ by simplicial automorphisms. By a \textbf{diffeomorphism of marked surfaces} $F:\mathcal{S} \rightarrow \mathcal{S}'$, we mean a diffeomorphism $F: S \rightarrow S'$ which restricts to a bijection $\marked \rightarrow \marked'$.  A \textbf{simplicial isomorphism} between simplicial complexes $A$ and $A'$ is a bijection between their sets of vertices which induces a bijection between their sets of simplices.

 Given any diffeomorphism $F:\mathcal{S} \rightarrow \mathcal{S}'$, it induces a bijection between homotopy classes of essential arcs on $\mathcal{S}$ and essential  arcs on $\mathcal{S}'$. By Remark \ref{RemarkFlagComplex}, it admits a unique extension to a simplicial isomorphism $A_{\ast}(F):A_{\ast}(\mathcal{S}) \rightarrow A_{\ast}(\mathcal{S}')$. In a similar way, $F$ induces a simplicial isomorphism $\overline{A}_{\ast}(F):\overline{A}_{\ast}(\mathcal{S}) \rightarrow \overline{A}_{\ast}(\mathcal{S}')$. It is clear that $A_{\ast}(F)$ and $\overline{A}_{\ast}(F)$ do not change if we replace $F$ by an isotopic diffeomorphism.\bigskip
 
\noindent Let $\Diffeo(\mathcal{S}, \mathcal{S}')$   denote the equivalence classes of  diffeomorphisms $\mathcal{S} \rightarrow \mathcal{S}'$ modulo isotopies in $\mathcal{S}$ and $\mathcal{S'}$, which fix marked points. By $\Diffeo^+(\mathcal{S}, \mathcal{S}')$ we mean the subset of orientation preserving diffeomorphisms $\Diffeo(\mathcal{S}, \mathcal{S}')$.  For two simplicial complexes $C,C'$ denote by $\Simp(C,C')$ the set of simplicial isomorphisms. The assignments $F \mapsto A_{\ast}(F)$  and $F \mapsto \overline{A}_{\ast}(F)$ determine maps 

\[\begin{tikzcd}\Phi:\Diffeo(\mathcal{S}, \mathcal{S}') \arrow{r} & \Simp(A_{\ast}(\mathcal{S}), A_{\ast}(\mathcal{S}')),\end{tikzcd}\]
and
\[\begin{tikzcd}\overline{\Phi}:\Diffeo(\mathcal{S}, \mathcal{S}') \arrow{r} & \Simp(\overline{A}_{\ast}(\mathcal{S}), \overline{A}_{\ast}(\mathcal{S}')),\end{tikzcd}\]
which are homomorphism of groups if $\mathcal{S}=\mathcal{S}'$ with multiplication being given by composition. If $\mathcal{S}=\mathcal{S}'$, the subgroup $\MCG(\mathcal{S})=\Diffeo^+(\mathcal{S},\mathcal{S})$ is called the \textbf{mapping class group} of $\mathcal{S}$.
\bigskip

\noindent A natural question to ask is whether $\Phi$ is surjective or injective. An answer is given in \cite{Disarlo}. In order to state it and for later reference, it is convenient to introduce the following definition.
\begin{definition}\label{DefinitionSpecialSurface}
	We call a marked surface $\mathcal{S}$  \textbf{special} if its arc complex is empty or is of dimension at most $1$.
\end{definition}

\begin{thm}[Theorem 1.1 \& Theorem 1.2, \cite{Disarlo}]\label{TheoremDisarlo}
	Let $\mathcal{S}, \mathcal{S}'$ be marked surfaces, such that $A_{\ast}\left(\mathcal{S}\right)$ is non-empty and isomorphic to $A_{\ast}\left(\mathcal{S}'\right)$. Then, $\Phi$ is surjective. If $\mathcal{S}$ is non-special, then $\Phi$ is a bijection. 
\end{thm}

\noindent 

We discuss the cases of special surfaces and provide a complete list of them in Section \ref{SectionSpecialSurfaces}. \medskip

\subsection{The simplicial isomorphism of an equivalence}
\ \smallskip

\noindent The main ingredient of our approach to study equivalences is the construction of a map
\[\begin{tikzcd} A_{\ast}: \Equ(\T, \T') \arrow{r} & \Simp(A_{\ast}(\mathcal{S}), A_{\ast}(\mathcal{S}')),\end{tikzcd}\]
where $\Equ(\T, \T')$ denotes the class of equivalence classes of $k$-linear triangle equivalences modulo natural isomorphisms. The set $\Aut(\T):=\Equ(\T, \T)$ is a group with multiplication given by composition.
 The idea is to find a diffeomorphism $\mathcal{S} \rightarrow \mathcal{S}'$, which realizes the action of $T$ on $[1]$-orbits of isomorphism classes in terms of its action on homotopy classes on curves. This is the approach we pursue below. For the rest of this section we fix surface-like categories $(\T, \mathcal{S}, \eta, \gamma, \basis)$ and $(\T', \mathcal{S}', \eta', \gamma, \basis)$. \medskip
  
\noindent We propose the following algebraic counterpart of essential arcs.

\begin{definition}\label{DefinitionEssentialObjects}
	Let $X \in \T$ be an arbitrary indecomposable. Then $X$ is called \textbf{essential} if it satisfies the following conditions.
	\begin{itemize}
	\setlength\itemsep{1ex}
		\item[1)] If $X \in \Perf(\T)$, then the middle term $Z$ of any Auslander-Reiten triangle \triangle[X][Z][\tau^{-1}(X)][][][][1] 
		is  decomposable. 
		\item[2)]  $\dim \Hom^*(X,X) \leq 2$. 
	\end{itemize}
\end{definition} \

\noindent It is clear that if $T: \T \rightarrow \T'$ is a triangle equivalence, then $X \in \T$ is essential if and only if $T(X)$ is essential.
\noindent Our first observation about essential objects is the following.
\begin{lem}\label{LemmaEssentialObjectsNotTauInvariant}
	Let $X \in \Perf(\T)$ be essential. Then $\tau X$ is not isomorphic to a shift of $X$. In particular, essential objects are arc objects.
\end{lem}
\begin{proof}
	Suppose $\tau X \cong X[n]$ for some $n \in \mathbb{Z}$. By Proposition \ref{PropositionFractionallyCalabiYauArcs}, $X$ must be represented by a loop or a boundary arc $\gamma$ contained in a boundary component with a single marked point.
	Suppose $\gamma$ is a loop. By condition 1) in Definition \ref{DefinitionEssentialObjects} and condition 3) in Definition \ref{DefinitionFukayaLikeQuintuples}, it has to sit at the base of its homogeneous tube, as otherwise $\dim \Hom^*(X,X) > 2$. However, this implies that the middle term of any Auslander-Reiten triangle starting in $X$ is indecomposable.
	
If $\gamma$ is an arc, it must be homotopic to a simple arc because of the second condition of Definition \ref{DefinitionEssentialObjects} and hence $\gamma$ is a boundary segment. Thus, the middle term of every Auslander-Reiten triangle starting in $X$ is indecomposable.
\end{proof}
\begin{lem}\label{LemmaCharacterizationEssentialObjects}
	Let $X \in \T$ be an arbitrary indecomposable. Then $X$ is essential if and only if $\gamma(X)$ contains an essential arc.
\end{lem}
\begin{proof}If $X$ is represented by a simple arc $\gamma$, then $\dim \Hom^*(X,X) \leq 2$ since it has at most two boundary intersections. Proposition \ref{PropositionCharacterizationBoundaryObjects} and Lemma \ref{LemmaEssentialObjectsNotTauInvariant} imply that every essential object is represented by an essential arc.\\
	For the converse, note that every interior self-intersection of a curve $\gamma$ gives rise to two distinct intersections of $\gamma$ with any perturbation $\gamma'$ of $\gamma$. In particular, the presence of an interior self-intersection forces $\dim \Hom_{\Int}^*(X,X) \geq 2$ and $\dim \Hom^*(X,X) \geq 3$. This follows from condition 2) ii) in Definition \ref{DefinitionFukayaLikeQuintuples} and the fact that the endpoints of an arc give rise to an automorphism of any its representing objects by Lemma \ref{LemmaBoundaryPointsAreInvertible}.
\end{proof}
\noindent Next, we define the map $A_{\ast}(-)$.
\begin{definition}
Let $T: \T \rightarrow T'$ a triangle equivalence. Define $A_{\ast}(T) \in \Simp(A_{\ast}\left(\mathcal{S}), A_{\ast}(\mathcal{S}')\right)$ as the unique simplicial isomorphism  which sends a vertex $\gamma(X) \in A_{\ast}(\mathcal{S})$, $X \in \T$ essential, to $\gamma(T(X)) \in A_{\ast}(\mathcal{S}')$.
\end{definition}
\noindent We recall that the arc complex of $\mathcal{S}$ is empty if and only if $\mathcal{S}$ is a disc without punctures and at most $3$ marked points.
\begin{lem}\label{LemmaAastWellDefined}Assume that $A_{\ast}(\mathcal{S}) \neq \emptyset$. Then, $A_{\ast}(-)$ is well-defined and is a group homomorphism if $(\T, \mathcal{S}, \eta, \gamma, \basis)=(\T', \mathcal{S}', \eta, \gamma, \basis)$.
\end{lem}
\begin{proof} Let $\gamma$ be an essential arc. By Lemma \ref{LemmaCharacterizationEssentialObjects}, $\gamma$ is represented by essential object $X_{\gamma}$. Let $T: \T \rightarrow \T'$ be a triangle equivalence. Then, $T(X_{\gamma})$ is essential and represented by a homotopy class of essential arcs. The assignment $\gamma \mapsto \gamma(T(X_{\gamma}))$ gives rise to a well defined (invertible) map between the vertices of $A_{\ast}(\mathcal{S})$ and $A_{\ast}(\mathcal{S}')$. We claim that this maps extends to a simplicial isomorphism $A_{\ast}(T): A_{\ast}(\mathcal{S}) \rightarrow A_{\ast}(\mathcal{S}')$. It is sufficient to prove that $\gamma$ and $\gamma'$ are connected by an edge if and only if $\gamma(T(X_{\gamma}))$ and $\gamma(T(X_{\gamma'}))$ are connected by an edge. This follows from the fact that $A_{\ast}(-)$ is a flag complex (see Remark \ref{RemarkFlagComplex}).\\
By Proposition \ref{PropositionInteriorMorphismsVectorSpace}, $\gamma$ and $\gamma'$ are connected by an edge if and only if the space of interior morphisms $\Hom^*_{\Int}(X_{\gamma}, X_{\gamma'})$ trivial. Thus, the claim follows from Lemma \ref{LemmaEquivalencesPreserveInteriorMorphisms}.\\
Note that $A_{\ast}(T)$ is the identity if $T$ is isomorphic to the identity functor proving that $A_{\ast}(-)$ is well-defined. The second assertion follows immediately from the definition of $A_{\ast}(-)$.
\end{proof}

\noindent It is a consequence of the previous lemma and Theorem \ref{TheoremDisarlo} that in case $A_{\ast}(\mathcal{S}) \neq \emptyset$, the surface $\mathcal{S}$ is a triangle invariant of the surface-like category $\T$.
	\begin{cor}
	Assume that $\mathcal{S}$ is non-empty. If $\T$ and $\T'$ are triangle equivalent, then $\mathcal{S} \cong \mathcal{S}'$. In particular, if $A,A'$ are derived equivalent gentle algebras, then $\mathcal{S}_A \cong \mathcal{S}_{A'}$.
	\end{cor}

\subsubsection{The diffeomorphism of an equivalence}
\label{SectionThediffeomorphismOfEquivalence}
\ \smallskip 

\noindent Given an equivalence $T \in \Equ(\T, \T')$ and assuming that $\mathcal{S}$ is non-special, we set $\Psi(T)\coloneqq  \Phi^{-1} \circ A_{\ast}(T)$. This defines a map
\[\begin{tikzcd}\Psi: \Equ(\T, \T') \arrow{r} & \Diffeo(\mathcal{S}, \mathcal{S}'). \end{tikzcd}\]
 which is a homomorphism of groups in case $\T$ and $\T'$ and their models are identical. \ \medskip

\noindent For the remainder of this section, however, we assume that $\mathcal{S}$ and $\mathcal{S}'$ are non-special. By construction, $\Psi$ satisfies the following.
\begin{cor}\label{CorollaryPsiRealizesActionOnEssentialObjects}
	Let $T: \T \rightarrow \T'$ be a triangle equivalence. Then, $\Psi(T)(\gamma(X))=\gamma(T(X))$ for all essential objects $X \in \T$.
\end{cor}
\noindent In light of our discussion of $\tau$-invariant families in the bounded derived category of the Kronecker algebra in Section \ref{SectionKroneckerAlgebra}, we can not expect Corollary \ref{CorollaryPsiRealizesActionOnEssentialObjects} to hold true for arbitrary indecomposable objects in $\T$. However, in the forthcoming sections we will show a generalization of Corollary \ref{CorollaryPsiRealizesActionOnEssentialObjects}, where we replace equality by the equivalence relation $\simeq_{\ast}$ from Section \ref{SectionFractionalCalabiYauTauInvariantFamilies}.
\medskip

\section{On diffeomorphisms induced by equivalences}\label{SectionOndiffeomorphisms}
\noindent In this section we investigate the properties of the diffeomorphism $\Psi(T)$ which is attached to an equivalence $T:\T \rightarrow \T'$ of surface-like categories. The most important results we  obtain are that $\Psi(T)$ preserves the orientation and winding numbers of loops. The key ingredient, however, is the following result which asserts that $\Psi(T)$ is a geometric realization of $T$.

\begin{thm}\label{TheoremPsiRealizesAction}
	Let $(\T, \mathcal{S}, \eta, \gamma, \basis)$ and $(\T', \mathcal{S}', \eta', \gamma, \basis)$ be surface-like quintuples and $T: \T \rightarrow \T'$ be a triangle equivalence. If $\mathcal{S}$ is non-special and $X \in \T$ is indecomposable, then 
	\[\Psi(T)(\gamma_{\ast}(X))=\gamma_{\ast}(T(X)).\]
\end{thm}
Its proof is found in Section \ref{SectionPreserveOrientation} and the necessary preparation occupies the majority of this section. 
\noindent An extension of the theorem to special surfaces is Theorem \ref{TheoremSpecialSurfaces}. Both theorems show that families of $\tau$-invariant objects are preserved under triangle equivalences in the following sense.
\begin{cor}\label{CorollaryEquivalencesPreserveFamilies}
	Let $\mathcal{X}$ be a family of $\tau$-invariant objects. Then, with the notation of Theorem \ref{TheoremPsiRealizesAction}, $T(\mathcal{X})$ is a family of $\tau$-invariant objects of the same type as $\mathcal{X}$. 
\end{cor}\smallskip

\noindent A direct consequence of Theorem \ref{TheoremPsiRealizesAction} is that if $X$ is not $\tau$-invariant, then any curve representing the object $T(X)$ is homotopic to the image $\Psi(T)(\gamma_X)$ of a representative $\gamma_X$ of $X$.

\subsubsection{Structure of the proof of Theorem \ref{TheoremPsiRealizesAction}}
\ \smallskip

\noindent We  outline the proof of Theorem \ref{TheoremPsiRealizesAction}. The idea is to reduce the general case to the case of essential objects in Corollary \ref{CorollaryPsiRealizesActionOnEssentialObjects}. We achieve this in the following way. The homotopy class of a curve is uniquely determined by its sequence of intersections with a triangulation assuming that the edges of the triangulation and the curves are in minimal position. Pursuing the approach that diffeomorphism invariant geometric properties correspond to triangle invariant properties in the surface-like category, we define what it means to be a \textit{triangulation} of $\T$ and translate intersections of the curve and the triangulation into a \textit{characteristic sequence of morphisms} between $X$ and objects of a triangulation. We show that these sequences are sufficient to recover the $\simeq_{\ast}$-class of an object.\medskip

\noindent Although the above ideas are comparatively simple in nature, their implementation  require us to be suprisingly careful about the definition of such sequences of morphisms. This seems to be a consequence of the failure of $\gamma(-)$ and the correspondence between morphisms and intersections to be canonical.\medskip

\subsection{Triangulations}\label{SectionTriangulations} We introduce the notion of a triangulation of a surface-like category which is inspired by the concept of a triangulation of a surface.

\subsubsection{Triangulations on surfaces}\label{SectionTriangulationsSurfaces}
\ \smallskip

\noindent By a \textbf{triangulation} of a marked surface $\mathcal{S}=(S, \marked)$, we mean a maximal collection $\Delta$ of homotopy classes of essential arcs on $\mathcal{S}$ with pairwise disjoint interiors. In other words, a triangulation is a simplex  in $A_{\ast}(\mathcal{S})$ of maximal dimension. By abuse of notation, we do not distinguish between a triangulation and a (fixed) choice of representatives for its homotopy classes in minimal position. The arcs of a triangulation cut $S$ into regions, the closure of which we call \textbf{triangles}. This possibly includes self-folded triangles and triangles with one or two sides given by boundary segments.

 From time to time it is convenient to add all boundary segments to a triangulation $\Delta$. The set of homotopy classes $\overline{\Delta}$ we obtain in this way is called an \textbf{extended triangulation} and corresponds to a simplex of maximal dimension of the extended arc complex.
If $\Delta$ is an (extended) triangulation and $\gamma$ is a non-contractible curve we say that $\gamma$ \textbf{$\Delta$-admissible} if $\{\gamma\}\cup \Delta$ is in minimal position and $\gamma$ is not homotopic to an arc of $\Delta$.\medskip

\noindent For technical reasons we will frequently restrict ourselves to particular types triangulations of the following kind. Let $\Delta$ be a triangulation and denote $\Delta' \subseteq \Delta$ its subset of homotopy classes of finite arcs. The triangulation $\Delta$ is said to  \textbf{separate punctures} if $\Delta'$ contains a  triangulation of $(S, \marked \cap\ \partial S)$ and $\Delta$ only contains finite or semi-finite arcs.
\begin{lem}\label{LemmaSeparatingPunctureTriangulationExist}
Every marked surface $(S, \marked)$ admits a triangulation which separates punctures.
\end{lem}
\begin{proof} First, note that the marked surface $(S, \marked \cap {\,\partial S})$ has a triangulation $\Delta'$ as by definition, $\partial S \neq \emptyset$ and every boundary component contains a marked point. As $\Delta'$ cuts $S$ into topological discs, the proof of the existence of a triangulation of $(S, \marked)$ which separates punctures reduces to the assertion that every punctured  triangle with any number of punctures admits such a triangulation. The latter follows by induction on the number of punctures. The construction is outlined in Figure \ref{FigureTriangulationSeparatesPunctures}.

\end{proof}
\begin{figure}[H]
	\begin{displaymath}
	\begin{tikzpicture}[scale=1.5]
	%infinite arcs
	\draw[orange] (-0.35,{1/sqrt(3)+0.15})--(-1,0);
	\draw[orange] (-0.35,{1/sqrt(3)+0.15})--(0, {sqrt(3)});
	
	\draw[orange] (+0.35,{1/sqrt(3)+0.15})--(+1,0);
	\draw[orange] (+0.35,{1/sqrt(3)+0.15})--(0, {sqrt(3)});
	
	% other arcs
	\draw [line width=0.5, orange] plot  [smooth, tension=0.3] coordinates { (-1,0) ({-0.15}, {1/sqrt(3)}) (0, {sqrt(3)})  };
	\draw [line width=0.5, orange] plot  [smooth, tension=0.3] coordinates { (1,0) ({0.15}, {1/sqrt(3)}) (0, {sqrt(3)})  };
	
	% triangle		
	\draw (-1,0)--(0,{sqrt(3)});
	\draw (-1,0)--(1,0);
	\draw (1,0)--(0,{sqrt(3)});
	
	%punctures
	\filldraw (-0.35,{1/sqrt(3)+0.15}) circle (0.75pt);
	\filldraw (0.35, {1/sqrt(3)+0.15}) circle (0.75pt);

	\end{tikzpicture}
	\end{displaymath}
	\caption{A triangulation of a twice-punctured triangle which separates punctures.} \label{FigureTriangulationSeparatesPunctures}
\end{figure}
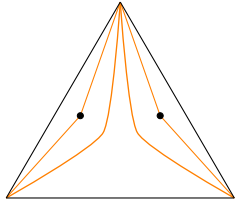

\noindent The following will be useful in subsequent parts of this paper.
\begin{lem}\label{LemmaTriangulationSeparatesPuncturesLoopAroundPuncture}Let $\Delta$ be a triangulation which separates punctures and let $\gamma$ be $\Delta$-admissible. Then, $\gamma$ intersects a finite arc of $\Delta$ or $\gamma$ is homotopic to a loop around a puncture.
\end{lem}
\begin{proof}
	Suppose $\gamma$  has no intersection with a finite arc of $\Delta$. In particular, $\gamma$ is a loop or an arc both end points of which are punctures. Since $\gamma$ is $\Delta$-admissible, it intersects at least one arc $\delta \in \Delta$ and denote by $p$ the unique puncture which is an end point of $\delta$. Let $\delta' \in \Delta$ be a another arc which is crossed by $\gamma$. Then, $\delta'$ start or ends on a puncture $q$ and we show that $q=p$. If $p \neq q$, then $p$ and $q$ lie on the boundary of a triangle and hence there exists an arc $\epsilon \in \Delta$ connecting $p$ and $q$. This contradicts our assumption on $\Delta$ and hence $p=q$. We conclude that $\gamma$ is not an arc as otherwise it would be contractible. Hence $\gamma$ is a loop which intersects only infinite arcs of $\Delta$ with end point $p$ proving that $\gamma$ is a loop around $p$.
\end{proof}

\noindent A well-known property of triangulations is that they allow us to relate the homotopy classes of a curve with its sequence of intersections with the triangulation. We rephrase this fact in a suitable way.
\begin{definition}\label{DefinitionOrderingSets}Let $\Delta$ be an (extended) triangulation of $\mathcal{S}$ and let $\gamma$ be an oriented $\Delta$-admissible curve. By the \textbf{sequence of oriented intersections} of $\gamma \overrightarrow{\cap}\Delta$ we mean the set $\bigcup_{\delta \in \Delta} \gamma \overrightarrow{\cap} \delta$ with a linear or cyclic order defined as follows.
\begin{enumerate}
\setlength\itemsep{1ex}
\item If the $\simeq_{\ast}$-class of $\gamma$ contains a loop, $\gamma \overrightarrow{\cap} \Delta$ is cyclicly ordered according to the orientation of $\gamma$.
\item If the $\simeq_{\ast}$-class of $\gamma$ does not contains loops, $\gamma \overrightarrow{\cap} \Delta$ is linearly ordered by the orientation $\gamma$ and the orientation of the surface, i.e. if $\delta, \delta' \in \Delta$ have an oriented intersection $q=(t,t') \in \delta \overrightarrow{\cap} \delta'$ on the boundary and $p=(s,t) \in \gamma \overrightarrow{\cap} \delta$, $p'=(s,t') \in \gamma \overrightarrow{\cap} \delta'$, then $p < p'$.
\end{enumerate}
The set $\Delta \overrightarrow{\cap} \gamma$ is defined in the analogous way and is referred to as the \textbf{dual sequence of oriented intersections}.
\end{definition}\smallskip

\noindent We say that the (linear or cyclic) sequence of arcs $(\delta_i)_{i \in I}$ of $\Delta$ crossed by a $\Delta$-admissible curve $\gamma$ is \textbf{reduced} if the existence of $i \in I$ with $\delta_i=\delta_{i+1}$, implies that  $\delta_i$ ends at  a puncture $q$ and $\delta_i$ is the only arc of $\Delta$ which ends at $q$. \medskip

\noindent The following lemma is reformulation in our language of the fact that the homotopy class of a curve is determined by its sequence of intersections with a triangulation.

\begin{lem}\label{LemmaIntersectionsDetermineCurves}Let $\Delta$ be an extended triangulation of $\mathcal{S}$ which separates punctures. The following is true. 

	\begin{enumerate}
	\setlength\itemsep{1ex}
		\item The $\simeq_{\ast}$-class of a $\Delta$-admissible curve $\gamma$ is uniquely determined by its (dual) sequence of oriented intersections and such a sequence is reduced.
	  	\item If $(\delta_{i})_{i \in I}$ is a reduced linear or cyclic sequence of arcs of $\Delta$ such that for each $i \in I$, $\delta_{i}$ and it successor lie on the boundary of a triangle of $\Delta$, then there exists a unique $\simeq_{\ast}$-class of a $\Delta$-admissible curve $\gamma$ such that $(\delta_i)_{i \in I}$ is the sequence of oriented intersections.  \smallskip
	\end{enumerate}
	
\end{lem}
\begin{rem}
	The previous lemma holds true for non-extended triangulations whenever $\mathcal{S}$ is not a disc with at most $4$ marked points and no punctures. In case $\mathcal{S}$ is a disc with at most $3$ marked points, the empty set is a triangulation. In the remaining case, a triangulation consists of a single arc and and there exist boundary segments with the same sequence of oriented intersections.
	\end{rem}\smallskip

\subsubsection{Triangulations of surface-like categories}\label{SectionTriangulationsFukayaLikeCategories}
\ \smallskip

\noindent We propose the following definition of a triangulation of a surface-like category. Throughout this section, we fix a surface-like category $(\T, \mathcal{S}, \eta, \gamma, \basis)$.

\begin{definition}\label{DefinitionTriangulationCategory}
	A \textbf{triangulation} of $\T$ is a set $\mathcal{X}\subset \T$ of essential objects with the following properties.
	\begin{itemize}
	\setlength\itemsep{1ex}
		\item[1)] If $X, X' \in \mathcal{X}$ are distinct, then $X$ and $X'$ are not isomorphic up to shift and 
		\[\Hom^*_{\Int}(X, X')=0.\]
		\item[2)] $\mathcal{X}$ is maximal among all sets of essential objects satisfying 1).
	\end{itemize}
	
	A triangulation $\mathcal{X} \subset \T$ \textbf{separates punctures} if further
	
	\begin{itemize}
	\setlength\itemsep{1ex}
		\item[3 a)] $\mathcal{X} \cap \Perf(\T)$ is a maximal subset of essential objects in $\Perf(\T)$ satisfying condition 1) above  and,
		\item[3 b)]for all $X \in \mathcal{X}$, $X$ is semi-perfect (see Definition \ref{DefinitionSemiPerfect}). 

	\end{itemize}
	
\noindent One obtains an \textbf{extended triangulation} of $\T$ by adding a representative of every $[1]$-orbit of isomorphism classes of segment objects to a triangulation of $\T$. 

 An  indecomposable object $X \in \T$ is said to be \textbf{$\mathcal{X}$-admissible} if $X$ is not a segment object and $X$ is not isomorphic to a shift of an object in $\mathcal{X}$.
\end{definition}

\noindent The first observation is that if $\mathcal{X} \subset \T$ is an (extended) triangulation of $\T$ and $T: \T \rightarrow \T'$ is a triangle equivalence between surface-like categories, then $T(\mathcal{X})$ is a triangulation and $\mathcal{X}$ separates punctures if and only if $T(\mathcal{X})$ separates punctures. 
\begin{lem}\label{AlgebraicAndGeometricTriangulationAgree}
	A collection of indecomposable objects $\mathcal{X} \subseteq \T$ is a triangulation if and only if the associated collection $\gamma(\mathcal{X})$ of homotopy classes is a triangulation of $\mathcal{S}$. Moreover, 
	\begin{enumerate}
	\setlength\itemsep{1ex}
		\item $\mathcal{X}$ separates punctures if and only if $\gamma(\mathcal{X})$ separates punctures, and
		\item an indecomposable object $X \in \T$ is $\mathcal{X}$-admissible if and only if $\gamma(X)$ is $\gamma(\mathcal{X})$-admissible. 
	\end{enumerate} 
\end{lem}
\begin{proof}Set $\Delta:=\gamma(\mathcal{X})$. Lemma \ref{PropositionInteriorMorphismsVectorSpace} implies that $\mathcal{X}$ is triangulation if and only if $\Delta$ is a triangulation. An object $X \in \mathcal{X}$ is (semi-)perfect if and only its representing arcs are (semi-)finite. We conclude that $\mathcal{X}$ separates punctures if and only if $\Delta$  only contains finite or semi-finite and contains a maximal collection of pairwise non-homotopic finite arcs with disjoint interior.

Let $\Delta'$ is the subcollection of homotopy classes of finite arcs in $\Delta$. If $\Delta$ separates punctures, then every polygon of $\Delta'$ contains at most one puncture as every arc of $\Delta \setminus \Delta'$ start or ends on the boundary and $\Delta$ is maximal. Therefore $\Delta'$ is a maximal collection of pairwise non-homotpic finite arcs with pairwise disjoint interior. On the other hand, if $\Delta'$ is such a maximal collection of finite arcs, then the closure of every polygon of $\Delta'$ contains at most one puncture and hence every arc of $\Delta \setminus \Delta'$ has an end point on the boundary. 
\end{proof}
\begin{cor}
	Every surface-like category admits a triangulation which separates punctures.
\end{cor} \medskip

\subsection{Characteristic sequences}\label{SectionOnCharacteristicSequences} 
\ \smallskip

\noindent The next concept we introduce is the \textit{characteristic sequence} of an object which is the algebraic counterpart to sequences of intersections of a curve with a triangulation as in Lemma \ref{LemmaIntersectionsDetermineCurves}. It consists of `` dual''  sequences of morphisms and allow us to recover the $\simeq_{\ast}$-equivalence class of any object in $\T$.\medskip

\noindent As before, and unless stated otherwise, we assume that every indecomposable object appearing in any of the subsequent sections is either an arc object or a loop object associated to a $1$-dimensional local system.

\subsubsection{The idea behind characteristic sequences}
\label{SectionIdeaChracteristicSequences}
\ \smallskip

\noindent Suppose $\Delta$ is a triangulation of $S$ and assume that $\gamma$ is a gradable loop or arc on $\mathcal{S}$ which is $\Delta$-admissible. Assume further  that $p_1$ and $p_2$ are consecutive interior intersections of $\gamma$ with edges $\delta_1, \delta_2 \in \Delta$. There exists a unique intersection $q$ of $\delta_1$ and $\delta_2$, such that $p_1, p_2$ and $q$ lift to an embedded triangle in the universal cover of $S$. Assuming that  $q \in \delta_1 \overrightarrow{\cap} \delta_2$ and regarding $p_i$ as an element in $\gamma \overrightarrow{\cap} \delta_i$, we have $\basis(q) \circ \basis(p_1)= \lambda \cdot \basis(p_2)$ for some $\lambda \neq 0$ as implied by Definition \ref{DefinitionFukayaLikeQuintuples}. It suggest that our desired sequences of morphisms should be such that every two consecutive morphisms in this sequence are related via composition with a morphism between objects of the corresponding triangulation of $\T$. 

As it turns out, a single sequence obtained in this way is not sufficient to achieve what we want. However, we show that the sequence and its dual obtained by regarding each $p_i$ as an element in $\delta_i \overrightarrow{\cap} \gamma$, encode enough information.
Note that in case of the dual sequence, $\mu \cdot \basis(p_1)=\basis(p_2) \circ \basis(q)$ for some $\mu \neq 0$.

Since it is possible that $\delta_1$ and $\delta_2$ intersect more than once at the boundary, it becomes necessary to distinguish morphisms between essential objects which are multiples of basis elements from proper linear combinations of such. This idea motivates the notions of \textit{pure morphisms} and \textit{arrow morphisms} defined below.

\subsubsection{Pure morphisms and arrow morphisms}
\ \smallskip

We say that the pair $(\mathcal{S}, \omega_{\eta})$ is of \textbf{tubular type} if $\mathcal{S}$ is a cylinder without punctures and $\omega_{\eta}(\gamma)=0$ for the unique homotopy class of loops. Note that if $\mathcal{S}$ is a cylinder, then by Proposition \ref{PropositionFractionallyCalabiYauArcs}, $(S, \omega_{\eta})$ is of tubular type if and only if $\T$ has $\tau$-invariant indecomposable objects.
\begin{definition}\label{DefinitionArrowMorphisms}Assume that $(\mathcal{S}, \omega_{\eta})$ is not of tubular type and let $X, Y \in \T$ be perfect or semi-perfect essential objects or segment objects and assume that $\Hom^*_{\Int}(X,Y)=0$.
We call a non-zero morphism $f: X \rightarrow Y[m]$ \textbf{pure} if any of the following conditions is satisfied:
	\begin{itemize}
	\setlength\itemsep{1ex}
		\item[1)]  $X$ and $Y$ are perfect and there exist segment objects $X'$, $Y'$ and non-zero morphisms $g: X' \rightarrow X$ and $h: Y[m] \rightarrow Y'$, such that $f \circ g=0=h \circ f$.
		\item[2)] $X$ or $Y$ is not perfect and if $f$ does not factor through $\Perf(\T)$, then there exists a segment object $U$ and a morphism $g: U \rightarrow X$ such that $\Hom(g, Y) \neq 0$ and $f\circ g=0$.
	\end{itemize}
If $\mathcal{X}$ is a triangulation which separates punctures and $X,Y \in \mathcal{X}$, then $f$ is an \textbf{arrow morphism of $\mathcal{X}$} if $f$ is pure and it is not a composition of pure morphisms $X\rightarrow Z[n]$ and $Z[n] \rightarrow Y[m]$, where $Z \in \mathcal{X} \setminus \{X,Y\}$. 

	\end{definition}

It follows that being pure (resp. an arrow morphism) is preserved under triangle equivalences. The assumptions on the objects $X$ and $Y$ in Definition \ref{DefinitionArrowMorphisms} imply that they are representable by essential arcs or boundary segments with disjoint interiors.

\noindent Next we show, that all pure morphisms are multiples of morphisms associated with intersections. We distinguish between perfect and non-perfect objects.
\begin{lem}\label{LemmaCharacterizationPureMorphismsPerfectObjects}Assume that $X, Y$ and $0 \neq f:X \rightarrow Y$ satisfy the prerequisites of Definition \ref{DefinitionArrowMorphisms} and assume that $X$ and $Y$ are perfect. If $\gamma_X \in \gamma(X)$ and $\gamma_Y \in \gamma(Y)$ are in minimal position, then $f$ is pure if and only if it is multiple of $\basis(p)$ for some $p \in \gamma_X \overrightarrow{\cap} \gamma_Y$.
\end{lem}
\begin{proof}
	There is nothing to show, if $|\gamma_X \overrightarrow{\cap} \gamma_Y|=1$. Assume therefore that $q,q' \in \gamma_X \overrightarrow{\cap} \gamma_Y$ are distinct intersections. In a universal cover of $\mathcal{S}$, there exists a lift $\widetilde{\gamma}_X$ of $\gamma_X$  and lifts of $\gamma_Y$ arranged as in Figure \ref{FigureCharacterizationPureMorphisms}.
	\begin{figure}[H]
		\begin{displaymath}
		\begin{tikzpicture}[scale=2]
		% boundary
		\draw (-1, 0.2)--(-1, 1);
		\draw (-1, -0.6)--(-1, -1);
		\draw (1,-1)--(1,-0.2);
		\draw (1,1)--(1,0.6);
		% marked points
		% left
		\filldraw (-1,-0.6) circle (1pt);
		
		\filldraw (-1,0.2) circle (1pt);
		
		% right
		
		\filldraw (1,-0.2) circle (1pt);
		
		\filldraw (1,0.6) circle (1pt);
		
		% lifts
		\draw  (-1, -0.6)--(1, -0.2);
		\draw (1, -0.2)--(-1,0.2);
		\draw (-1, 0.2)--(1,0.6);
		
		% labels
		
		\draw (0, 0+0.12) node{$\widetilde{\gamma}_X$};
		
		\draw (1+0.12, -0.2) node{$q$};
		
		\draw (-1-0.12, 0.2) node{$q'$};
		\end{tikzpicture}
		\end{displaymath}
		\caption{Two lifts of $\gamma_Y$ intersect the lift $\widetilde{\gamma}_X$.} \label{FigureCharacterizationPureMorphisms}
	\end{figure}
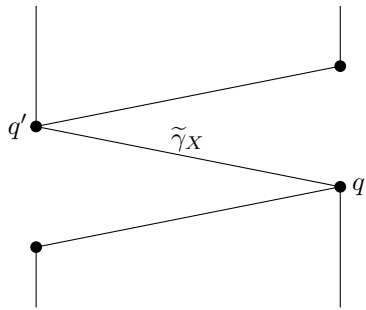
	\noindent We like to stress that the boundaries containing $q$ and $q'$ might coincide.  However in that case, the end points of $\widetilde{\gamma}_X$ are not consecutive marked points since otherwise $\gamma_X \simeq \tau \gamma_Y$ or $\gamma_X \simeq \gamma_Y$ and $|\gamma_X \overrightarrow{\cap} \gamma_Y| \leq 1$.  Let $\delta$ denote the boundary arc connecting $q'$ with its predecessor in $\marked$. Then $\delta$ and $\widetilde{\gamma}_X$ have a unique intersection $p \in \widetilde{\gamma}_Y \overrightarrow{\cap} \delta$. As $\delta$ and $\widetilde{\gamma}_Y$ lie on different sides of $\widetilde{\gamma}_X$ it follows from Definition \ref{DefinitionFukayaLikeQuintuples} 4) that  $\basis(p) \circ \basis(q)=0$. Similarly, if $p'$ denotes the unique intersection in $\delta' \overrightarrow{\cap} \widetilde{\gamma}_X$, where $\delta'$ is the boundary arc connecting $q'$ and its successor in $\marked$, then $\basis(q) \circ \basis(p')=0$. This proves that $\basis(q)$ is pure. A similar proof shows that $\basis(q')$ is pure.

	To prove the other implication, note that $p$ and $q'$ give rise to a fork and hence $\basis(p) \circ \basis(q')$ is a multiple of the morphisms associated to the unique intersection in $\widetilde{\gamma}_X \overrightarrow{\cap} \delta$. Similarly, if $\epsilon$ denotes the boundary arc which connects $q$ and its predecessor and $u \in \widetilde{\gamma}_Y \overrightarrow{\cap} \epsilon$ is the unique intersection, then $\basis(u) \circ \basis(q)$ is a multiple of $\basis(u')$, where $u' \in \widetilde{\gamma}_X \overrightarrow{\cap} \epsilon$ denotes the unique intersection.\\
	Now suppose $g:Y \rightarrow Y'$ is a non-zero morphism to a segment object $Y'$ and $g \circ f=0$. If $Y'$ is represented by $\delta$ or $\epsilon$, then by the above, $f$ is a multiple of $\basis(q)$ or a multiple of $\basis(q')$. Otherwise, $Y'$ is a loop object and w.l.o.g. we may assume that $\gamma_{Y'}$ can be homotoped on the boundary component of $q$. If the intersection of $\tilde{\gamma}_{Y'}$ with $\tilde{\gamma}_X$ and $\tilde{\gamma}_Y$ formed an intersection triangle with $q'$, then $(\mathcal{S}, \omega_{\eta})$ would be of tubular type with $\gamma_X$ and $\gamma_Y$ forming a triangulation. Thus, $f$ is a multiple of $\basis(q')$.
\end{proof}
	
	\begin{lem}\label{LemmaCharacterizationPureMorphismsNonPerfectObjects}
	Assume that $X, Y$ and $0 \neq f:X \rightarrow Y$ satisfy the prerequisites of Definition \ref{DefinitionArrowMorphisms} and assume that $X$ or $Y$ is semi-perfect. If $\gamma_X \in \gamma(X)$ and $\gamma_Y \in \gamma(Y)$ are in minimal position, then $f$ is pure if and only if it is multiple of $\basis(p)$ for some $p \in \gamma_X \overrightarrow{\cap} \gamma_Y$.
	\end{lem}
	\begin{proof}
		Let $\gamma_X$ and $\gamma_Y$ be representing arcs of $X$ and $Y$ in minimal position. By symmetry of the following arguments we may assume that $\gamma_X$ is an semi-finite arc. In particular, $\gamma_X$ and $\gamma_Y$ share at most one boundary point $p$ and at most one puncture $q$. By Lemma
		\ref{LemmaMorphismFactorsThroughPerfIffPuncture}, $f$ factors through $\Perf(\T)$, if and only if it is a multiple of $\basis(p)$.

Next, assume that $f$ does not factor through $\Perf(\T)$.
As in the proof of Lemma \ref{LemmaCharacterizationPureMorphismsPerfectObjects}, we see that there exists a segment object $U$ represented by a boundary segment, which connects $p$ and its successor, and a morphism  $g: U \rightarrow X$ (unique up to a scalar), such that $h \circ g\neq 0$ for a morphism $h:X \rightarrow \tau U[1]$. Write $f=\lambda \cdot \basis(q) + \mu \cdot \basis(p)$. Consequently,  $g \circ f=0$ if and only if $\mu=0$ if and only if $f$ is a multiple of $\basis(q)$. This finishes the proof. 
	\end{proof}

\noindent The following is a direct consequence of Lemma \ref{LemmaCharacterizationPureMorphismsPerfectObjects} and Lemma \ref{LemmaCharacterizationPureMorphismsNonPerfectObjects}.
\begin{cor}\label{CorollaryCharacterizationArrowMorphisms}Assume that $(\mathcal{S}, \omega_{\eta})$ is not of tubular type. Let $\mathcal{X}$ be a triangulation which separates punctures and $X,Y \in \mathcal{X}$. The arrow morphisms  $X \rightarrow Y$ are precisely the multiples of morphisms $\basis(p)$, where $p$ is the corner of a triangle in the corresponding triangulation of $S$. In particular, up to scalar, there exists at most one arrow morphism $X \rightarrow Y$.
\end{cor}

\subsubsection{Definition of characteristic sequences}\label{SectionDefinitionCharacteristicSequences}
\ \smallskip

\noindent We are finally prepared to state the definition of characteristic sequences. In order to unify notation, we extend the definition of $\tau$ to non-perfect objects. If $X \not \in \Perf(\T)$ is indecomposable but not semi-perfect, we set $\tau X \coloneqq X$. Otherwise, we require that there exists a distinguished triangle
\[\begin{tikzcd} \tau X \arrow{r} & U \arrow{r}{f} & X \arrow{r} & \tau X[1] \end{tikzcd},\]
where $U$ is a segment object and $f$ is any non-zero morphism. This determines $X$ uniquely up isomorphism. In particular, we have $\gamma(\tau X)=\tau \gamma(X)$.

\begin{definition}\label{DefinitionCharacteristicSequence}Let $\mathcal{X}$ be a triangulation of $\T$ which separates punctures. Let $X \in \T$ be $\mathcal{X}$-admissible and assume that $X$ is $\tau$-invarant (resp.\@ not $\tau$-invariant).\\
	A cyclic (resp.\@ linear) sequence of pairs $(\phi_0, \overline{\phi}_0), \dots, (\phi_m, \overline{\phi}_m)$  of non-zero morphisms
	\[(\phi_j, \overline{\phi}_j) \in \Hom(X, Y_j[n_j]) \times \Hom(\tau^{-1}Y_j[n_j-1],X),\] 
	where $Y_0, \dots, Y_m \in \mathcal{X}$, is called a \textbf{characteristic sequence of $X$} (with respect to $\mathcal{X}$) if it satisfies all of the following conditions.
	\begin{enumerate}
	\setlength\itemsep{1ex}
		\item \begin{itemize}
		\setlength\itemsep{0.75ex}
			\item[a)] The morphisms $\phi_j$ form basis of $\bigoplus_{Y \in \mathcal{X}}{\Hom^*(X, Y)}$.
			\item[b)] The morphisms 
			$\overline{\phi}_j[-n_j+1]$ form a basis of $\bigoplus_{Y \in \mathcal{X}}{\Hom^*(\tau^{-1}Y, X)}$.
		\end{itemize}  
		\item  For every $j \in [0,m]$ (resp.\@ $j \in [0,m)$), there exists a pair $\{\sigma_j^1, \sigma_j^2\}=\{j,{j+1}\}$ with the following properties:
		\begin{itemize}
		\setlength\itemsep{0.75ex}
			\item [a)] There exists  an arrow morphism 
			\[\alpha_j: Y_{\sigma_j^1}[n_{\sigma_j^1}] \rightarrow Y_{\sigma_j^2}[n_{\sigma_j^2}],\]
			such that $\alpha_j \circ \phi_{\sigma_j^1}$ is a non-zero multiple of $\phi_{\sigma_j^2}$.
			\item[b)] There exists an arrow morphism  
			\[\overline{\alpha}_j: \tau^{-1}Y_{\sigma_j^1}[n_{\sigma_j^1}-1] \rightarrow \tau^{-1}Y_{\sigma_j^2}[n_{\sigma_j^2}-1],\]
			such that $\overline{\phi}_j \circ \overline{\alpha}_j$ is a non-zero multiple of $\overline{\phi}_{\sigma_j^2}$.
		\end{itemize}
		\item If $Y_j \in \Perf(\T)$, then $\phi_j \circ \overline{\phi}_j$ is a connecting morphism in an Auslander-Reiten triangle.
	\end{enumerate}
	We refer to the cyclic (resp.\@ linear) sequence $Y_0, \dots, Y_m$ as a \textbf{characteristic sequence of objects} of  $X$.
\end{definition}

\noindent We shall consider equivalence classes of characteristic sequences up to rotation and inversion.  Part 1) and 2) of Definition \ref{DefinitionCharacteristicSequence} follow the ideas described in Section \ref{SectionIdeaChracteristicSequences}. The third part incorporates the idea that $\phi_j$ and $\overline{\phi}_j$ should be ``dual'' with respect to the Serre pairing as discussed in Section \ref{SectionSerrePairing}.\medskip

\noindent We have the following.
\begin{lem}\label{LemmaCharacteristicSequencePreservedUnderEquivalences}Let $\mathcal{X}$ a triangulation of $\T$ which separates punctures and let $X \in \T$ be $\mathcal{X}$-admissible. If $(\phi_0, \overline{\phi}_0), \dots, (\phi_m, \overline{\phi}_m)$ is a characteristic sequence of $X$ with respect to $\mathcal{X}$ and $T: \T \rightarrow \T'$ is a triangle equivalence, then 
	\[\left(T\left(\phi_0\right), T\left(\overline{\phi}_0\right)\right), \dots,  \left(T\left(\phi_m\right), T\left(\overline{\phi}_m\right)\right),\]
	
	 \noindent is a characteristic sequence of $T(X)$ with respect to $T(\mathcal{X})$.
\end{lem}

\subsubsection{Existence and reconstruction of objects from characteristic sequences}
\ \smallskip

\noindent The following lemma shows that characteristic sequences exist.
\begin{lem}Let $\T$ be a surface-like category, $\mathcal{X}$ be a triangulation of $\T$ which separates punctures and let $X \in T$ be $\mathcal{X}$-admissible. Then, $X$ has a characteristic sequence with respect to $\mathcal{X}$.
\end{lem}
\begin{proof}Let $\Delta$ be a triangulation corresponding to $\mathcal{X}$. Since $X$ is $\Delta$-admissible, $\gamma$ intersects at least one of its arcs  in the interior or is homotopic to a boundary segment.\\
	Let $\gamma \in \gamma_{\ast}(X)$ such that $\{\gamma\} \cup \Delta$ is in minimal position. We assume that $\gamma$ is a loop if $X$ is $\tau$-invariant.   
	Let $p_0, \dots, p_m$ be the cyclic (resp.\@ linear) sequence of oriented intersections from $\gamma$ to arcs in $\Delta$. For each $j \in [0,m]$, let  $\delta_j \in \Delta$, such that $p_i \in \gamma \overrightarrow{\cap} \delta_{i}$ and denote $Y_j \in \mathcal{X}$, such that $\delta_j \in \gamma(Y_j)$.\\
	If $X$ is a loop object or not $\tau$-invariant, set $\phi_j\coloneqq \basis(p_j) \in \Hom^*(X, Y_j)$ and $\overline{\phi}_j\coloneqq \basis(\mathbb{S}^{-1}(p_j)) \in \Hom^*(\tau^{-1}Y_j[n_j-1], X)$.  If $X$ is a $\tau$-invariant arc object, then any homotopy $H$ from $\gamma$ to a $\Delta$-admissible representative $\gamma'$ of $X$ induces bijections 
{$\hat{H}:\gamma \overrightarrow{\cap} \delta_j \rightarrow \gamma' \overrightarrow{\cap} \delta_j$} and the pair $\left(\phi_j, \overline{\phi}_j\right)$ the pair associated with $\left(\hat{H}(p_i), \mathbb{S}\hat{H}\mathbb{S}^{-1}(p_i)\right).$

	It follows from Proposition \ref{PropositionExistenceARTrianglesForArcs} that  $\phi_j \circ \overline{\phi}_j$  is a connecting morphism in an Auslander-Reiten triangle, whenever $\delta_j$ is a finite arc.
	Property (1) of Definition \ref{DefinitionCharacteristicSequence} follows from Definition \ref{DefinitionFukayaLikeQuintuples} 2) III). \\
	Since $p_{j+1}$ and $p_j$ lie on the boundary of the same triangle of $\Delta$, there exist $\sigma_j^1, \sigma_j^2$ with $\{\sigma_j^1, \sigma_j^2\}=\{j,j+1\}$ and a unique intersection $q_j \in \delta_{\sigma_j^1}  \overrightarrow{\cap} \delta_{\sigma_j^2}$, such that $p_j, p_{j+1}$ and $q_j$ form an intersection triangle. Then, for $\alpha_j\coloneqq \basis(q_j):Y_{\sigma_j(1)}[n_{\sigma_j(1)}] \rightarrow Y_{\sigma_j(2)}[n_{\sigma_j(2)}]$, the composition $\alpha_j \circ \phi_{\sigma_j(1)}$ is a multiple of $\basis(p_{\sigma_j(2)})$. For $\overline{\alpha}_j \coloneqq \mathbb{S}^{-1}(q_j)$,  Property 2b) in Definition \ref{DefinitionCharacteristicSequence} follows in a similar way.
\end{proof}

\noindent Next, we show that two objects have equivalent characteristic sequences of objects if and only if their associated $\simeq_{\ast}$-classes coincide.\medskip

\noindent \textbf{Assumptions \& notation. \,} For the remainder of this section we fix a triangulation $\mathcal{X}$ of $\T$, which separates punctures, and denote by $\Delta$ a corresponding triangulation. We fix an $\mathcal{X}$-admissible object $X \in \T$ and let $\gamma \in \gamma_{\ast}(X)$ be a $\Delta$-admissible curve which we assume to be a loop if $\gamma_{\ast}(X)$ contains such a curve. We write $(p_i)_{i}$ for its linear or cyclic sequence of oriented  intersections with $\Delta$. Finally, let $(\phi_j, \overline{\phi}_j)_{j}$ be a characteristic sequence of $X$, let $(Y_j)$ denote the corresponding sequence of objects in $\mathcal{X}$ and let $\delta_j \in \Delta$ denote the representative of $Y_j$.

 Depending on whether $X$ is $\tau$-invariant or not, we identify the parametrizing set of indices $j$ with either an interval of integers or elements in a cyclic group in the natural way. For each  $j$, let $\sigma_j^1, \sigma_j^2$, $\alpha_j$ and $\overline{\alpha}_j$ be as in Definition \ref{DefinitionCharacteristicSequence}.\\
Finally for each $j$,  let 
\[\begin{array}{ccc}\displaystyle \phi_{j}=\sum_{l \in I_j}{b_l} & \text{and} & \displaystyle  \overline{\phi}_{j}=\sum_{\overline{l} \in \overline{I}_j}{\overline{b}_{\overline{l}}} \end{array}\]
denote standard decompositions with respect to $\gamma$, $\delta_j$ and $\tau^{-1}\delta_j$. We say that $\overline{l} \in \overline{I}_j$ and $l \in I_j$ are \textbf{dual} if the intersections corresponding to $b_l$ and $\overline{b}_{\overline{l}}$ are related by the bijection $\mathbb{S}: \gamma \overrightarrow{\cap} \Delta \rightarrow \tau^{-1}\Delta \overrightarrow{\cap} \gamma$.\\

\noindent The following proposition is an algebraic analogue of Lemma \ref{LemmaIntersectionsDetermineCurves}.
\begin{prp}\label{PropositionCharacteristicSequenceDefineObjectTauInvariant}Let $X, Y \in \T$ be $\mathcal{X}$-admissible. Then, all characteristic sequences of objects in $\mathcal{X}$ of $X$ coincide up to equivalence. Moreover, $\gamma_{\ast}(X)=\gamma_{\ast}(Y)$ if and only if their characteristic sequences of objects are equivalent. 

\end{prp}
\begin{proof}It is sufficient to show that we can reconstruct the sequence of arcs in $\Delta$ crossed by $\gamma$ from a characteristic sequence.
	
The set $N=\{0,\dots, n\}$ is equipped with a linear or cyclic ordering. By an interval $[x,y]$ of two of its elements $x,y$ we mean the set of elements which lie between $x$ and $y$, including $x$ and $y$.
The set $N$ is divided into non-empty intervals $N_0, \dots, N_m$ which are maximal with respect to the property that  for all $0 \leq i \leq m$, either $(\sigma_j^1, \sigma_j^2)=(j,j+1)$ for all $j \in N_i$, or  $(\sigma_j^1, \sigma_j^2)=(j+1,j)$ for all $j \in N_i$. Thus, if $(\sigma_j^1, \sigma_j^2)=(j,j+1)$ for $j \in N_i$, then $(\sigma_l^1, \sigma_l^2)=(l+1,l)$ for all $l \in N_{i+1}$.

Let $(l, \overline{l}) \in I_j \times \overline{I}_j$ be a dual pair. Suppose $(\sigma_j^1, \sigma_j^2)=(j,j+1)$, then $\overline{\phi}_j=\overline{\phi}_{j+1} \circ \overline{\alpha}_j$. Consequently, there exists a unique element $\overline{l}' \in \overline{I}_{j+1}$ such that $b_{\overline{l}'} \circ \overline{\alpha}_j$ is a multiple of $b_{\overline{l}}$. Since $\overline{\alpha}_j$ is an arrow morphism, it follows that $\overline{l}$ and $\overline{l}'$ correspond to consecutive intersections $p, p' \in \tau^{-1}\Delta \overrightarrow{\cap} \gamma$. Let $p \in \tau^{-1}\delta_j \overrightarrow{\cap} \tau^{1}\delta_{j+1}$ such that  $\overline{\alpha}_j$ is a multiple of $\basis(p)$. Corollary \ref{CorollaryCharacterizationArrowMorphisms} implies that $\alpha_j$ equals $\basis(\tau(p))$ up to scalar. Since $\alpha_j \circ b_l \neq 0$ and $\phi_{j+1}= \alpha_j \circ \phi_{j}$, it follows that there exists $l' \in I_{j+1}$ such that $b_{l'}$ is a multiple of $\basis(\mathbb{S}^{-1}(p'))$. By construction, $(l', \overline{l}')$ is dual pair. In a similar way, we construct such a dual pair $(l', \overline{l}') \in I_{j+1} \times \overline{I}_{j+1}$ if $(\sigma_j^1, \sigma_j^2)=(j+1,j)$. Vice versa, we can construct a dual pair in $I_j \times \overline{I}_j$ from a dual pair $I_{j+1} \times \overline{I}_{j+1}$. In this way, we reconstruct the sequence of intersections of $\gamma$ with $\Delta$ from a sequence of dual pairs (and the intersections are those associated with the dual pairs).

By Lemma \ref{LemmaTriangulationSeparatesPuncturesLoopAroundPuncture}, our assumptions guarantee that at least one of the objects $Y_j$ is perfect or that $\gamma$ is a loop around a puncture.  In the former case, there exists a dual pair $(l,\overline{l})$ from which we can reconstruct the sequence of intersections of $\gamma$ and $\Delta$. In the latter case, $N=N_0$ and and the homotopy class of $\gamma$ is determined by any of the arcs $\delta_j$. 
\end{proof}

\medskip

\subsection{Properties of diffeomorphisms induced by equivalences}

\subsubsection{Induced diffeomorphisms preserve the orientation}\label{SectionPreserveOrientation}
\ \smallskip

\noindent It turns out that the diffeomorphism associated to a triangle equivalence preserves the orientation as a consequence of covariance.

\begin{prp}\label{PropositionPsiYieldsOrientationPreservingdiffeomorphisms}
	Let $T:\T \rightarrow \T'$ be a triangle equivalence. Then, $\Psi(T)$ preserves the orientation.
\end{prp}
\begin{proof}Write $\mathcal{S}=(S, \marked)$. Let $\Delta=\{\gamma_1, \dots, \gamma_m\}$ be an extended triangulation of $(S, \marked \cap {\, \partial S})$ and let $\{X_1, \dots, X_m\} \subset \T$ be a representing set of objects for $\Delta$. By assumption, $\Delta$ contains no self-folded triangles. We claim the following. If  $\gamma_{\sigma_0}, \gamma_{\sigma_1}, \gamma_{\sigma_2}$ are edges of a triangle in $\Delta$, then the given order coincides with the clockwise order in the orientation if and only if there exists an arrow morphism in at least one of the sets $\Hom^*(X_{\sigma_i},X_{\sigma_{i+1}})$, where $i \in \{0,1,2\}$ (indices modulo $3$). Recall that, by Corollary \ref{CorollaryCharacterizationArrowMorphisms}, arrow morphisms (up to a scalar) between objects corresponding to a triangle are in bijection with the corners of said triangle. Let $U$ be a triangle of $\Delta$ and let $\widetilde{U}$ be a lift of $U$ to the universal cover of $S$. Then $\widetilde{U}$ is an embedded triangle and at least two sides of it are arcs. Such a pair of arcs has a unique boundary intersection. Moreover, this intersection defines a morphism $X_i \rightarrow X_j$ (which is an arrow morphism)  if and only if $\gamma_i$ comes immediately before $\gamma_j$ in clockwise order.\\ To prove the claim it is sufficient to show that $\widetilde{U}$ is not bounded by more than one lift of the same arc $\gamma_i$. Suppose this was not the case, then the two lifts, which we denote by $\delta$ and $\delta'$, intersect at the boundary in a point $p$. By uniqueness, $p$ is the start point of $\delta$ and the end point of $\delta'$ or vice versa. But this implies that $U$ contains an embedded Möbius strip in contradiction to the orientability of $\mathcal{S}$.\\

	The assertion follows from the fact that any diffeomorphism $H: S \rightarrow S$  preserves the orientation if it preserves the order of the edges of each triangle in $\Delta $. However, this follows from the fact that equivalences preserve arrow morphisms.
\end{proof}

\begin{proof}[Proof of Theorem \ref{TheoremPsiRealizesAction}] Let $X \in \T$ be indecomposable, let $\mathcal{X}$ be any triangulation of $\T$ and denote by $\Delta\coloneqq \gamma(\mathcal{X})$ the associated triangulation of $\mathcal{S}_{\T}$. By Corollary \ref{CorollaryPsiRealizesActionOnEssentialObjects}, we may assume that $X$ is $\mathcal{X}$-admissible. The corollary also implies that
the triangulation $\Psi(T)(\Delta)$ of $\mathcal{S}_{\T'}$ coincides with $\gamma(T(\mathcal{X}))$.
By Lemma \ref{LemmaCharacteristicSequencePreservedUnderEquivalences}, every characteristic sequence of $X$ with respect to $\mathcal{X}$ is mapped to a characteristic sequence of $T(X)$ with respect to $T(\mathcal{X})$. In particular, the characteristic sequence of objects $Y_0, \dots, Y_m$ of $X$ is mapped to a characteristic sequence $T(Y_0), \dots, T(Y_m)$ of $T(X)$ and Proposition \ref{PropositionCharacteristicSequenceDefineObjectTauInvariant} and Proposition \ref{PropositionPsiYieldsOrientationPreservingdiffeomorphisms} imply $\Psi(T)\left(\gamma_{\ast}(X)\right)=\gamma_{\ast}(T(X))$.\ \smallskip
\end{proof}

\subsubsection{Induced diffeomorphisms preserve winding numbers}
\ \smallskip

\noindent As an application of the dichotomy of interior and boundary morphisms, we present a categorical characterization of the winding number function $\omega=\omega_{\eta}$ and show that the diffeomorphism of an equivialence between surface-like categories preserves winding numbers.

\begin{lem}\label{LemmaBoundaryMorphismWindingNumber}
Let $X \in \T$ be indecomposable and suppose $\gamma \in \gamma(X)$ is a finite closed arc in minimal position such that the index of the self-intersection at the boundary is $1$. If $f: X \rightarrow X[d]$ is a non-invertible boundary morphism, then $d= \omega(\gamma) + 1$, where $\sigma$ is the index of the self-intersection of $\gamma$ at its end points.
\end{lem}
\begin{proof}
Since $f$ is not invertible and a boundary morphism, any standard decomposition of $f$ contains a morphism corresponding to the self-intersection $p$ of $\gamma$ at its two ends. This follows from Proposition \ref{PropositionInteriorMorphismsVectorSpace} and Lemma \ref{LemmaBoundaryPointsAreInvertible}. In particular, $d$ is uniquely determined as the degree of this morphism and the stated formula follows Lemma \ref{LemmaWindingNumberPiecewiseSmoothCurves}.
\end{proof}

\noindent As an application of Theorem \ref{TheoremPsiRealizesAction} (which we proved in the previous paragraph) we show the following.
\begin{prp}\label{PropositionPsiPreservesWindingNumbers}Let $(\T, \mathcal{S}, \eta, \gamma, \basis)$ and  $(\T', \mathcal{S}', \eta', \gamma', \basis')$  be surface-like quintuples and let $\mathcal{S}$ be non-special. Let further $T:\T \rightarrow \T'$ be a triangle equivalence. Then, $\omega_{\eta'} \circ \Psi(T)=\omega_{\eta}$. In particular, $\Psi(T)$ maps $\eta$ to a line field in the same homotopy class as $\eta'$. 

\end{prp}
\begin{proof}We write $\omega=\omega_{\eta}$ and $\omega'=\omega_{\eta'}$. As discussed in Section \ref{SectionLineFieldsWindingNumbers}, the homotopy class of $\eta$ (and similar $\eta'$) is completely determined by the values of $\omega$ on unobstructed loops. Let $\gamma$ be an unobstructed loop on $\mathcal{S}$ and let $\delta$ be any closed finite arc whose associated unobstructed loop is homotopic to $\gamma$. Let $X=X_{\delta} \in \T$. First, assume that $X$ is not $\tau$-invariant. Then, there exists a non-invertible boundary morphism $f: X \rightarrow X[d]$ and $\omega(\gamma)= \sigma \cdot (d-1)$ by Lemma \ref{LemmaBoundaryMorphismWindingNumber}, where $\sigma$ is the index of the boundary self-intersection of $\delta$. It follows from Theorem \ref{TheoremPsiRealizesAction} that $T(X)$ is represented by $\Psi(T)(\delta)$. As $\Psi(T)$ preserves the orientation (Proposition \ref{PropositionPsiYieldsOrientationPreservingdiffeomorphisms}), $\sigma$ is also the index of the unique boundary self-intersection of $\Psi(T) \circ \delta$. Finally, $T$ sends boundary morphisms to boundary morphisms (Lemma \ref{LemmaEquivalencesPreserveInteriorMorphisms}) and it follows that  
\[\omega'(\Psi(T) \circ \gamma)=\omega'(\Psi(T) \circ \delta)=\sigma \cdot (d-1)=\omega(\delta)=\omega(\gamma).\]

Finally, suppose that $X$ is $\tau$-invariant. Then, $\omega(\gamma)=0$. This is clear, if $X$ is a loop object and if $X$ is an arc object, then $\omega(\gamma)=0$ by Proposition \ref{PropositionFractionallyCalabiYauArcs}. The object $T(X)$ is $\tau$-invariant and, by Theorem \ref{TheoremPsiRealizesAction}, is represented by a curve in $\gamma_{\ast}(\Psi(T) \circ \gamma)$ and as before we conclude that $\omega'(\Psi(T) \circ \gamma)=0$.
\end{proof}

\medskip

\subsection{Diffeomorphisms of spherical twists are Dehn twists}\label{SectiondiffeomorphismDehntwists}
\ \smallskip

\noindent Suppose $\T$ is a surface-like category wich admits a DG-enhancement of its triangulated structure.  Rephrasing the original definition in \cite{SeidelThomas}, we recall that an \textbf{$m$-spherical object} in $\T$ is an object $X \in \T$ such that $\tau X \cong X[m-1]$ 
and there exists an isomorphism of $\mathbb{Z}$-graded $k$-algebras

\[\Hom^*(X,X) \cong \faktor{k[z]}{z^2},\]

\noindent where $z$ is a generator of degree  $m$.
Under the above hypotheses, $X$ induces an auto-equivalence $T_X$ of $\T$ called the \textbf{spherical twist} of $X$. Note that $T_X(Y)$ it is defined up to (non-unique) isomorphism by a distinguished triangle

\triangle[\Hom^*(X,Y) \otimes X][Y][T_X(Y)][\text{ev}][][][1]
\ \smallskip

\noindent Note that by Property 2) of Definition \ref{DefinitionFukayaLikeQuintuples}, a spherical object of a Fukaya like category is always represented by a simple finite arc or loop, i.e.\ a curve without interior self-intersections.

\subsubsection{Dehn twists}
\ \smallskip The definition of spherical twists in \cite{SeidelThomas} seems to be largely influenced by the concept of Dehn twists which are typical examples of a mapping classes. \medskip

\noindent The class of a \textbf{Dehn twist} $D_{\gamma}$ about an  (oriented) simple loop $\gamma$ on $S$ is defined as follows.  If $W$ is a  tubular neighborhood of $\gamma$, i.e.\@ a neighborhood $W$ with a diffeomorphism $\phi:S^1 \times [-1,1] \rightarrow W$, such that $\phi|_{S^1 \times \{0\}}=\gamma$, then $D_{\gamma}: S \rightarrow S$ is defined by $D_{\gamma}|_{S \setminus W}:=\text{Id}_{S \setminus W}$ on $S \setminus W$ and on $W$ by 
\[D_{\gamma}(\phi(z, t)):=\phi(z \cdot e^{\pi i (t+1)}, t).\]
While $D_{\gamma}$ depends on $W$, its  mapping class is well-defined. \medskip

\noindent On the level of homotopy classes, the twist $D_{\gamma}$ sends the homotopy class of a curve $\delta$, which is in minimal position with $\gamma$, to the homotopy class of the curve obtained by resolving all intersections of $\gamma$ and $\delta$ at once, following the orientation of $\gamma$. \label{DehnTwistsViaSurgery} That is, whenever $\delta$ crosses $\gamma$ from the right (resp.\@ left) of $\gamma$, we turn right (resp.\@ left) at the intersection.

\noindent We prove the following:
\begin{thm}\label{TheoremSphericalTwistsDehnTwists}
Assume that the surface of $\T$ is non-special and let $X \in \T$ be spherical. If $\gamma(X)$ contains a loop, then $\Psi(T_X)$  is the mapping class of the Dehn twist around a simple loop $\gamma$. Otherwise, $\gamma(X)$ contains the boundary segment of a component $B$ and $\Psi(T_X)$ is in the mapping class of the Dehn twist around $B$.
\end{thm}
\begin{proof}Let $\Delta$ be any triangulation of $\mathcal{S}$. If $\gamma$ is a simple loop in the homotopy class of $\gamma(X)$, then by \cite{Disarlo}, Lemma 4.1., $\Psi(T_X) \simeq D_{\gamma}$ if and only if $H(\delta) \simeq D_{\gamma}(\delta)$ for all $\delta \in \Delta$. We conclude from the correspondence between mapping cones and resolution of crossings that $H$ and $D_{\gamma}$ are isotopic if $\iota(\gamma, \delta) \leq 1$ for all $\delta \in \Delta$. Lemma \ref{LemmaSimpleLoopGoodTriangulation} below asserts that there always exists a triangulation $\Delta$ with this property.
\end{proof}

\begin{rem}
Theorem \ref{TheoremSphericalTwistsDehnTwists} extends to the case of special surfaces which follows from Theorem \ref{TheoremSpecialSurfaces}.
\end{rem}

\begin{lem}\label{LemmaSimpleLoopGoodTriangulation}
Let $\mathcal{S}$ be a simple loop. Then, there exists a triangulation $\Delta$ of $\mathcal{S}$ such that $\iota(\gamma, \delta) \leq 1$ for all $\delta \in \Delta$.
\end{lem}
\begin{proof}
Denote $\mathcal{S}_{\gamma}$ the compact surface obtained by cutting at $\gamma$ and $\mathcal{S}'$ the surface obtained from $\mathcal{S}_{\gamma}$ by gluing discs to the boundary components $B_0, B_1 \subseteq \partial \mathcal{S}_{\gamma}$ corresponding to $\gamma$. Let $p$ be any point on $\gamma$ and denote by $q_i \in B_i$ the corresponding point. Then $(\mathcal{S}_ {\gamma}, \marked \cup \{q_1, q_2\})$ has a triangulation $\Delta_{\gamma}$ such that if $\delta \in \Delta_{\gamma}$  has both end points in the set $\{q_0, q_1\}$, then $\delta$ is closed and moerover every two closed arcs of such kind share the same end point. To see that such a triangulation $\Delta_{\gamma}$ exists, consider a triangulation $\Delta'$ of $(\mathcal{S}', \marked)$, where we consider the empty set as a triangulation of a surface without boundary and without punctures. Regarding $\mathcal{S}_{\gamma}$ as a subsurface of $\mathcal{S}'$ in the natural way we may deform the arcs of $\Delta'$ and assume them to be contained in $\mathcal{S}_{\gamma}$ (the resulting homotopy classes in $\mathcal{S}_{\gamma}$ are not canonical). It is possible to extend $\Delta'$ to a triangulation $\Delta_{\gamma}$ with the desired properties which boils down to showing that a triangle with one or two boundary components in its interior (each of which contains a single marked point) has a triangulation with the desired properties. This is similar to the proof of the existence of triangulations which separates punctures, see Lemma \ref{LemmaSeparatingPunctureTriangulationExist}. 

Note that if $\Delta_{\gamma}$ contains a closed arc $\delta$ connecting $q_i$ with itself, then the component of $\mathcal{S}' \supset \mathcal{S}_{\gamma}$ which contains $\delta$ is the unique component of $\mathcal{S}'$ without boundary and without punctures. Therefore, all closed arcs of $\Delta_{\gamma}$ share the same end points.
We regard the arcs of $\Delta_{\gamma}$ as arcs of the surface $(\mathcal{S}, \marked \cup \{p\})$. 

In the final step, we transform $\Delta_{\gamma}$ into a triangulation $\Delta$ of $(\mathcal{S}, \marked)$ as in the assertion.  We produce a finite sequence $\Delta^0, \dots, \Delta^m$ of collections of arcs as follows. Set $\Delta^0\coloneqq\Delta_{\gamma}$. Denote by $W_i^j$ the set of arc segments of arcs $\delta \in \Delta^j \subseteq \mathcal{S}_{\gamma}$ which start or end on $q_i$. Then $W_i^j$ is linearly ordered by the orientation of $\gamma$ analogous to the definition of oriented intersections at the boundary.\\
Suppose $\epsilon_i \in W_i^j$ is minimal and that $\epsilon_{1-i} \in W_{i-1}^j$ is maximal. Then, for some $l \in \{0,1\}$, $\epsilon_l$ is not closed. Let $\delta \in \Delta^j \subseteq \mathcal{S}$ denote the unique arc containing $\epsilon_{1-l}$ by the unique concatenation, denoted by $\delta'$, of $\delta$ and the arc containing  $\epsilon_l$ such that $\epsilon_0$ and $\epsilon_1$ are subarcs of $\delta'$. Set $\Delta^{j+1}\coloneqq \left(\Delta^j \setminus \{\delta\}\right) \cup \{\delta'\}$. We observe that $\delta' \in W^{j+1}_0 \cup W^{j+1}_1$ or $\delta'$ has end points in $\marked$.
Repeating the previous step as long as possible, we find a collection $\Delta^m$ such that $W^j_m= \emptyset$ and $|W^m_j| =1$ for some $j \in \{0,1\}$. In other words we replaced every arc of $\Delta_{\gamma}$  with end point $p \in \mathcal{S}$  by an arc of $(\mathcal{S}, \marked)$ except for one arc $\epsilon$. Set $\Delta\coloneqq \Delta^m \setminus \{\epsilon\}$.\\
 We claim that $\Delta$ is a triangulation of $(\mathcal{S}, \marked)$ with the desired properties. First of all, it is a subset of a triangulation, since we may deform the new arc in $\Delta^{j+1}$ such that it has disjoint interior from all arcs in $\Delta^j$. Second, it is clear that each new arc crosses $\gamma$ exactly once. Suppose there exists an essential arc $\beta$ in $(\mathcal{S}, \marked)$ the interior of which is disjoint from all other arcs in $\Delta$. Then $\beta$ crosses $\gamma$ since $\Delta_{\gamma}$ is a triangulation of $(\mathcal{S}_{\gamma}, \marked \cup \{q_0, q_1\})$ and $\beta$ is homotopic to the concatenation of arcs $\delta_0 \in W_0^0$ and $\delta_1 \in W_1^0$. Let $n \geq 0$ be maximal such that $\delta_j \in W^n_j$ for all $j \in \{0,1\}$. Then $\beta$ is homotopic to the new arc $\delta' \in \Delta^n$ constructed by concatenation of arcs of $W^{n-1}_0$ and $W^{n-1}_1$. Since $\delta' \not \in W^n_0 \cup W^n_1$, the arc $\delta' \simeq \beta$ is homotopic to an arc in $\Delta$.\end{proof}
 \ \smallskip

\subsection{Surface-like categories of special surfaces}\label{SectionSpecialSurfaces}
\noindent We extend our previous results to the case of surface-like categories modeled by special surfaces. Recall from Definition \ref{DefinitionSpecialSurface} that a marked surface is special if its arc complex is empty or has dimension at most $1$.\smallskip

\noindent The following is a complete list of special marked surfaces and is taken from \cite{Disarlo}, Figure 1.
A marked surface $\mathcal{S}=(S, \marked)$ is special if and only if 

\begin{itemize}
\setlength\itemsep{1ex}

	\item $S$ is a disc with no punctures and $|\marked| \leq 5$, or
	\item  $S$ is a disc with one puncture and at most two marked boundary points, or
	\item $S$ is a cylinder with no puncture and  a single marked point on each boundary component. 

\end{itemize}

\noindent Special surfaces occur as surface models of prominent examples of gentle algebras. The surface of a quiver of type $A_n$ is a disc with $n+1$  marked points on the boundary and no punctures and hence is special for $n \leq 4$. 

The second case in the list above is obtained as the surface of the algebra of dual numbers. As shown in Lemma \ref{LemmaP1FamiliesKroneckerOnly}, the Kronecker is the only  gentle algebra which realizes the third entry of the previous list.\medskip

\noindent We shall prove the following extension of Theorem \ref{TheoremPsiRealizesAction}, Proposition \ref{PropositionPsiYieldsOrientationPreservingdiffeomorphisms} and Proposition \ref{PropositionPsiPreservesWindingNumbers}.

\begin{thm}\label{TheoremSpecialSurfaces}Let $(\T, \mathcal{S}, \eta, \gamma, \basis )$ and $(\T', \mathcal{S}', \eta', \gamma, \basis)$  be surface-like quintuples and assume that $\mathcal{S}$ is special. Then, the following is true.

\begin{enumerate}

\setlength\itemsep{1ex}

\item If $\T$ and $\T'$ are equivalent, then $\mathcal{S} \cong \mathcal{S}'$.
\item For every every $T \in \Aut(\T)$, there exists a unique mapping class $\Psi(T) \in \MCG(\mathcal{S})$ such that \smallskip
\begin{itemize}
\item[a)] $\omega \circ \Psi(T)= \omega$, and
\item[b)]  for every indecomposable object $X \in \T$, 
\[\Psi(T)\left(\gamma_{\ast}(X)\right)= \gamma_{\ast}(T(X)).\]

\end{itemize}
\item The mapping $T \mapsto \Psi(T)$ is a group homomorphism.

\end{enumerate}

\end{thm}

\noindent In what follows we describe the group of diffeomorphisms of $\mathcal{S}$ up to isotopy  and the kernel of the group homomorphism from Section \ref{SectionIsomorphismsArcComplexesdiffeomorphisms},
\[\begin{tikzcd}\Phi:\Diffeo(\mathcal{S}, \mathcal{S}) \arrow{r} & \Simp(A_{\ast}(\mathcal{S}), A_{\ast}(\mathcal{S})).\end{tikzcd}\] 
Note that the mapping class group is a normal subgroup of index $2$ in $\Diffeo(\mathcal{S}, \mathcal{S})$. Its non-trivial coset contains the isotopy classes of orientation-reversing diffeomorphisms. In particular, $\Diffeo(\mathcal{S}, \mathcal{S})$ is generated by the mapping class group of $\mathcal{S}$ and any orientation-reversing diffeomorphism. Our main reference for computing mapping class groups is \cite{FarbMargalit}.
\begin{enumerate}
\setlength\itemsep{1ex}
	\item If $\mathcal{S}$ is a disc with $4$ marked boundary points, then $A_{\ast}(\mathcal{S})$ consists of two disconnected points, $\MCG(\mathcal{S})$ is generated by 
	$\tau$  and $\ker \Phi \cap \MCG(\mathcal{S})$ is generated by $\tau^2$. The kernel of $\Phi$ is generated by two commuting reflections.
	
	\item If $\mathcal{S}$ is a disc with $5$ marked boundary points, then $\MCG(\mathcal{S})$ is generated by $\tau$ and $\Diffeo(\mathcal{S}, \mathcal{S}) $ is the dihedral group $D_5$. The arc complex of $\mathcal{S}$ is a $5$-gon. Since $\Phi$ is surjective, it is a bijection. 
	
	\item If $\mathcal{S}$ is a disc with one puncture and one marked boundary point, then $A_{\ast}(\mathcal{S})$ is a point and $\Simp(A_{\ast}(\mathcal{S}), A_{\ast}(\mathcal{S}))$ and $\MCG(\mathcal{S})$ are both trivial. 
	Assuming that the puncture coincides with the center of the disc, $\Diffeo(\mathcal{S}, \mathcal{S})$ is generated by the reflection through the line connecting the marked points.
	
	\item If $\mathcal{S}$ is a disc with $2$ marked points on its boundary and one puncture, then $A_{\ast}(\mathcal{S})$ is a graph of type $A_4$ and its simplicial automorphism group is isomorphic to $\mathbb{Z}/2\mathbb{Z}$ the generator given by the reflection. 
	Assuming that the puncture is the center of the disc and that all marked points lie on a single line, $\Diffeo(\mathcal{S}, \mathcal{S})$ is generated by the reflection $\rho$ at said line and $\tau$ and $\ker \Phi$ is generated by $\rho$.
	
	The mapping class group of $\mathcal{S}$ has order $2$ and is generated by $\tau$. In particular, the restriction of $\Phi$ to $\MCG(\mathcal{S})$ is an isomorphism.
	\item If $\mathcal{S}$ is a cylinder with a single marked point on each boundary component, then, $A_{\ast}(\mathcal{S})$ is an infinite line of vertices and has automorphism group $\mathbb{Z} \rtimes \mathbb{Z}/{2\mathbb{Z}}$. Identifying $\mathcal{S}$ with $(S^1 \times [0,1], \{(1,0), (1,1)\})$, we see that $\MCG(\mathcal{S})$ is generated by

	$\tau_B$ for any of its boundary components $B$ and the map $\rho$ given  by $\rho(z, t)=(-z, 1-t)$ which permutes the boundary components of $\mathcal{S}$. Thus, $\MCG(\mathcal{S}) \cong \mathbb{Z} \rtimes \mathbb{Z}/{2\mathbb{Z}}$. Moreover, $\ker \Phi$ clearly contains $\rho$ and it follows by comparing cardinalities that $\Phi$ has kernel of order $2$. 

\end{enumerate}
\
\begin{proof}[Proof of Theorem \ref{TheoremSpecialSurfaces}]  Let $T: \T \rightarrow \T'$ be a triangle equivalence. We distinguish between three cases.
\begin{itemize}
\setlength\itemsep{1ex}
\item[a)] Assume that $A_{\ast}(\mathcal{S}) = \emptyset$, i.e.\ $\mathcal{S}$ is a disc with $1$, $2$ or $3$ marked points and no punctures. Every arc on $\mathcal{S}$ corresponds a segment object. Note that two boundary segments are neighbors if and only if there is a non-zero morphism between their representing objects in the orbit category. It follows, that there is a unique rotation $H$ of the disc satisfying $H \circ \delta \in \gamma(T(X_{\delta}))$.  Note that every mapping class of $\mathcal{S}$ is representable by such a rotation in the above case. The boundary segments for a cycle of arcs and characterization of winding numbers in Lemma \ref{LemmaBoundaryMorphismWindingNumber} implies that $H$ is compatible with $\omega$. Therefore $\Psi(T):=H$ is the desired diffeomorphism.

\item[b)] Assume that $A_{\ast}(\mathcal{S}) \neq \emptyset$ and  that $\mathcal{S}$ has an extended triangulation $\overline{\Delta}$ which contains a triangle which is not self-folded. As can be verified by hand, this includes all cases of special surfaces not mentioned in a) with the exception of the once-punctured disc with a single marked point on its boundary. Theorem \ref{TheoremDisarlo} shows that there exists some diffeomorphism $H:\mathcal{S} \rightarrow \mathcal{S}'$ such that $H \circ \delta \in \gamma(T(X_{\delta}))$ for all essential arcs $\delta \in \overline{\Delta}$. In particular, since the proof of Proposition \ref{PropositionCharacteristicSequenceDefineObjectTauInvariant} does rely on the the assumption that $\mathcal{S}$ is not special, it follows $H \circ \gamma \in \gamma_{\ast}(T(X_{\gamma}))$ for all curves $\gamma$ which are not boundary segments. Via case-by-case analysis, our description of $\ker \Phi$ above, and the fact that $\overline{\Delta}$ contains no self-folded triangle, we see that the requirement for this relationship to be true for all boundary segments $\gamma$, forces $H$ to preserve the orientation and to be uniquely determined by Property (2) b) in all cases but the punctured disc with $2$ marked points on the boundary.  However, in this case, every element of $\ker \Phi$ acts on the boundary segments in a trivial way and we may choose $H$ to be orientation preserving and this choice makes $H$ unique. As in case a), the winding number function of $\mathcal{S}$ and $\mathcal{S}'$ is determined by a cycle of boundary segments and since $H$ preserves the orientation, it is also compatible with the winding number functions.

\item[c)] If $\mathcal{S}$ is a disc with one puncture and one marked boundary point, we have seen before that the restriction of $\Phi$ to $\MCG(\mathcal{S})$ is an isomorphism. Because there is only one boundary segment it follows as in case b) of this proof that there exists a unique orientation preserving diffeomorphism $H:\mathcal{S} \rightarrow \mathcal{S}'$ such that $H \circ \gamma \in \gamma_{\ast}(T(X_{\gamma}))$ for all curves $\gamma$. As in case a), it follows that $H$ is compatible with the winding number functions. 

\end{itemize}
\end{proof}

\section{The kernel of \texorpdfstring{$\Psi$}{Psi}} \label{SectionKernelOfPsi}

\noindent In this section, we study the kernel of 
\[\begin{tikzcd} \Psi:\Aut(\mathcal{D}^b(A)) \arrow{r} & \MCG(\mathcal{S}_A),\end{tikzcd}\]
associated with the surface-like quintuple $(\mathcal{D}^b(A), \mathcal{S}_{A}, \eta_A, \gamma, \basis)$ of a gentle algebra $A$. The most satisfying result we obtain for the class of \textbf{triangular} gentle algebras, i.e.\ connected gentle algebras associated with quivers without oriented cycles. The main result of this section is the following Theorem.

\begin{thm}\label{TheoremKernelOfPsi}
	Let $A$ be a triangular gentle algebra. Then, $\ker \Psi $ is a subgroup of the group of outer automorphisms of $A$. Moreover, $\ker \Psi$ is generated by the groups of \textit{rescaling equivalences} and \textit{coordinate transformations}.

\end{thm}

\noindent Set $\mathcal{D}:=\mathcal{D}^b(A)$. By Theorem \ref{TheoremPsiRealizesAction}, $\ker \Psi(T)$ consists of all $T$ such that $\gamma_{\ast}(X)=\gamma_{\ast}(T(X))$ for all indecomposable objects $X \in \T$. It follows:
\begin{lem}
The shift functor is an element of $\ker \Psi$. 
\end{lem}
\medskip

\subsection{Rescaling equivalences}\label{SectionRescalingEquivalences}

\noindent We present further  typical examples of auto-equivalences in the kernel of $\Psi$. The map which attaches to any algebra automorphism $\sigma: A \rightarrow A$ its corresponding equivalence
${_{\sigma}A} \otimes_{\mathbb{L}} - :\mathcal{D} \rightarrow \mathcal{D}$, defines a group homomorphism
  \[\begin{tikzcd}\mathcal{O}:\Aut_k(A) \arrow{r} & \Aut\left(\mathcal{D}\right).\end{tikzcd}\]
  
  \noindent  Two such equivalences associated to automorphisms $\sigma, \sigma'$ are naturally isomorphic as functors if and only if $_{\sigma}A \cong {_{\sigma'}A}$ as $A$-$A$-bimodules and  that $\ker \mathcal{O}$ coincides with the set $\Inn A$ of inner automorphisms of $A$, c.f.\ \cite{RoggenkampTaylor}, Chapter VII. Thus, $\mathcal{O}$ descends to an embedding of the group $\out(A)$ of outer automorphisms of $A$ into $\Aut(\mathcal{D})$.\medskip

 \noindent   As an important special case, suppose $f: A \rightarrow A$ is an algebra isomorphism that fixes every vertex of $Q$ and multiplies every arrow by a scalar, i.e.\ $f(\alpha)=\lambda_{\alpha} \cdot \alpha$ for some $\lambda_{\alpha} \in k^{\times}$ for all $\alpha \in Q_1$.
The set of such automorphisms forms a subgroup of $\Aut_k(A)$. We call the derived equivalences associated to elements of this subgroup \textbf{rescaling equivalences} and denote their generated subgroup of $\Aut(\mathcal{D})$ by $\mathcal{R}$. \medskip

\noindent There is a surjective homomorphism 

\begin{displaymath}\begin{tikzcd}
(k^{\times})^{Q_1} \arrow[ twoheadrightarrow]{r}{\mathcal{O}} &  \mathcal{R},  \end{tikzcd}
\end{displaymath}

\noindent mapping  $(\lambda_{\alpha})_{\alpha \in Q_1}$ to the rescaling equivalence associated with the automorphism of $A$ which multiplies every $\alpha$ with $\lambda_{\alpha}$ and leaves every vertex fixed. Rescaling automorphisms are special cases of \textit{linear changes of variables}, c.f.\ \cite{AsensioSaorin}.  \medskip

\noindent Next, we describe the structure of $\mathcal{R}$ explicitely.

\begin{lem}\label{LemmaExactSequenceRescalingEquivalences}
There exists a short exact sequence of groups
\begin{displaymath}\label{SesRescalingEquivalences}\begin{tikzcd}
0 \arrow{r} & \faktor{(k^{\times})^{Q_0}}{k^{\times}} \arrow{r}{\phi} & (k^{\times})^{Q_1} \arrow{r}{\mathcal{O}} & \mathcal{R} \arrow{r} & 0,
\end{tikzcd} \end{displaymath}
where the quotient on the left is taken with respect to the diagonal embedding and the map $\phi$  is defined by 
\[\phi\left(\overline{(\lambda_x)_{x \in Q_0}}\right):=(\lambda_{s(\alpha)}^{-1} \cdot \lambda_{t(\alpha)})_{\alpha \in Q_1}.\]

\noindent The map on the right hand side maps an element of $(k^{\times})^{|Q_1|}$ to its associated element in $\mathcal{R}$.
 \end{lem}
 \begin{proof}
We discussed above that the map on the right hand side is surjective.  The image of $\phi$ are the so-called \textit{acyclic characters} of $Q$, c.f.\ \cite{AsensioSaorin}, Proposition 10. Note that $\phi((\lambda_x)_{x \in Q_0})$ coincides with the automorphism of $A$ induced by conjugation with the element $\sum_{x \in Q_0}{\lambda_x x}$ and hence $\Img \phi \subseteq \Inn A$ is in the kernel of the map on the right hand side.  
If two tuples $(\lambda_{\alpha})_{\alpha \in Q_1}, (\mu_{\alpha})_{\alpha \in Q_1}$ induce isomorphic rescaling equivalences, the the automorphism associated with the tuple $(\lambda_{\alpha} \cdot \mu_{\alpha}^{-1})_{\alpha \in Q_1}$ is inner and, by the second short exact sequence of Theorem 15 in \cite{AsensioSaorin}, it is an acyclic character.
 \end{proof}
\noindent  The following shows that $\mathcal{R}$ is determined by $\mathcal{S}_A$ up to isomorphism.

 \begin{prp}
 There exist isomorphisms $\mathcal{R} \cong (k^{\times})^{1+ |Q_1| - |Q_0| } \cong (k^{\times})^{1-\chi(\mathcal{S}_A)}$.
\end{prp}

\begin{proof}The quiver $Q$ carries the structure of a simplicial complex with vertices given by elements of $Q_0$ and edges parametrized by the set of arrows $Q_1$.
The exact sequence (see Lemma \ref{LemmaExactSequenceRescalingEquivalences})
\begin{displaymath}\begin{tikzcd}
0 \arrow{r} & k^{\times} \arrow{r}{\Delta} & (k^{\times})^{Q_0} \arrow{r}{\phi} & (k^{\times})^{Q_1} \arrow{r}{} & 0 \arrow{r} & \cdots,
\end{tikzcd} \end{displaymath}
 where $\Delta$ denotes the diagonal embedding,
is isomorphic to the co-chain complex of the reduced simplicial cohomology of $Q$ with coefficients in $k^{\times}$. The isomorphism identifies the basis element $e_{\beta}=(\delta_{\alpha \beta})_{\alpha \in Q_1} \in (k^{\times})^{Q_1}$ with the element $e_{\beta}^{\ast}$ of the natural dual basis of $C^1(Q, k^{\times})=\Hom( \mathbb{Z}^{Q_1}, k^{\times})\cong (k^{\times})^{Q_1}$ and the basis element $e_{x}=(\delta_{x y})_{y \in Q_0}$ with the dual  basis element of $C^0(Q, k^{\times})=\Hom(\mathbb{Z}^{Q_0}, k^{\times})\cong (k^{\times})^{Q_0}$. Finally, the evaluation homomorphism  $\Hom(\mathbb{Z}, k^{\times}) \rightarrow k^{\times}$ provides the last isomorphism. It follows that $\mathcal{R} \cong \tilde{H}^1_{\text{Simp}}(Q, k^{\times})=\tilde{H}^1( |Q|, k^{\times})$, where $|Q|$ denotes the topological space associated with $Q$, i.e.\ its underlying graph. Thus, $\mathcal{R}$ is isomorphic to $(k^{\times})^c$ for some $c \geq 0$. As shown in the proof of \cite{OpperPlamondonSchroll}, Proposition 1.22, $\mathcal{S}_A$ and $|Q|$ are homotopy equivalent and therefore $c=1-\chi(|Q|)=1-\chi(\mathcal{S}_A)$. 

\end{proof}

\begin{lem}\label{LemmaRescalingIsInKernel}$\mathcal{R}$ is a subset of $\ker \Psi$.
\end{lem}
\begin{proof}
	Let $f \in \Aut(A)$ be a rescaling automorphism and $F=\mathcal{O}(f) \in \Aut(\mathcal{D})$ the corresponding equivalence. 
	Then $F$ maps projective modules to projective modules and since $f$ sends arrows to multiples of themselves, $F$ sends string complexes to string complexes with rescaled components of its differential. This does not change the isomorphism class of a string complex. For the same reasons, $\mathcal{F}$ may change the isomorphism classes of band complexes, but only to a band complex in the same family of band complexes (a family of $\tau$-invariant objects). In particular, the representing homotopy class of every indecomposable object in $\mathcal{D}$ remains unchanged under the action of $F$ showing that $\mathcal{F} \in \ker \Psi$.
\end{proof}
\noindent In a similar way one proves:
\begin{lem}
Let $T \in \Aut( \mathcal{D}^b(B))$ be an affine or projective coordinate transformation as defined in Section \ref{SectionFractionalCalabiYauTauInvariantFamilies}. Then, $T \in \ker \Psi$. 
\end{lem}

\noindent Note that in case of canonical algebras, rescaling equivalences are special cases of coordinate transformations.\medskip

\subsection{The proof of Theorem \ref{TheoremKernelOfPsi}}
The proof of Theorem \ref{TheoremKernelOfPsi} is split across the following two Lemmas. \medskip

\begin{lem}\label{LemmaShiftOfKernelIsOuter}Let $A$ be a triangular gentle algebra. For every  $T \in \ker \Psi$, there exists an integer $n \in \mathbb{Z}$, such that $T\circ [n] \in \out(A)$.
\end{lem}
\begin{proof}
	Let $T  \in \ker \Psi$. Since every indecomposable projective module is essential, it follows that $T$ preserves their isomorphism classes up to shift. For $P \in \T$ an indecomposable projective $A$-module, let $m_P \in \mathbb{Z}$ be such that $T(P) \cong P[m_P]$. If $P$ and $P'$ indecomposable projective $A$-modules and $\Hom(P,P')\neq 0$, then by virtue of the isomorphism 
	\[\Hom(T(P),T(P'))\cong \Hom(P, P'[m_{P'}-m_P]),\]
	
	\noindent  it follows $m_P=m_{P'}$. Since $Q$ is connected, it follows by induction that $m_P=m_{P'}$ for all indecomposable projective $A$-modules $P$ and $P'$ and we may assume that $T(P)\cong P$ for all projective $A$-modules $P$. Since $Q$ has no oriented cycles, it follows from Proposition 2.4 in \cite{Chen}  that $T$ is standard and naturally isomorphic to a derived tensor product $_{\sigma}A \otimes^{\mathbb{L}} -$ for some $\sigma \in \Aut_k(A)$.
\end{proof}

\begin{rem}
The arguments of the previous Lemma work for arbitrary connected gentle algebras and standard functors $T \in \ker \Psi$. We do not know whether all auto-equivalences of gentle algebras are standard.
\end{rem}

\noindent Together with Lemma \ref{LemmaShiftOfKernelIsOuter} the following lemma completes the proof of Theorem \ref{TheoremKernelOfPsi}.
\begin{lem}Let $A=kQ/I$ be a gentle algebra. Then, $\out(A) \cap \ker \Psi$ is the group which is generated by all coordinate transformations and rescaling equivalences. 
\end{lem}
\begin{proof}The inclusion ``$\supseteq $'' was shown in Lemma \ref{LemmaRescalingIsInKernel}. 

Let  $\sigma$ be an automorphism of $A$ such that $T \coloneqq\mathcal{O}(\sigma) \in \ker \Psi$.
	Since $T$ preserves the isomorphism class of every indecomposable projective $A$-module it follows from Lemma 20,\cite{AsensioSaorin}, that we may assume $\sigma(x)=x$ for all $x \in Q_0$ after composition with an inner automorphism. Note that since $A$ is gentle, it is isomorphic to the graded algebra associated with the filtration by powers of its radical.
	
For $\epsilon$ a linear combination of parallel paths in $Q$ starting in a vertex $x$ and ending in a vertex $y$, we denote by $P_{\epsilon}^{\bullet}$ the complex
			\begin{displaymath}
	\begin{tikzcd}
	\cdots \arrow{r} & 0 \arrow{r} & P_{y} \arrow{rr}{\epsilon} && P_{x} \arrow{r}  & 0 \arrow{r} & \cdots,
	\end{tikzcd}
	\end{displaymath}
	concentrated in degrees $0$ and $1$ (although the precise degrees are not important).
	A complex $P_{\epsilon}^{\bullet}$ such that $\epsilon \neq 0$ and $\Img \epsilon \subseteq \rad P_x$ as above is a string complex  if $\epsilon$ is a non-zero multiple of a single path and a band complex, otherwise.

	Suppose that for some $\alpha \in Q_1$, $\sigma(\alpha)$  and $\alpha$ are linearly independent.
	 Since $A$ is gentle, it follows that there exists at most one non-trivial path $\alpha \neq \beta \not \in I$ in $Q$ which is parallel to $\alpha$. As $\sigma(x)=x$ for all $x \in Q_0$, it follows $\sigma(\alpha)=a \cdot \alpha + b \cdot \beta$ for some $a, b \in k$ and $b \neq 0$. Note that, if $a=0$, then $T(P_{\alpha}^{\bullet})\cong P_{\beta}^{\bullet}$ is not isomorphic to a shift of $P_{\alpha}^{\bullet}$. Therefore, if $a=0$, it follows from $T \in \ker \Psi$ that the degenerated string complexes $P_{\alpha}^{\bullet}$ and $P_{\beta}^{\bullet}$ are $\tau$-invariant and must belong to arcs of the same $\simeq_{\ast}$-class, which implies that $Q$ is the Kronecker quiver by Lemma \ref{LemmaP1FamiliesKroneckerOnly}. 
	
	 If $a \neq 0$, then $b=0$ and then $\sigma(\alpha)$ is a multiple of $\alpha$, or $T(P_{\alpha}^{\bullet})$ is a band complex isomorphic to $P_{\sigma(\alpha)}^{\bullet}$

	implying that $P_{\alpha}^{\bullet}$ must be $\tau$-invariant as well. Thus, if $P_{\alpha}^{\bullet}$ is not $\tau$-invariant, then $b=0$ and $\sigma(\alpha)=a\cdot \alpha$.
	
	Next, suppose the arc object $P_{\alpha}^{\bullet}$ is $\tau$-invariant. Since $\End(P_{\alpha}^{\bullet}) \cong k$, Serre duality implies $\Hom^*(P_{\alpha}^{\bullet}, P_{\alpha}^{\bullet}) \cong k \oplus k[-1]$ and it follows from Proposition \ref{PropositionFractionallyCalabiYauArcs} that $P_{\alpha}^{\bullet}$ is represented by the  boundary segment on a component $B$ with a single marked point. This can also be seen easily from the construction of the bijection $\gamma$ described in Section \ref{SectionBijectionObjectsCurves}. It follows that the middle term in an Auslander-Reiten triangle starting in $P_{\alpha}^{\bullet}$ is indecomposable. By Corollary 6.3, \cite{Bobinski}, $\alpha$ is a maximal antipath, i.e.\ if $\delta \in Q_1$, then $\delta \alpha \not \in I $ and $\alpha \delta \not \in I$ whenever any such expression is defined.

The simple boundary loop around $B$ represents the family of complexes  $(\P_{u \alpha + v \beta})_{[u:v] \in \mathbb{P}^1_k}$. Since $T$ preserves this family of objects, there exists an invertible matrix $M \in \GL_2(k)$ such that
\begin{equation}\label{Equation1}\begin{pmatrix}\sigma(\alpha) \\ \sigma(\beta)\end{pmatrix}= M\begin{pmatrix}\alpha \\ \beta\end{pmatrix}.\end{equation}

	If $Q$ is the Kronecker quiver, it means that $T$ is isomorphic to a coordinate transformation.
	If not, then, $P_{\alpha}^{\bullet}$ is contained in an $\mathbb{A}_k^1$-family of $\tau$-invariant objects (Lemma \ref{LemmaP1FamiliesKroneckerOnly}) proving that $P_{\beta}^{\bullet}$ is not $\tau$-invariant. By our considerations above we know that $\beta$ is a path of length at least $2$ or $\sigma(\beta)$ is a multiple of $\beta$.
	
We assume that $\beta=\beta_l \cdots \beta_1$ for some $l \geq 2$ and arrows $\beta_i \in Q_1$. Suppose there exists an arrow $\delta \in Q_1$ such that $t(\delta)=s(\alpha)$. Then, the composition of $\delta$ and $\alpha$ is non-zero since $\alpha$ is a maximal antipath and $A$ is gentle. Since $\delta$ is not a maximal antipath, our previous arguments imply that $\P_{\delta}$ is not $\tau$-invariant and  as above it follows $\sigma(\delta)= d \delta$ for some $d \in k^{\times}$. Since, $\delta \alpha \neq 0$ and $0=\sigma(\delta \beta)= d \delta \sigma(\beta)$,  we conclude that $M$ lower triangular. Analogous arguments apply for arrows $\delta$ with $s(\delta)=t(\alpha)$. The action in (\ref{Equation1}) extends to an affine coordinate transformation $\rho$ of $A$ such that

\[\rho^{-1} \circ \sigma(x)= \begin{cases} x, & \text{if }x \in \{\alpha, \beta_1, \dots, \beta_l\}; \\ \sigma(x), & \text{otherwise.} \end{cases}\]

\noindent It follows that we arrive at a rescaling equivalence after a finite number of steps,

\end{proof}

\medskip

\section{The image of \texorpdfstring{$\Psi$}{Psi}}\label{SectionImageOfPsi}

\noindent We investigate the image of 
\[\begin{tikzcd} \Psi:\Aut(\mathcal{D}^b(A)) \arrow{r} & \MCG(\mathcal{S}_A),\end{tikzcd}\]
associated with a gentle algebra $A$.
We know from Proposition \ref{PropositionPsiPreservesWindingNumbers} that $\Img \Psi$ is contained in the stabilizer $\MCG(\mathcal{S}_A, \eta_A)$ of $\eta$ under the natural action of $\MCG(\mathcal{S}_A)$. The action of a diffeomorphism $H$ on $\eta$ is defined by pullback, i.e.\ for all $x \in \mathcal{S}_A$, $(H.\eta)(x) \coloneqq \mathbb{P}(TH)^{-1} \circ \eta (H(x))$.\medskip

\noindent We prove the following.

\begin{thm}\label{TheoremImageOfPsi}
The image of $\Psi$ is equal to $\MCG(\mathcal{S}_A, \eta_A)$.
\end{thm}
\begin{proof}Since the homotopy class of $\eta$ is completely determined by its winding number function $\omega=\omega_{\eta}$, $\MCG(\mathcal{S}_A, \eta_A)$ is the set of all mapping classes $H$ such  that $\omega \circ H=\omega$. In particular, Theorem \ref{IntroTheoremEquivalencesFukayaLikeCategories} and Proposition \ref{PropositionPsiPreservesWindingNumbers} show that $\Img \Psi$ is a subset of the the stabilizer of $\eta$.

Let  $\gamma_1, \dots, \gamma_n$ be finite simple arcs which represent all isomorphism classes  of indecomposable projective $A$-modules.  For the opposite inclusion, let $H \in \MCG(\mathcal{S})_{\omega}$, set $\delta_i \coloneqq H(\gamma_i)$ and let $T_i \in \mathcal{D}^b(A)$ be a representative of $\delta_i$. As in Lemma \ref{LemmaEndomorphismTiltingObject}, we may assume - after applying suitable powers of the shift  functor to each $T_i$ - that $T\coloneqq \bigoplus_{i=1}^{n}{T_i}$ is a tilting object and  $B\coloneqq\Hom(T,T)^{\text{op}} \cong A$. Moreover, the isomorphism can be chosen in such a way that the identity of $T_i$ is mapped to the idempotent of $A$ corresponding to $\gamma_i$. 

Due to a result of Rickard \cite{RickardMoritaTheory}, there exists an associated equivalence 
$F_H: \mathcal{D}^b(A) \rightarrow \mathcal{D}^b(A)$
such that $F_H(T_i) \cong P_i$ for all $i \in [1,n]$.

We claim that $\Psi(F_H)=H$. Indeed, $\Psi(F_H)(\gamma_i) \simeq \delta_i = H(\gamma_i)$ by Theorem \ref{TheoremPsiRealizesAction}. Note that the complement of all $\gamma_i$ is a disjoint union of discs, each of which contains at most one puncture. Since the mapping class of a disc with at most one puncture which fixes every marked point on its boundary is trivial, it follows that $\Psi(F_H)$ and $H$ agree on a triangulation of $\mathcal{S}_A$ which contains the arcs $\gamma_1, \dots, \gamma_n$. Thus, $\Psi(F_H)=H$.
\end{proof}

\section{Diffeomorphisms induce derived equivalences}\label{SectionHomeorphismsInduceEquivalences}

\noindent In this section we prove that gentle algebras with equivalent surface models are derived equivalent. It is a generalization of Corollary 3.2.4, \cite{LekiliPolishchukGentle}, to the case of (ungraded) gentle algebras of arbitrary global dimension and builds on similar ideas. The main result reads as follows.
\begin{thm}\label{TheoremEquivalentSurfacesImplyDerivedEquivalence}
Let $A$ and $B$ be gentle algebras and let $H: \mathcal{S}_A \rightarrow \mathcal{S}_B$ be a diffeomorphism of marked surfaces such that $\omega_A=\omega_B \circ H$. Then $A$ and $B$ are derived equivalent.
\end{thm}
\medskip

\subsection{Tilting complexes in derived categories of gentle algebras} We recall the definition of a tilting object.
\begin{definition}\label{DefinitionTiltingObject} Let $A$ be a finite dimensional algebra. An object $X \in \Perf(A)$ is a \textbf{tilting object} if all of the following conditions are satisfied:
	\begin{itemize}
	\setlength\itemsep{1ex}
		\item[1)] $\Hom^*(X,X)$ is concentrated in degree zero. 
		\item[2)] Let $\mathcal{T}$ denote the smallest triangulated subcategory of $\mathcal{D}^b(A)$ which contains $X$ and is closed under taking direct summands. Then $A \in \mathcal{T}$. 
	\end{itemize}
\end{definition}
\noindent We only consider tilting objects $X$ which are \textit{basic}, i.e.\@ the multiplicity of every indecomposable object in $X$ is at most one.
Next, we define what if means for a set of arcs to generate a surface.
\begin{definition}
A system $\mathcal{P}=\{\gamma_1, \dots, \gamma_m\}$ of arcs on a marked surface $\mathcal{S}$ \textbf{geometrically generates} $\mathcal{S}$ if for every finite arc $\delta$ on $\mathcal{S}$, there exists a sequence of finite arcs $\delta_1, \dots, \delta_m \simeq \delta$, such that $\delta_1 \in \mathcal{P}$ and such that for all $i \in (1, m)$, $\delta_{i+1}$ is the concatenation of $\delta_i$ and an arc in $\mathcal{P}$.
\end{definition}

\begin{lem}\label{LemmaGeometricalGeneratorsGentle} Let $A$ be a gentle algebra and let $\mathcal{P}=\{\gamma_1, \dots, \gamma_m\}$ denote a complete set of representatives of all indecomposable projective $A$-modules on $\mathcal{S}_A$. Then, $\mathcal{P}$ geometrically generates $\mathcal{S}_A$. 
\end{lem}
\begin{proof}
Every perfect string complex in $\mathcal{D}^b(A)$ is obtained by a sequence of mappings cones  with respect to ALP maps (see Section \ref{SectionMorphismsOfIntersection}) from or to indecomposable projective $A$-modules. Moreover, these maps correspond to boundary intersections and the assertion follows from the relationship between concatenations and mapping cones. 
\end{proof}

\noindent The following lemma provides a characterisation for a collection of arcs to represent a tilting object.
\begin{lem}\label{LemmaGeometricTiltingComplex}Let $\mathcal{P}=\{\gamma_1, \dots, \gamma_n\}$ be a set of pairwise non-homotopic curves on $\mathcal{S}_A$ in minimal position. Then, there exists a tilting object $X=\bigoplus_{i=1}^n{X_i}$  such that $\gamma_i \in \gamma(X_i)$ if and only if  the following conditions are satisfied:
	\begin{enumerate}
	\setlength\itemsep{1ex}
		\item For all $1 \leq i, j \leq n$, $\gamma_i$ and $\gamma_j$ are finite simple arcs with disjoint interior.
		\item The winding  number of every piecewise smooth loop formed by arcs in $\mathcal{P}$ is zero.
		\item $\mathcal{P}$ geometrically generates $\mathcal{S}_A$.
	\end{enumerate}
	
\end{lem}
\begin{proof}Suppose $X=\bigoplus_{i=1}^n{X_i}$ is a tilting object and $X_i$ is indecomposable with $\gamma_i \in \gamma(X_i)$. Since $X_i \in \Perf(A)$ and $\Hom^*(X,X)$ is concentrated in degree $0$, it is not $\tau$-invariant and hence $\gamma_i$ is a finite arc. Since every interior intersection of $\gamma_i$ and $\gamma_j$ gives rise to morphisms $f: X_i \rightarrow X_j[m]$ and $f': X_j \rightarrow X_i[1-m]$, the same reasoning shows that $\gamma_i$ and $\gamma_j$ have no interior intersections for all $i, j \in [1,n]$. In particular, $B:=\Hom(X,X)^{\text{op}}$ is gentle (c.f.\ Lemma \ref{LemmaEndomorphismTiltingObject} below or \cite{SchroerZimmermann}). Due to a result of Rickard \cite{RickardMoritaTheory}, there exists a derived equivalence $T: \mathcal{D}^b(B) \rightarrow \mathcal{D}^b(A)$ which sends the indecomposable projective $B$-modules to the objects $X_1, \dots, X_n$. Let $\mathcal{P}'$ be a  complete set of representatives of the indecomposable $B$-modules. Then, $\Psi(T)(\mathcal{P}')=\mathcal{P}$ up to homotopy. Applying  Lemma \ref{LemmaGeometricalGeneratorsGentle} to $\mathcal{P}'$, we conclude that $\mathcal{P}$ geometrically generates $\mathcal{S}_A$. Thus, $T$ satisfies conditions (1)-(3).

Next, suppose $T$ satisfies conditions (1)-(3). Let $\mathcal{U}:=\{U_1, \dots, U_n\} \subset \mathcal{D}^b(A)$ be a complete set of representatives of the arcs $\gamma_1, \dots, \gamma_n$ and let $i \in [1, n]$. Denote $Z(m) \subset \mathcal{U}$ the set consisting of all objects $U$ for which there exists a sequence  $U=U_{i_1}, \dots, U_{i_l}=U_m$ in $\mathcal{U}$ such that for all $j \in [1,l)$, $\Hom^*(U_{i_j}, U_{i_{j+1}})\neq 0$ or $\Hom^*(U_{i_{j+1}}, U_{i_{j}})\neq 0$.  Then either $Z(m)=Z(m')$ or $Z(m) \cap Z(m')=\emptyset$ for all $m, m' \in [1,n]$. Let $a_m \in \mathbb{Z}$. We specify $a_l$ for every $l \in [1,n]$ such that $U_l \in Z(m)$ as follows. Let  $U_m[a_m]=U_{i_0}[b_{i_0}], \dots, U_{i_s}[b_{i_s}]=U_l[b_{i_s}]$ and morphisms $f_0, \dots, f_{s}$, such that for each $j \in [0, s)$, $f_j$ is a morphism from $U_{i_j}[b_{i_j}]$ to $U_{i_{j+1}}[b_{i_{j+1}}]$ or vice versa which is associated to an intersection of $\gamma_{i_j}$ and $\gamma_{i_{j+1}}$.
	Set $a_l \coloneqq b_{i_s}$. Then, $a_l$ only depends on $a_m$. Namely, given another pair of sequences $U_{j_0}[c_{j_0}], \dots, U_{j_q}[c_{j_t}]$  and  $g_0, \dots, g_{t}$ as above, then, the sequences
	\[\begin{array}{cc}U_l[b_{i_s}]=U_{i_s}[b_{i_s}], \dots, U_{i_0}[b_{i_0}], U_{j_0}[c_{j_1}], \dots, U_{j_t}[c_{j_t}]=U_l[c_{j_t}]  &  f_s, \dots, f_{0} , g_0, \dots, g_{t}\end{array}\]		
	
\noindent	correspond to a piecewise smooth cycle of arcs in $T$ and the second condition together with Property 5) of surface-like categories implies $b_{i_s}=c_{j_t}$. Making choices of this kind for all sets $Z(m)$, the same argument shows that the $\Hom^*(X, X)$ of the object $X \coloneqq \bigoplus_{i=1}^n{U_i[a_i]}$ is concentrated in degree $0$.
	
The correspondence between mapping cones and resolutions of intersections implies by the second condition that the smallest triangulated subcategory of $\mathcal{D}^b(A)$, which is closed under direct summands and which contains $X$, also contains every perfect indecomposable object which is not $\tau$-invariant. In particular, it contains every indecomposable projective $A$-module. It shows that $X$ is a tilting object.
\end{proof}

\noindent Next, we show that the endomorphism ring of a tilting object can be recovered from a representing set of arcs.
\begin{lem}\label{LemmaEndomorphismTiltingObject}
	Let $X \in \Perf(A)$ be a tilting object and assume that $X=\bigoplus_{i=1}^n{X_i}$ ($X_i$ indecomposable) is represented by a set of arcs $\mathcal{P}=\{\gamma_1, \dots, \gamma_n\}$. Then, $\Hom(X,X)$ is isomorphic to the algebra $k\Gamma/R$, where $\Gamma$ is a quiver and $R$ is an ideal generated by quadratic zero relations, given as follows:
	\begin{itemize}
	\setlength\itemsep{1ex}
		\item $\Gamma$ has vertices $\{x_1, \dots, x_n\}$ and  the arrows from $x_i$ to $x_j$ are in one-to-one correspondence with the directed intersections  $p \in \gamma_i \overrightarrow{\cap} \gamma_j$, such that there is no other arc of $T$ ending between $\gamma_i$ and $\gamma_j$.
		\item $R$ is generated by all expressions $pq$, where $p$ and $q$ are composable arrows of $\Gamma$ and  $p \neq q$ as points in $S_A$. 
		
	\end{itemize}
\end{lem}
\begin{proof}There exists an algebra homomorphism $\varphi: k\Gamma/_{R} \rightarrow B$, which sends each vertex $x_i$ of $\Gamma$ to the identity morphism of $X_i$ which we regard as an element in $B$ in the natural way. We require also that $\varphi$ sends an arrow $p \in \gamma_i \overrightarrow{\cap} \gamma_j$ to a morphism $X_i \rightarrow X_j$ associated to $p$. It follows from the correspondences between intersections and morphisms on one hand and the geometric description of  compositions that $\varphi$ is surjective.  By comparing dimensions (which we can express in terms of intersections), we see that $\varphi$ is an isomorphism. 
\end{proof}
\begin{proof}[Proof of Theorem \ref{TheoremEquivalentSurfacesImplyDerivedEquivalence}]
Denote $\mathcal{P}$ as set of arcs representing the isomorphism classes of indecomposable projective $A$-modules. Since $A$ is tilting, the set $\mathcal{P}$ satisfies the conditions of Lemma \ref{LemmaGeometricTiltingComplex}. The assumptions on $H$ ensure that $H(\mathcal{P})$ enjoys the same properties. Thus, the arcs of $H(\mathcal{P})$ represent the indecomposable direct summands of a tilting object $Y \in \mathcal{D}^b(B)$ and it follows from Lemma \ref{LemmaEndomorphismTiltingObject} that $\Hom(Y,Y)^{\text{op}} \cong A$  as a $k$-algebra. 
\end{proof}
\begin{rem}
	Note that Lemma \ref{LemmaEndomorphismTiltingObject} provides a geometric proof of the result in \cite{SchroerZimmermann} that the class of  gentle algebras is closed under derived equivalences. \ \smallskip
\end{rem}
\medskip

\bibliographystyle{plain}

\bibliography{Bibliography_Yana}{}

\begin{thebibliography}{10}

\bibitem{AmiotPlamondonSchroll}
C.~Amiot, {P-G}. Plamondon, and S.~Schroll.
\newblock A complete derived invariant for gentle algebras via winding numbers
  and {A}rf invariants.
\newblock {\em arXiv:1904.02555 [math.RT]}, 2019.

\bibitem{Antipov2007}
M.~Antipov.
\newblock Derived equivalence of symmetric special biserial algebras.
\newblock {\em Journal of Mathematical Sciences}, 147(5):6981--6994, 2007.

\bibitem{ArnesenLakingPauksztello}
K.~Arnesen, R.~Laking, and D.~Pauksztello.
\newblock Morphisms between indecomposable complexes in the bounded derived
  category of a gentle algebra.
\newblock {\em Journal of Algebra}, 467:1--46, 2016.

\bibitem{AssemSkowronski}
I.~Assem and A.~Skowro\'{n}ski.
\newblock Iterated tilted algebras of type {$\tilde{A}_n$}.
\newblock {\em Mathematische Zeitschrift}, 195:269--290, 1987.

\bibitem{Auroux}
D.~Auroux.
\newblock Fukaya categories and bordered {H}eegaard-{F}loer homology.
\newblock In {\em Proceedings of the International Congress of Mathematicians
  2010}, volume~2, pages 917--941, 2010.

\bibitem{AvellaAlaminosGeiss}
D.~Avella-Alaminos and C.~Geiss.
\newblock Combinatorial derived invariants for gentle algebras.
\newblock {\em Journal of Pure and Applied Algebra}, 212(1):228--243, 2008.

\bibitem{BekkertMerklen}
V.~Bekkert and H.~Merklen.
\newblock Indecomposables in derived categories of gentle algebras.
\newblock {\em Algebras and Representation Theory}, 6(3):285--302, 2003.

\bibitem{Bobinski}
G.~Bobiński.
\newblock The almost split triangles for perfect complexes over gentle
  algebras.
\newblock {\em Journal of Pure and Applied Algebra}, 215(4):642--654, 2011.

\bibitem{BocklandtPuncturedSurface}
R.~Bocklandt.
\newblock Noncommutative mirror symmetry for punctured surfaces.
\newblock {\em Transactions of the American Mathematical Society}, 368, 2011.

\bibitem{BurbanDrozd2004}
I.~Burban and Y.~Drozd.
\newblock Derived categories of nodal algebras.
\newblock {\em Journal of Algebra}, 272(1):46--94, 2004.

\bibitem{BurbanDrozdTilting}
I.~Burban and Y.~Drozd.
\newblock Tilting on non-commutative rational projective curves.
\newblock {\em Mathematische Annalen}, 351(3):665--709, 2011.

\bibitem{BurbanDrozd2018}
I.~Burban and Y.~Drozd.
\newblock Non-commutative nodal curves and derived tame algebras.
\newblock {\em arXiv:1805.05174 [math.AG]}, 2018.

\bibitem{Chen}
X.~Chen.
\newblock A note on standard equivalences.
\newblock {\em Bulletin of the London Mathematical Society}, 48:797--801, 2016.

\bibitem{Chillingworth}
D.~Chillingworth.
\newblock Winding numbers on surfaces, {I.}
\newblock {\em Mathematische Annalen}, 196:218--249, 1972.

\bibitem{Dade}
E.~Dade.
\newblock Blocks with cyclic defect groups.
\newblock {\em Annals of Mathematics}, 84(1):20--48, 1966.

\bibitem{Disarlo}
V.~Disarlo.
\newblock Combinatorial rigidity of arc complexes.
\newblock {\em arXiv:1505.08080 [math.RT]}, 2015.

\bibitem{Donovan}
P.~Donovan.
\newblock Dihedral defect groups.
\newblock {\em Journal of Algebra}, 56(1):184--206, 1979.

\bibitem{Epstein1966}
D.~Epstein.
\newblock Curves on {$2$}-manifolds and isotopies.
\newblock {\em Acta Mathematica}, 115(1):83--107, 1966.

\bibitem{FarbMargalit}
B.~Farb and D.~Margalit.
\newblock {\em A primer on mapping class groups}.
\newblock Princeton University Press, 2012.

\bibitem{FreedmanHassScott}
M.~Freedman, J.~Hass, and P.~Scott.
\newblock Closed geodesics on surfaces.
\newblock {\em Bulletin of the London Mathematical Society}, 14(5):385--391,
  1982.

\bibitem{AsensioSaorin}
F.~Guil-Asensio and M.~Saor{\'i}n.
\newblock The group of outer automorphisms and the {P}icard group of an
  algebra.
\newblock {\em Algebras and Representation Theory}, 2(4):313--330, 1999.

\bibitem{HaidenKatzarkovKontsevich}
F.~Haiden, L.~Katzarkov, and M.~Kontsevich.
\newblock Flat surfaces and stability structures.
\newblock {\em Publications Mathematiques de IHES}, 126(1):247--318, 2017.

\bibitem{LekiliPolishchuk2017}
Y.~Lekili and A.~Polishchuk.
\newblock Auslander orders over nodal stacky curves and partially wrapped
  {F}ukaya categories.
\newblock {\em arXiv:1705.06023 [math.SG]}, 2017.

\bibitem{LekiliPolishchukGentle}
Y.~Lekili and A.~Polishchuk.
\newblock Derived equivalences of gentle algebras via {F}ukaya categories.
\newblock {\em arXiv:1801.06370 [math.SG]}, 2018.

\bibitem{Neumann-Coto}
M.~{Neumann-Coto}.
\newblock A characterization of shortest geodesics on surfaces.
\newblock {\em Algebra and Geometric Topology}, 1(1):349--368, 2001.

\bibitem{OpperThesis}
S.~Opper.
\newblock A surface model for gentle algebras.
\newblock {\em PhD thesis}, Universit{\"a}t K{\"o}ln, January 2019.

\bibitem{OpperPlamondonSchroll}
S.~Opper, {P-G}. Plamondon, and S.~Schroll.
\newblock A geometric model for the derived category of gentle algebras.
\newblock {\em arXiv:1801.09659 [math.RT]}, 2018.

\bibitem{ReitenVanDenBergh}
I.~Reiten and {M.~Van den Bergh}.
\newblock Noetherian hereditary categories satisfying {S}erre duality.
\newblock {\em Journal of the American Mathematical Society}, 15(2):295--366,
  2002.

\bibitem{RickardMoritaTheory}
J.~Rickard.
\newblock Morita theory for derived categories.
\newblock {\em Journal of the London Mathematical Society}, s2-39(3):436--456,
  1989.

\bibitem{RoggenkampTaylor}
K.~Roggenkamp and M.~Taylor.
\newblock {\em Group rings and class groups}.
\newblock DMV Seminar. Birkh{\"a}user Verlag, 1992.

\bibitem{SchroerZimmermann}
J.~Schroer and A.~Zimmermann.
\newblock Stable endomorphism algebras of modules over special biserial
  algebras.
\newblock {\em Mathematische Zeitschrift}, 244(3):515--530, 2003.

\bibitem{SchrollTrivialExtension}
S.~Schroll.
\newblock Trivial extensions of gentle algebras and {B}rauer graph algebras.
\newblock {\em Journal of Algebra}, 444:183--200, 2015.

\bibitem{SeidelThomas}
P.~Seidel and R.~Thomas.
\newblock Braid group actions on derived categories of coherent sheaves.
\newblock {\em Duke Mathematical Journal}, 108(1):37--108, 2001.

\bibitem{Thurston}
D.~Thurston.
\newblock Geometric intersection of curves on surfaces.
\newblock {\em unpublished}, 2008.

\end{thebibliography}

\end{document}